\documentclass[final]{siamltex}

\newcommand{\lsp}{\vspace{3mm}}
\newcounter{remark}[section]

\usepackage{amsmath}
\usepackage{amssymb}
\usepackage{bm}
\usepackage{mathrsfs}
\usepackage[dvips]{graphicx}

\usepackage{color}


\DeclareMathAlphabet{\mathsfsl}{OT1}{cmss}{m}{sl}

\newcommand{\term}{\emph}

\newcommand{\cnst}[1]{\mathrm{#1}}

    {\makebox{\phantom{}} \\ %
        \textsc{Input:}\begin{itemize}}%
    {\end{itemize}}

    {\textsc{Output:}\begin{itemize}}%
    {\end{itemize}}

    {\textsc{Procedure:}\begin{enumerate}}%
    {\end{enumerate}}

\renewcommand{\phi}{\varphi}

\newcommand{\eps}{\varepsilon}

\newcommand{\econst}{\mathrm{e}}
\newcommand{\iunit}{\mathrm{i}}

\newcommand{\Id}{\mathbf{I}}

\newcommand{\Cspace}[1]{\mathbb{C}^{#1}}

\newcommand{\abs}[1]{\left\vert {#1} \right\vert}
\newcommand{\abssq}[1]{{\abs{#1}}^2}

\newcommand{\diff}[1]{\mathrm{d}{#1}}
\newcommand{\idiff}[1]{\, \diff{#1}}

\newcommand{\psdge}{\succcurlyeq}
\newcommand{\psdle}{\preccurlyeq}

\newcommand{\Probe}[1]{\mathbb{P}\left({#1}\right)}
\newcommand{\Prob}[1]{\mathbb{P}\left\{ {#1} \right\}}
\newcommand{\Expect}{\operatorname{\mathbb{E}}}

\newcommand{\vct}[1]{\bm{#1}}
\newcommand{\mtx}[1]{\bm{#1}}

\newcommand{\adj}{*}
\newcommand{\psinv}{\dagger}

\newcommand{\lspan}[1]{\operatorname{span}{#1}}

\newcommand{\range}{\operatorname{range}}

\newcommand{\trace}{\operatorname{trace}}

\newcommand{\ip}[2]{\left\langle {#1},\ {#2} \right\rangle}

\newcommand{\norm}[1]{\left\Vert {#1} \right\Vert}
\newcommand{\normsq}[1]{\norm{#1}^2}

\newcommand{\smnorm}[2]{{\bigl\Vert {#2} \bigr\Vert}_{#1}}

\newcommand{\fnorm}[1]{\norm{#1}_{\mathrm{F}}}
\newcommand{\fnormsq}[1]{\fnorm{#1}^2}

\newcommand{\triplenorm}[1]{\left\vert\!\left\vert\!\left\vert {#1} \right\vert\!\right\vert\!\right\vert}

\newcommand{\smtriplenorm}[1]{\big\vert\!\big\vert\!\big\vert {#1} \big\vert\!\big\vert\!\big\vert}

\newcommand{\bigO}{\mathrm{O}}

\newcommand{\anum}[1]{{\footnotesize{#1}\quad}}

\input{algorithm_labels.aux}

\newcounter{algorithm}[section]

\newtheorem{remark}{Remark}[section]

\newenvironment{widequote}{\begin{list}{}
      {\setlength{\rightmargin}{0mm}\setlength{\leftmargin}{5mm}}
      \item[]}{\end{list}}

\usepackage{url}
\usepackage{citesort}
\usepackage{rotating}

\title{Finding structure with randomness:\\ Probabilistic algorithms for
constructing approximate matrix decompositions}

\author{N.~Halko\footnotemark[2]
\and P.~G.~Martinsson\footnotemark[2]
\and J.~A.~Tropp\footnotemark[3]}

\begin{document}

\maketitle

\renewcommand{\thefootnote}{\fnsymbol{footnote}}

\footnotetext[2]{Department of Applied Mathematics, University of Colorado at Boulder, Boulder, CO 80309-0526.
Supported by NSF awards \#0748488 and \#0610097.}

\footnotetext[3]{Computing \& Mathematical Sciences, California Institute of Technology,
        MC 305-16, Pasadena, CA 91125-5000.  Supported by ONR award \#N000140810883.}

\renewcommand{\thefootnote}{\arabic{footnote}}

\begin{abstract}
Low-rank matrix approximations, such as the truncated singular value
decomposition and the rank-revealing QR decomposition, play a central
role in data analysis and scientific computing. This work surveys and
extends
recent research which demonstrates that \emph{randomization} offers a
powerful tool for performing low-rank matrix approximation.
These
techniques exploit modern computational architectures
more fully than classical methods and open the possibility of
dealing with truly massive data sets.

This paper presents a modular framework for constructing randomized
algorithms that compute partial matrix decompositions. These methods
use random sampling
to identify a subspace that captures most of the action
of a matrix. The input matrix is then compressed---either explicitly
or implicitly---to this subspace, and the reduced matrix is
manipulated deterministically to obtain the desired low-rank
factorization. In many cases, this approach beats its classical
competitors in terms of accuracy, speed, and robustness.
These claims
are supported by extensive numerical experiments and a detailed error
analysis.

The specific benefits of randomized techniques depend on the
computational environment. Consider the model problem of finding
the $k$ dominant components of the singular value decomposition
of an $m \times n$ matrix.
(i) For a dense input matrix, randomized algorithms require $\bigO(mn
\log(k))$ floating-point operations (flops) in contrast with $
\bigO(mnk)$ for classical algorithms.
%
(ii) For a sparse input matrix, the flop count matches classical
Krylov subspace methods, but the randomized approach is more robust
and can easily be reorganized to exploit multi-processor architectures.
(iii) For a matrix that is too large to fit in fast memory, the randomized
techniques require only a constant number of passes over the data,
as opposed to $\bigO(k)$ passes for classical algorithms.
In fact, it is sometimes possible to perform matrix approximation with a
\emph{single pass} over the data.

\end{abstract}


\begin{AMS}[MSC2010]
Primary: 65F30.  Secondary: 68W20, 60B20.  
\end{AMS}

\pagestyle{myheadings}
\thispagestyle{plain}
\markboth{HALKO, MARTINSSON, AND~TROPP}{PROBABILISTIC ALGORITHMS FOR MATRIX APPROXIMATION}

\lsp


\begin{center}
{\bf Part I: Introduction}
\end{center}

\section{Overview}
\label{sec:intro}

On a well-known list of the ``Top 10 Algorithms'' that have influenced
the practice of science and engineering during the 20th century~\cite{DS00:Top-10},
we find an entry that is not really an algorithm: the \emph{idea} of using matrix
factorizations to accomplish basic tasks in numerical linear algebra.  In
the accompanying article~\cite{Ste00:Decompositional-Approach}, Stewart
explains that

\lsp
\begin{quote}
The underlying principle of the decompositional approach to matrix computation
is that it is not the business of the matrix algorithmicists to solve particular
problems but to construct computational platforms from which a variety of
problems can be solved.
\end{quote}
\lsp

\noindent
Stewart goes on to argue that this point of view has had many fruitful consequences,
including the development of robust software for performing these factorizations
in a highly accurate and provably correct manner.

%
The decompositional approach to matrix computation remains
fundamental, but developments in computer hardware and the
emergence of new applications in the information sciences
have rendered the classical algorithms for this task inadequate
in many situations:

\lsp

\begin{itemize}
\item   A salient feature of modern applications, especially in data mining,
is that the matrices are stupendously big.  Classical algorithms are not
always well adapted to solving the type of large-scale problems that now arise.

\item   In the information sciences, it is common that data are missing or
inaccurate.  Classical algorithms are designed to produce highly accurate
matrix decompositions, but it seems profligate to spend extra computational
resources when the imprecision of the data inherently limits the resolution
of the output.

\item   Data transfer now plays a major role in the computational cost of
numerical algorithms.  Techniques that require few passes
over the data may be substantially faster in practice, even if
they require as many---or more---floating-point operations.


\item   As the structure of computer processors continues to evolve, it
becomes increasingly important for numerical algorithms to adapt to a range of
novel architectures, such as graphics processing units.
\end{itemize}

\lsp


The purpose of this paper is make the case that \emph{randomized} algorithms provide a
powerful tool for constructing approximate matrix factorizations.
These techniques are simple and effective, sometimes impressively so.
Compared with standard deterministic algorithms, the
randomized methods are often faster and---perhaps
surprisingly---more robust.  Furthermore, they can produce
factorizations that are accurate to any specified tolerance above
machine precision, which allows the user to trade accuracy for speed
if desired.
We present numerical evidence that
these algorithms succeed for real computational problems.

In short, our goal is to demonstrate how randomized methods
interact with classical techniques to yield effective, modern
algorithms supported by detailed theoretical guarantees.
We have made a special effort to help practitioners identify
situations where randomized techniques may outperform established
methods.

Throughout this article, we provide detailed citations to previous work
on randomized techniques for computing low-rank approximations.
The primary sources that inform our presentation include~\cite{BMD09:Improved-Approximation,drineas_kannan_mahoney,kannan_vempala,%
random1,papadimitriou,tygert_szlam,2008_rokhlin_leastsquares,Sar06:Improved-Approximation,random2}.


%



\lsp

\begin{remark}
\label{remark:queasy} \rm Our experience suggests that many
practitioners of scientific computing view randomized algorithms as a
desperate and final resort.  Let us address this concern immediately.
Classical Monte Carlo methods are highly sensitive to the random
number generator and typically produce output with low and uncertain
accuracy.  In contrast, the algorithms discussed herein are relatively
insensitive
to the quality of randomness and produce highly accurate results.
The probability of failure is a user-specified parameter that can
be rendered negligible (say, less than $10^{-15}$) with a nominal
impact on the computational resources required.
\end{remark}

\subsection{Approximation by low-rank matrices}

The roster of standard matrix decompositions includes the pivoted QR
factorization, the eigenvalue decomposition, and the singular value
decomposition (SVD), all of which expose the (numerical) range of a
matrix. Truncated versions of these factorizations are often used to
express a \emph{low-rank approximation} of a given matrix:
\begin{equation}
\label{eq:lowrank}
\begin{array}{ccccccccccc}
\mtx{A} &\approx& \mtx{B} & \mtx{C},\\
m\times n && m \times k & k\times n.
\end{array}
\end{equation}
The inner dimension $k$ is sometimes called the \emph{numerical rank} of the matrix.
When the numerical rank is much smaller than either dimension $m$ or $n$,
a factorization such as \eqref{eq:lowrank} allows the matrix to be stored
inexpensively and to be multiplied rapidly with vectors or other matrices.
The factorizations can also be used for data interpretation or to
solve computational problems, such as least squares.

Matrices with low numerical rank appear in a wide variety of scientific
applications.  We list only a few:


\lsp
\begin{itemize}
\item   A basic method in statistics and data mining is to compute the
directions of maximal variance in vector-valued data by performing
\emph{principal component analysis} (PCA) on the data matrix.
PCA is nothing other than a low-rank matrix
approximation~\cite[\S14.5]{HTF08:Elements-Statistical}.

\item   Another standard technique in data analysis is to perform
low-dimensional embedding of data under the assumption that there
are fewer degrees of freedom than the ambient dimension would suggest.
In many cases, the method reduces to computing a partial SVD of a matrix
derived from the data.  See~\cite[\S\S14.8--14.9]{HTF08:Elements-Statistical}
or \cite{coifman_PNAS_diffusionmaps}.

\item The problem of estimating parameters from measured data via
least-squares fitting often leads to very large systems of linear
equations that are close to linearly dependent. Effective techniques
for factoring the coefficient matrix lead to efficient techniques
for solving the least-squares problem,
\cite{2008_rokhlin_leastsquares}.

\item   Many fast numerical algorithms for solving PDEs and for rapidly evaluating
potential fields such as the fast multipole method~\cite{rokhlin1997}
and $\mathcal{H}$-matrices~\cite{hackbusch2003}, rely on low-rank approximations of
continuum operators.

\item Models of multiscale physical phenomena often involve PDEs with rapidly
oscillating coefficients. Techniques for  \emph{model reduction}
or \emph{coarse graining} in such environments are often based 
on the observation that the linear transform that maps the input data to the requested output
data 
can be approximated by an operator of low rank~\cite{engquist_wavelethomogenization}.
\end{itemize}

\lsp

\subsection{Matrix approximation framework}
\label{sec:framework}
The task of computing a low-rank approximation to a given matrix
can be split naturally into two computational stages. The first is
to construct a low-dimensional subspace that captures the action
of the matrix. The second is to restrict the matrix to the subspace
and then compute a standard factorization (QR, SVD, etc.) of the
reduced matrix.  To be slightly more formal, we subdivide the
computation as follows.

\lsp

\begin{widequote}
\noindent
\textbf{Stage A:} Compute an approximate basis for the range of the input matrix $\mtx{A}$. In other words,
we require a matrix $\mtx{Q}$ for which
\begin{equation} \label{eqn:Q-form}
\text{$\mtx{Q}$ has orthonormal columns and $\mtx{A} \approx \mtx{Q}\mtx{Q}^{\adj}\mtx{A}$.}
\end{equation}
We would like the basis matrix $\mtx{Q}$ to contain as few columns as possible, but it is
even more important to have an accurate approximation of the input matrix.


\vspace{2mm}

\textbf{Stage B:}
Given a matrix $\mtx{Q}$ that satisfies~\eqref{eqn:Q-form},
we use $\mtx{Q}$ to help compute a standard factorization (QR, SVD, etc.) of $\mtx{A}$.
\end{widequote}

\lsp

The task in Stage A can be executed very efficiently with random sampling methods,
and these methods are the primary subject of this work.  In the next subsection, we offer an
overview of these ideas.  The body of the paper provides details of the algorithms
(\S\ref{sec:algorithm}) and a theoretical analysis of their performance
(\S\S\ref{sec:lin-alg-prelim}--\ref{sec:SRFTs}).

Stage B can be completed with well-established deterministic methods.
Section~\ref{sec:converting} contains an introduction to these techniques,
and \S\ref{sec:otherfactorizations} shows how we apply them to produce
low-rank factorizations.

At this point in the development, it may not be clear why the output from
Stage~A facilitates our job in Stage~B.
Let us illustrate by describing how to obtain an approximate SVD
of the input matrix  $\mtx{A}$ given a matrix $\mtx{Q}$ that satisfies~\eqref{eqn:Q-form}.
More precisely, we wish to compute matrices $\mtx{U}$ and $\mtx{V}$ with orthonormal columns
and a nonnegative, diagonal matrix $\mtx{\Sigma}$ such that
$\mtx{A} \approx \mtx{U\Sigma V}^\adj$.
This goal is achieved after three simple steps:
\lsp
\begin{remunerate}
\item   Form $\mtx{B} = \mtx{Q}^{\adj}\mtx{A}$, which yields
the low-rank factorization $\mtx{A} \approx \mtx{Q}\mtx{B}$.
\item   Compute an SVD of the small matrix: $\mtx{B} = \widetilde{\mtx{U}}\mtx{\Sigma}\mtx{V}^{\adj}$.
\item   Set $\mtx{U} = \mtx{Q}\widetilde{\mtx{U}}$.
\end{remunerate}
\lsp

When $\mtx{Q}$ has few columns, this procedure is efficient
because we can easily construct the reduced matrix $\mtx{B}$ and rapidly compute its SVD.
In practice, we can often avoid forming $\mtx{B}$
explicitly by means of subtler techniques.  In some cases,
it is not even necessary to revisit the input matrix $\mtx{A}$
during Stage B.  This observation allows us to develop
\term{single-pass algorithms}, which look at each
entry of $\mtx{A}$ only once.

Similar manipulations readily yield other standard factorizations,
such as the pivoted QR factorization, the eigenvalue decomposition, etc.


\subsection{Randomized algorithms}
\label{sec:sketchofalgorithm}

This paper describes a class of randomized algorithms for completing Stage A
of the matrix approximation framework set forth in \S\ref{sec:framework}.
We begin with some details about the approximation problem these algorithms
target (\S\ref{sec:modelproblem}).
Afterward, we motivate the random sampling technique
with a heuristic explanation (\S\ref{sec:intuition})
that leads to a prototype algorithm (\S\ref{sec:proto-algorithm}).


\subsubsection{Problem formulations}
\label{sec:modelproblem}

The basic challenge in producing low-rank matrix approximations
is a primitive question that we call the \term{fixed-precision
approximation problem}.  Suppose we are given a matrix
$\mtx{A}$ and a positive error tolerance $\eps$.
We seek a matrix $\mtx{Q}$ with $k = k(\eps)$ orthonormal columns
such that
\begin{equation} \label{eq:fixed_precision}
\norm{ \mtx{A} - \mtx{QQ}^\adj \mtx{A} } \leq \eps,
\end{equation}
where $\norm{ \cdot }$ denotes the $\ell_2$ operator norm.
The range of $\mtx{Q}$ is a $k$-dimensional subspace that captures most of the
action of $\mtx{A}$, and we would like $k$ to be as small as possible.

The singular value decomposition furnishes an optimal answer to
the fixed-precision problem~\cite{Mir60:Symmetric-Gauge}.
Let $\sigma_{j}$ denote the $j$th largest singular value of $\mtx{A}$.
For each $j \geq 0$,
\begin{equation}
\label{eqn:mirsky}
\min_{{\rm rank}(\mtx{X}) \leq j}\norm{ \mtx{A} - \mtx{X} } = \sigma_{j+1}.
\end{equation}
One way to construct a minimizer is to choose $\mtx{X} = \mtx{QQ}^\adj\mtx{A}$,
where the columns of $\mtx{Q}$ are $k$ dominant left singular vectors of $\mtx{A}$.
Consequently, the minimal rank $k$ where \eqref{eq:fixed_precision} holds
equals the number of singular values of $\mtx{A}$ that exceed the tolerance $\eps$.

To simplify the development of algorithms, it is convenient to assume
that the desired rank $k$ is specified in advance.
We call the resulting problem the \term{fixed-rank
approximation problem}.  Given a matrix $\mtx{A}$, a target rank $k$,
and an oversampling parameter $p$, we seek to construct a matrix
$\mtx{Q}$ with $k + p$ orthonormal columns such that
\begin{equation}
\label{eq:fixed_rank}
\norm{ \mtx{A} - \mtx{Q}\mtx{Q}^{\adj}\mtx{A} } \approx
\min_{{\rm rank}(\mtx{X}) \leq k}\norm{ \mtx{A} - \mtx{X} }.
\end{equation}
Although there exists a minimizer $\mtx{Q}$ that solves the fixed rank
problem for $p=0$,  the opportunity to use a small number of
additional columns provides a flexibility that is crucial for the
effectiveness of the computational methods we discuss.

We will demonstrate that algorithms for the fixed-rank
problem can be adapted to solve the fixed-precision problem.
The connection is based on the observation that we can build the
basis matrix $\mtx{Q}$ incrementally and, at any point in the
computation, we can inexpensively estimate the residual error
$\norm{\mtx{A}-\mtx{QQ}^\adj\mtx{A}}$.
Refer to~\S\ref{sec:algorithm1} for the details
of this reduction.

\subsubsection{Intuition}
\label{sec:intuition}

To understand how randomness helps us solve the fixed-rank
problem, it is helpful to consider some motivating examples.

First, suppose that we seek a basis for the range of a matrix $\mtx{A}$ with
\textit{exact} rank $k$. Draw a random vector $\vct{\omega}$,
and form the product $\vct{y} = \mtx{A} \vct{\omega}$.
For now, the precise distribution of the random vector is
unimportant; just think of $\vct{y}$ as a random sample from the
range of $\mtx{A}$. Let us repeat this sampling process $k$ times:
\begin{equation}
\label{eq:samples}
\vct{y}^{(i)} = \mtx{A} \vct{\omega}^{(i)},
\quad
i = 1, 2, \dots, k.
\end{equation}
Owing to the randomness, the set $\{ \vct{\omega}^{(i)} : i = 1, 2, \dots, k \}$ of
random vectors is likely to be in general linear position.
In particular, the random vectors form a linearly independent set and
no linear combination falls in the null space of $\mtx{A}$.  As a
result, the set $\{ \vct{y}^{(i)} : i = 1, 2, \dots, k \}$ of sample vectors is
also linearly independent, so it spans the range of $\mtx{A}$.
Therefore, to produce an orthonormal basis for the range of
$\mtx{A}$, we just need to orthonormalize the sample vectors.


Now, imagine that $\mtx{A} = \mtx{B} + \mtx{E}$ where $\mtx{B}$ is a rank-$k$
matrix containing the information we seek and $\mtx{E}$ is a small
perturbation.  Our priority is to obtain a basis that covers as much of the
range of $\mtx{B}$ as possible, rather than to minimize the number of basis vectors.
Therefore, we fix a small number $p$, 
and we generate $k+p$ samples
\begin{equation}
\vct{y}^{(i)} = \mtx{A} \vct{\omega}^{(i)}
	= \mtx{B} \vct{\omega}^{(i)} + \mtx{E} \vct{\omega}^{(i)},
\quad
i = 1, 2, \dots, k + p.
\end{equation}
The perturbation $\mtx{E}$ shifts the direction of each sample
vector outside the range of $\mtx{B}$, which can prevent the span of
$\{\vct{y}^{(i)} : i = 1, 2, \dots, k \}$ from covering
the entire range of $\mtx{B}$.  In contrast, the enriched set
$\{\vct{y}^{(i)} :  i = 1, 2, \dots, k + p \}$ of samples has a much better chance
of spanning the required subspace.

Just how many extra samples do we need?  Remarkably,
for certain types of random sampling schemes,
the failure probability decreases superexponentially with
the oversampling parameter $p$; see~\eqref{eq:intro_err_prob}.
As a practical matter, setting $p = 5$ or $p = 10$
often gives superb results.
This observation is one of the principal facts supporting the randomized
approach to numerical linear algebra.



\subsubsection{A prototype algorithm}
\label{sec:proto-algorithm}

The intuitive approach of \S\ref{sec:intuition} can be applied
to general matrices. Omitting computational details for now, we
formalize the procedure in the figure labeled Proto-Algorithm.

\begin{figure}
\begin{center}
\framebox{\begin{minipage}{.9\textwidth}
\begin{center}
\textsc{Proto-Algorithm: Solving the Fixed-Rank Problem}
\end{center}

\lsp

\textit{Given an $m\times n$ matrix $\mtx{A}$, a target rank $k$, and an oversampling parameter $p$,
this procedure computes an $m\times (k+p)$ matrix $\mtx{Q}$ whose columns are orthonormal and whose range
approximates the range of $\mtx{A}$.}

\lsp
\begin{tabbing}
\hspace{5mm} \= \hspace{5mm} \= \hspace{5mm} \= \hspace{5mm} \= \kill

\anum{1} \>Draw a random $n \times (k + p)$ test matrix $\mtx{\Omega}$.\\

\anum{2} \>Form the matrix product $\mtx{Y} = \mtx{A\Omega}$.\\

\anum{3} \>Construct a matrix $\mtx{Q}$ whose columns form an orthonormal basis for \\
         \>the range of $\mtx{Y}$.
\end{tabbing}
\end{minipage}}
\end{center}
\end{figure}

This simple algorithm is by no means new. It is essentially the
first step of a subspace iteration with a random initial
subspace~\cite[\S7.3.2]{golub}. The novelty comes from the
additional observation that the initial subspace should have a
slightly higher dimension than the invariant subspace we are trying
to approximate.  With this revision, it is often the case that
\emph{no further iteration is required} to obtain a high-quality
solution to~\eqref{eq:fixed_rank}.  We believe this idea can be
traced to~\cite{Sar06:Improved-Approximation,random1,papadimitriou}.

In order to invoke the proto-algorithm with confidence, we must
address several practical and theoretical issues:

\lsp

\begin{itemize}
\item   What random matrix $\mtx{\Omega}$ should we use?  How much oversampling do we need?
\item   The matrix $\mtx{Y}$ is likely to be ill-conditioned. How do we orthonormalize its
        columns to form the matrix $\mtx{Q}$?
\item   What are the computational costs?
\item   How can we solve the fixed-precision problem~\eqref{eq:fixed_precision}
        when the numerical rank of the matrix is not known in advance?
\item   How can we use the basis $\mtx{Q}$ to compute other matrix factorizations?
\item   Does the randomized method work for problems of practical interest?
        How does its speed/accuracy/robustness compare with standard techniques?
\item   What error bounds can we expect?  With what probability?
\end{itemize}

\lsp

\noindent The next few sections provide a summary of the answers to
these questions. We describe several problem regimes where the
proto-algorithm can be implemented efficiently, and we present a
theorem that describes the performance of the most important
instantiation.  Finally, we elaborate on how these ideas can
be applied to approximate the truncated SVD of a large data matrix.
The rest of the paper contains a more exhaustive
treatment---including pseudocode, numerical experiments, and a
detailed theory.

\subsection{A comparison between randomized and traditional techniques}

To select an appropriate computational method for finding a low-rank
approximation to a matrix, the practitioner must take into account
the properties of the matrix.  Is it dense or sparse?  Does it fit
in fast memory or is it stored out of core?  Does the singular spectrum
decay quickly or slowly?
The behavior of a numerical linear algebra algorithm may depend on
all these factors~\cite{Bjo96:Numerical-Methods,golub,trefethen_bau}.
To facilitate a comparison between classical
and randomized techniques, we summarize their relative performance
in each of three representative environments.
Section~\ref{sec:costs} contains a more in-depth treatment.

We focus on the task of computing an approximate SVD
of an $m \times n$ matrix $\mtx{A}$ with numerical rank $k$.
For randomized schemes, Stage A generally dominates the cost of Stage B
in our matrix approximation framework (\S\ref{sec:framework}).
Within Stage A, the computational bottleneck is usually the
matrix--matrix product $\mtx{A\Omega}$ in Step 2 of the
proto-algorithm (\S\ref{sec:proto-algorithm}).
The power of randomized algorithms stems from the fact that
we can reorganize this matrix multiplication for maximum
efficiency in a variety of computational architectures.

\subsubsection{A general dense matrix that fits in fast memory}
\label{sec:intro_fits in RAM}

A standard deterministic technique for computing an approximate SVD is to
perform a rank-revealing QR factorization of the matrix,
and then to manipulate the factors to obtain the final decomposition.
The cost of this approach is typically $\bigO(kmn)$ floating-point
operations, or \term{flops}, although these methods require slightly
longer running times in rare cases~\cite{gu_rrqr}.

In contrast, randomized schemes can produce an approximate SVD using
only $\bigO(mn\log(k) + (m+n)k^2)$ flops.  The gain in asymptotic
complexity is achieved by using a random matrix $\mtx{\Omega}$ that
has some internal structure, which allows us to evaluate the
product $\mtx{A\Omega}$ rapidly.  For example,
randomizing and subsampling the discrete Fourier transform
works well.  Sections~\ref{sec:ailonchazelle} and~\ref{sec:SRFTs}
contain more information on this approach.

\subsubsection{A matrix for which matrix--vector products can be evaluated rapidly}

When the matrix $\mtx{A}$ is sparse or structured, we may be able to apply
it rapidly to a vector.  In this case, the classical prescription for computing
a partial SVD is to invoke a Krylov subspace method, such as the Lanczos or Arnoldi
algorithm.  It is difficult to summarize the computational cost of these methods because
their performance depends heavily on properties of the input matrix
and on the amount of effort spent to stabilize the algorithm.
(Inherently, the Lanczos and Arnoldi methods are numerically unstable.)
For the same reasons, the error analysis of such schemes is unsatisfactory
in many important environments.


At the risk of being overly simplistic, we claim that the typical
cost of a Krylov method for approximating the $k$ leading singular vectors
of the input matrix is proportional to $k \, T_{\rm mult} + (m+n) k^2$,
where $T_{\rm mult}$ denotes the cost of a matrix--vector multiplication with
the input matrix and the constant of proportionality is small.
We can also apply randomized methods using a Gaussian test matrix
$\mtx{\Omega}$ to complete the factorization at the same cost,
$\bigO(k \, T_{\rm mult} + (m+n)k^2)$ flops.


With a given budget of floating-point operations, Krylov methods
sometimes deliver a more accurate approximation than randomized
algorithms.  Nevertheless, the methods described in this survey have at least
two powerful advantages over Krylov methods.  First, the randomized schemes
are inherently stable, and they come with very strong performance guarantees
that do not depend on 
subtle spectral properties of the input matrix.  Second, the matrix--vector multiplies
required to form $\mtx{A\Omega}$ can be performed \emph{in parallel}.  This
fact allows us to restructure the calculations to take full advantage of the
computational platform, which can lead to dramatic accelerations in practice,
especially for parallel and distributed machines.


A more detailed comparison or randomized schemes and Krylov subspace methods
is given in \S\ref{sec:fastmatvec}.

\subsubsection{A general dense matrix stored in slow memory or streamed}

When the input matrix is too large to fit in core memory, the cost of
transferring the matrix from slow memory typically dominates the cost
of performing the arithmetic.  The standard techniques for low-rank
approximation described in~\S\ref{sec:intro_fits in RAM} require
$\bigO(k)$ passes over the matrix, which can be prohibitively expensive.

In contrast, the proto-algorithm of~\S\ref{sec:proto-algorithm}
requires only one pass over the data to produce the approximate basis
$\mtx{Q}$ for Stage A of the approximation framework.
This straightforward approach, unfortunately, is not accurate enough for
matrices whose singular spectrum decays slowly, but we can address this
problem using very few (say, $2$ to $4$) additional passes over the
data~\cite{tygert_szlam}.
See \S\ref{sec:pca} or~\S\ref{sec:powerscheme} for more discussion.


Typically, Stage B uses one additional pass over the matrix to
construct the approximate SVD.
With slight modifications, however, the two-stage randomized scheme
can be revised so that it only makes a single pass over the
data.  Refer to~\S\ref{sec:onepass} for information.

\subsection{Performance analysis}
\label{sec:prototheorem} A principal goal of this paper is to
provide a detailed analysis of the performance of the
proto-algorithm described in~\S\ref{sec:proto-algorithm}. This
investigation produces precise error bounds, expressed in terms of
the singular values of the input matrix. Furthermore, we determine
how several choices of the random matrix $\mtx{\Omega}$ impact the
behavior of the algorithm.

Let us offer a taste of this theory.  The following theorem
describes the average-case behavior of the proto-algorithm
with a Gaussian test matrix, assuming we perform the computation
in exact arithmetic.  This result is a simplified version of
Theorem~\ref{thm:avg-spec-error-gauss}.

\lsp

\begin{theorem} 
Suppose that $\mtx{A}$ is a real $m \times n$ matrix.  Select
a target rank $k \geq 2$ and an oversampling parameter $p \geq 2$,
where $k + p \leq \min\{m,n\}$.
Execute the proto-algorithm with
a standard Gaussian test matrix to obtain an
$m \times (k + p)$ matrix $\mtx{Q}$ with orthonormal columns.  Then
\begin{equation}
\label{eq:intro_err_bd}
\Expect \norm{ \mtx{A} - \mtx{QQ}^\adj \mtx{A} }
    \leq \left[ 1 + \frac{4 \sqrt{k+p}}{p-1} \cdot\sqrt{\min\{m,n\}} \right] \sigma_{k+1},
\end{equation}
where $\Expect$ denotes expectation with respect to the
random test matrix and $\sigma_{k+1}$ is the $(k+1)$th
singular value of $\mtx{A}$.
\end{theorem}

\lsp

We recall that the term $\sigma_{k+1}$ appearing in \eqref{eq:intro_err_bd}
is the smallest possible error~\eqref{eqn:mirsky} achievable with any basis
matrix $\mtx{Q}$. The theorem asserts that, on average,
the algorithm produces a basis whose error lies within a small
polynomial factor of the theoretical minimum.
Moreover, the error bound~\eqref{eq:intro_err_bd} in the randomized
algorithm is slightly sharper than comparable bounds for
deterministic techniques based on rank-revealing QR algorithms~\cite{gu_rrqr}.

The reader might be worried about whether the expectation provides a
useful account of the approximation error.  Fear not: the actual outcome
of the algorithm is {\em almost always} very close to
the typical outcome because of measure
concentration effects.  As we discuss in~\S\ref{sec:prob-failure},
the probability that the error satisfies
\begin{equation} \label{eq:intro_err_prob}
\norm{ \mtx{A} - \mtx{QQ}^\adj \mtx{A} }
    \leq \left[ 1 + 11 \sqrt{k+p} \cdot\sqrt{\min\{m,n\}} \right] \sigma_{k+1}
\end{equation}
is at least $1 - 6 \cdot p^{-p}$ under very mild assumptions on $p$.
This fact justifies the use of an oversampling term as small as $p = 5$.
This simplified estimate is very similar to the major results in~\cite{random1}.

The theory developed in this paper provides much more detailed information about the
performance of the proto-algorithm.

\lsp

\begin{itemize}
\item
When the singular values of $\mtx{A}$ decay slightly,
the error $\norm{ \mtx{A} - \mtx{QQ}^\adj \mtx{A} }$
does not depend on the dimensions of the matrix
(\S\S\ref{sec:gauss-avg-case}--\ref{sec:prob-failure}).

\item   We can reduce the size of the bracket in the error bound~\eqref{eq:intro_err_bd}
by combining the proto-algorithm with a power iteration (\S\ref{sec:avg-power-method}).
For an example, see \S\ref{sec:pca} below.

\item   For the structured random matrices we mentioned in~\S\ref{sec:intro_fits in RAM},
related error bounds are in force (\S\ref{sec:SRFTs}).

\item   We can obtain inexpensive {\em a posteriori} error estimates to verify the quality
of the approximation~(\S\ref{sec:aposteriori}).

\end{itemize}

\subsection{Example: Randomized SVD}
\label{sec:pca}

We conclude this introduction with a short discussion of how these ideas allow
us to perform an approximate SVD of a large data matrix,
which is a compelling application of randomized matrix approximation~\cite{tygert_szlam}.

The two-stage randomized method offers a natural approach to SVD computations.
Unfortunately, the simplest version of this scheme is inadequate in many
applications because the singular spectrum of the input matrix may decay
slowly.  To address this difficulty, we incorporate $q$ steps of a power
iteration, where $q = 1$ or $q = 2$ usually suffices in practice.  The
complete scheme appears in the box labeled Prototype for Randomized SVD.
For most applications, it is important to incorporate additional refinements,
as we discuss in \S\S\ref{sec:algorithm}--\ref{sec:otherfactorizations}.

\begin{figure}
\begin{center}
\framebox{\begin{minipage}{.9\textwidth}
\begin{center}
\textsc{Prototype for Randomized SVD}
\end{center}

\lsp

\textit{Given an $m\times n$ matrix $\mtx{A}$, a target number $k$ of singular vectors,
and an exponent $q$ (say $q=1$ or $q=2$), this procedure computes an approximate
rank-$2k$ factorization $\mtx{U\Sigma V}^\adj$, where $\mtx{U}$ and $\mtx{V}$
are orthonormal, and $\mtx{\Sigma}$ is nonnegative and diagonal.}

\lsp

{\bf Stage A:}

\begin{tabbing}
\hspace{5mm} \= \hspace{5mm} \= \hspace{5mm} \= \hspace{5mm} \= \kill
\anum{1} \>Generate an $n \times 2k$ Gaussian test matrix $\mtx{\Omega}$.\\

\anum{2} \>Form $\mtx{Y} = (\mtx{AA}^\adj)^q \mtx{A\Omega}$ by multiplying alternately with $\mtx{A}$ and $\mtx{A}^\adj$.\\

\anum{3} \>Construct a matrix $\mtx{Q}$ whose columns form an orthonormal basis for \\
         \>the range of $\mtx{Y}$.
\end{tabbing}

\lsp

{\bf Stage B:}

\begin{tabbing}
\hspace{5mm} \= \hspace{5mm} \= \hspace{5mm} \= \hspace{5mm} \= \kill
\anum{4}    \>Form $\mtx{B} = \mtx{Q}^{\adj}\mtx{A}$.\\

\anum{5}    \>Compute an SVD of the small matrix: $\mtx{B} = \widetilde{\mtx{U}}\mtx{\Sigma}\mtx{V}^{\adj}$.\\

\anum{6}    \>Set $\mtx{U} = \mtx{Q}\widetilde{\mtx{U}}$.
\end{tabbing}

\lsp

{\bf Note:} The computation of $\mtx{Y}$ in Step 2 is vulnerable to round-off errors.
When high accuracy is required, we must incorporate an orthonormalization
step between each application of $\mtx{A}$
and $\mtx{A}^{\adj}$; see Algorithm \ref{alg:subspaceiteration}.
\end{minipage}}
\end{center}
\end{figure}

The Randomized SVD procedure requires 
only $2(q+1)$ passes over the matrix, so it is
efficient even for matrices stored out-of-core.
The flop count satisfies
$$
T_{\rm rand SVD} = (2q+2)\,k \, T_{\rm mult} + \bigO(k^2 (m + n)),
$$
where $T_{\rm mult}$ is the flop count of a matrix--vector multiply
with $\mtx{A}$ or $\mtx{A}^\adj$.
We have the following theorem on the performance of this method
in exact arithmetic, which is a consequence of
Corollary~\ref{cor:power-method-spec-gauss}.


\lsp

\begin{theorem}
Suppose that $\mtx{A}$ is a real $m \times n$ matrix.  Select
an exponent $q$ and a target number $k$ of singular vectors,
where $2 \leq k \leq 0.5 \min\{m,n\}$.
Execute the Randomized SVD algorithm to obtain a rank-$2k$
factorization $\mtx{U\Sigma V}^\adj$.  Then
\begin{equation}
\label{eq:intro_pca_bd}
\Expect \norm{ \mtx{A} - \mtx{U\Sigma V}^\adj }
    \leq \left[ 1 + 4 \sqrt{\frac{2\min\{m,n\}}{k-1}} \right]^{1/(2q+1)} \sigma_{k+1},
\end{equation}
where $\Expect$ denotes expectation with respect to the
random test matrix and $\sigma_{k+1}$ is the $(k+1)$th
singular value of $\mtx{A}$.
\end{theorem}

\lsp

This result is new.
Observe that the bracket in~\eqref{eq:intro_pca_bd} is essentially
the same as the bracket in the basic error bound~\eqref{eq:intro_err_bd}.
We find that the power iteration drives the leading constant to one
exponentially fast as the power $q$ increases.
The rank-$k$ approximation of $\mtx{A}$ can never achieve an error
smaller than $\sigma_{k+1}$, so the randomized procedure computes $2k$
approximate singular vectors that capture as much of the matrix
as the first $k$ actual singular vectors.

In practice, we can truncate the approximate SVD, retaining only the
first $k$ singular values and vectors.  Equivalently, we
replace the diagonal factor $\mtx{\Sigma}$ by the matrix
$\mtx{\Sigma}_{(k)}$ formed by zeroing out all but the
largest $k$ entries of $\mtx{\Sigma}$.  For this truncated SVD, we have the error bound
\begin{equation} \label{eqn:pca_trunc}
\Expect \norm{ \mtx{A} - \mtx{U} \mtx{\Sigma}_{(k)} \mtx{V}^\adj }
	\leq \sigma_{k+1} + \left[ 1 + 4 \sqrt{\frac{2\min\{m,n\}}{k-1}} \right]^{1/(2q+1)} \sigma_{k+1}.
\end{equation}
In words, we pay no more than an additive term $\sigma_{k+1}$ when we
perform the truncation step.  Our numerical experience suggests that
the error bound~\eqref{eqn:pca_trunc} is pessimistic.
See Remark~\ref{rem:truncation} and~\S\ref{sec:truncation-analysis} for some
discussion of truncation.

\subsection{Outline of paper}

The paper is organized into three parts: an introduction
(\S\S\ref{sec:intro}--\ref{sec:LA_prel}), a description of the
algorithms (\S\S\ref{sec:algorithm}--\ref{sec:numerics}), and
a theoretical performance analysis (\S\S\ref{sec:lin-alg-prelim}--\ref{sec:SRFTs}).
The two latter parts commence with a short internal outline.
Each part is more or less self-contained, and after a brief review of our notation
in~\S\S\ref{sec:basic_def}--\ref{sec:standard_factorizations},
the reader can proceed to either the algorithms or the theory part.

\lsp

\section{Related work and historical context}
\label{sec:related}

Randomness has occasionally surfaced in the numerical linear algebra
literature; in particular, it is quite standard to initialize
iterative algorithms for constructing invariant subspaces with a randomly
chosen point.  Nevertheless, we believe that sophisticated ideas from
random matrix theory have not been incorporated into classical matrix
factorization algorithms until very recently.  We can trace this
development to earlier work in computer science and---especially---to
probabilistic methods in geometric analysis.
This section presents an overview of the relevant work.
We begin with a survey of randomized methods for matrix approximation;
then we attempt to trace some of the ideas backward to their sources.

\subsection{Randomized matrix approximation}

Matrices of low numerical rank contain little
information relative to their apparent dimension owing to
the linear dependency in their columns (or rows).
As a result, it
is reasonable to expect that these matrices can be approximated
with far fewer degrees of freedom.
A less obvious fact is that randomized schemes
can be used to produce these approximations efficiently.

Several types of approximation techniques build on this
idea.  These methods all follow the same basic pattern:

\lsp


%
%

\begin{enumerate}
\item   Preprocess the matrix, usually to calculate sampling probabilities.

\item   Take random samples from the matrix, where the term \term{sample}
refers generically to a linear function of the matrix.

\item   Postprocess the samples to compute a final approximation, typically with classical techniques
from numerical linear algebra.  This step may require another look at the matrix.
\end{enumerate}

\lsp

\noindent
We continue with a description of the most common approximation schemes.


\subsubsection{Sparsification}

The simplest approach to matrix approximation is the method of
\term{sparsification} or the related technique of \term{quantization}.
The goal of sparsification is to replace the matrix by a surrogate that contains
far fewer nonzero entries.  Quantization produces an approximation
whose components are drawn from a (small) discrete set of values.  These
methods can be used to limit storage requirements or to accelerate computations
by reducing the cost of matrix--vector
and matrix--matrix multiplies~\cite[Ch.~6]{McS04:Spectral-Methods}.
The manuscript~\cite{Asp09:Subsampling-Algorithms} describes applications in
optimization.

Sparsification typically involves very simple elementwise calculations.  Each
entry in the approximation is drawn independently at random from a distribution
determined from the corresponding entry of the input matrix.  The expected
value of the random approximation equals the original matrix, but the distribution
is designed so that a typical realization is much sparser.

The first method of this form was devised by Achlioptas and
McSherry~\cite{achlioptas_mcsherry}, who built on earlier work on
graph sparsification due to
Karger~\cite{Kar99:Random-Sampling,Kar00:Minimum-Cuts}.
Arora--Hazan--Kale presented a different sampling method
in~\cite{AHK06:Fast-Random}.
See~\cite{SS08:Graph-Sparsification,GT09:Error-Bounds} for some
recent work on sparsification.

\subsubsection{Column selection methods}

A second approach to matrix approximation is based on the idea that
a small set of columns describes most of the action of a numerically
low-rank matrix. Indeed, classical existential
results~\cite{Rus64:Auerbachs-Theorem} demonstrate that every $m
\times n$ matrix $\mtx{A}$ contains a $k$-column submatrix $\mtx{C}$
for which
\begin{equation} \label{eqn:CSSP-rel}
\norm{ \mtx{A} - \mtx{CC}^\psinv \mtx{A} }
    \leq \sqrt{1 + k(n-k)} \cdot \norm{ \mtx{A} - \mtx{A}_{(k)} },
\end{equation}
where $k$ is a parameter, the dagger $\psinv$ denotes the pseudoinverse,
and $\mtx{A}_{(k)}$ is a best rank-$k$ approximation of $\mtx{A}$.
It is ${\sf NP}$-hard
to perform column selection by optimizing natural objective functions,
such as the condition number of the submatrix~\cite{CM09:Selecting-Maximum}.  Nevertheless, there are
efficient deterministic algorithms, such as the rank-revealing QR method of~\cite{gu_rrqr},
that can nearly achieve the error bound~\eqref{eqn:CSSP-rel}.

There is a class of randomized algorithms that approach the
fixed-rank approximation problem~\eqref{eq:fixed_rank} using this intuition.
These methods first
compute a sampling probability for each column, either using the squared
Euclidean norms of the columns or their \emph{leverage scores}.
(Leverage scores reflect the relative importance of the columns to the
action of the matrix; they can be calculated easily from the dominant $k$
right singular vectors of the matrix.)  Columns are then selected randomly
according to this distribution.  Afterward, a postprocessing step is invoked
to produce a more refined approximation of the matrix.

We believe that the earliest method of this form appeared in a 1998 paper of
Frieze--Kannan--Vempala~\cite{FKV98:Fast-Monte-Carlo,kannan_vempala}.
This work was refined substantially in the
papers~\cite{DFKVV99:Clustering-Large,DFKVV04:Clustering-Large,drineas_kannan_mahoney}.
The basic algorithm samples columns from a distribution related to the
squared $\ell_2$ norms of the columns.  This sampling step produces a
small column submatrix whose range is aligned with the range of the
input matrix. The final approximation is obtained from a truncated
SVD of the submatrix. Given a target rank $k$ and a parameter $\eps
> 0$, this approach samples $\ell = \ell(k,\eps)$ columns of the
matrix to produce a rank-$k$ approximation $\mtx{B}$ that
satisfies
\begin{equation} \label{eqn:CSSP-abs}
\fnorm{ \mtx{A} - \mtx{B} } \leq \fnorm{ \mtx{A} - \mtx{A}_{(k)} } + \eps \fnorm{\mtx{A}},
\end{equation}
where $\fnorm{\cdot}$ denotes the Frobenius norm.
We note that the algorithm of~\cite{drineas_kannan_mahoney}
requires only a constant number of passes over the data.


Rudelson and Vershynin later showed that the same type of column sampling method also
yields spectral-norm error bounds~\cite{RV07:Sampling-Large}.  The
techniques in their paper have been very influential; their work has
found other applications in randomized regression~\cite{DMMS09:Faster-Least},
sparse approximation~\cite{Tro08:Conditioning-Random}, and compressive
sampling~\cite{CR07:Sparsity-Incoherence}.

Deshpande et al.~\cite{DRVW06:Matrix-Approximation,deshpande_vempala} demonstrated that the error
in the column sampling approach can be improved by iteration and adaptive volume sampling.
They showed that it is possible to produce a rank-$k$ matrix $\mtx{B}$ that satisfies
\begin{equation} \label{eqn:colselect-rel}
\fnorm{ \mtx{A} - \mtx{B} } \leq (1 + \eps) \fnorm{ \mtx{A} - \mtx{A}_{(k)} }
\end{equation}
using a $k$-pass algorithm.  Around the same time, Har-Peled~\cite{Har06:Matrix-Approximation}
independently developed a recursive algorithm that offers the same approximation guarantees.
Very recently, Desphande and Rademacher have improved the running time of volume-based
sampling methods~\cite{DR10:Efficient-Volume}.

Drineas et al.~and Boutsidis et al.~have also developed randomized algorithms for the
\term{column subset selection problem}, which requests a column submatrix $\mtx{C}$
that achieves a bound of the form~\eqref{eqn:CSSP-rel}.
Via the methods of Rudelson and Vershynin~\cite{RV07:Sampling-Large}, they showed
that sampling columns according to their leverage scores is likely
to produce the required submatrix~\cite{DMM06:Subspace-Sampling,DMM08:Relative-Error}.
Subsequent
work~\cite{BMD09:Improved-Approximation,BDM08:Unsupervised-Feature}
showed that postprocessing the sampled columns with a rank-revealing QR
algorithm can reduce the number of output columns required~\eqref{eqn:CSSP-rel}.
The argument in~\cite{BMD09:Improved-Approximation} explicitly decouples the
linear algebraic part of the analysis from the random matrix theory.
The theoretical analysis in the present work involves a very similar technique.







%



\subsubsection{Approximation by dimension reduction}

A third approach to matrix approximation is based on the concept of
\emph{dimension reduction}. Since the rows of a low-rank matrix are
linearly dependent, they can be embedded into a low-dimensional
space without altering their geometric properties substantially.  A
random linear map provides an efficient, nonadaptive way to perform
this embedding.  (Column sampling can also be viewed as an adaptive
form of dimension reduction.)

The proto-algorithm we set forth in~\S\ref{sec:proto-algorithm}
is simply a dual description of
the dimension reduction approach: collecting random samples from the
column space of the matrix is equivalent to reducing the dimension
of the rows.  No precomputation is required to obtain the sampling
distribution, but the sample itself takes some work to collect.
Afterward, we orthogonalize the samples as preparation for
constructing various matrix approximations.

We believe that the idea of using dimension reduction for algorithmic
matrix approximation first appeared in
a 1998 paper of Papadimitriou et al.~\cite{PRTV98:Latent-Semantic,papadimitriou},
who described an application to latent
semantic indexing (LSI). They suggested projecting the input matrix onto a
random subspace 
and compressing the original matrix to (a subspace of) the range
of the projected matrix.  They established error bounds that echo the
result~\eqref{eqn:CSSP-abs} of Frieze et al.~\cite{kannan_vempala}.
Although the Euclidean column selection method is a more computationally
efficient way to obtain this type of error bound, dimension reduction has
other advantages, e.g., in terms of accuracy.

Sarl{\'o}s argued in~\cite{Sar06:Improved-Approximation} that the
computational costs of dimension reduction can be reduced
substantially by means of the structured random maps proposed by
Ailon--Chazelle~\cite{AC06:Approximate-Nearest}.
Sarl{\'o}s used these ideas to develop efficient randomized algorithms
for least-squares problems; he also studied approximate
matrix multiplication and low-rank matrix approximation.
The recent paper~\cite{NDT09:Fast-Efficient} analyzes
a very similar matrix approximation algorithm
using Rudelson and Vershynin's methods~\cite{RV07:Sampling-Large}.

The initial work of Sarl{\'o}s on structured dimension reduction did not
immediately yield algorithms for low-rank matrix approximation
that were superior to classical techniques.
Woolfe et al.~showed how to obtain an improvement in
asymptotic computational cost, and they applied these techniques
to problems in scientific computing~\cite{random2}.
Related work includes~\cite{liberty_diss,2007_PNAS}.

Martinsson--Rokhlin--Tygert have studied dimension
reduction using a Gaussian transform matrix, and they demonstrated
that this approach performs much better than earlier analyses had
suggested~\cite{random1}. Their work highlights the importance of
oversampling, and their error bounds are very similar to the
estimate~\eqref{eq:intro_err_prob} we presented in the introduction.
They also demonstrated that dimension reduction can be used to
compute an interpolative decomposition of the input matrix, which is
essentially equivalent to performing column subset selection.

Rokhlin--Szlam--Tygert have shown that combining dimension reduction
with a power iteration is an effective way to improve its performance~\cite{tygert_szlam}.
These ideas lead to very efficient randomized methods for large-scale
PCA~\cite{2010_outofcore}.
An efficient, numerically stable version of the power iteration
is discussed in~\S\ref{sec:powerscheme}, as well as~\cite{Szlam10}.
Related ideas appear in a paper of Roweis~\cite{roweis}.

Very recently, Clarkson and Woodruff~\cite{2009_clarkson_woodruff}
have developed one-pass algorithms for performing low-rank matrix
approximation, and they have established lower bounds which prove
that many of their algorithms have optimal or near-optimal resource guarantees, modulo constants.









\subsubsection{Approximation by submatrices}

The matrix approximation literature contains a
subgenre that discusses methods for building an approximation from a
submatrix and computed coefficient matrices. For example, we can
construct an approximation using a subcollection of columns (the
interpolative decomposition), a subcollection of rows and a
subcollection of columns (the CUR decomposition), or a square
submatrix (the matrix skeleton).  This type of decomposition was
developed and studied in several papers,
including~\cite{mskel,GTZ97:Theory-Pseudoskeleton,Ste99:Four-Algorithms}.
For data analysis applications, see the recent paper~\cite{MD09:CUR-Matrix}.

A number of works develop randomized algorithms for this class
of matrix approximations.
Drineas et al.~have developed techniques for computing CUR decompositions,
which express $\mtx{A} \approx \mtx{CUR}$, where $\mtx{C}$ and $\mtx{R}$ denote
small column and row submatrices of $\mtx{A}$ and where $\mtx{U}$ is
a small linkage matrix. These methods identify columns (rows) that
approximate the range (corange) of the matrix; the linkage matrix is
then computed by solving a small least-squares problem. A randomized
algorithm for CUR approximation with controlled absolute error
appears in~\cite{DKM06:Fast-Monte-Carlo-III}; a relative error
algorithm appears in~\cite{DMM08:Relative-Error}.
We also mention a paper on computing a closely related factorization
called the \term{compact matrix decomposition}~\cite{SXZF08:Less-Is-More}.

It is also possible to produce interpolative decompositions and matrix
skeletons using randomized methods, as discussed in~\cite{random1,tygert_szlam}
and~\S\ref{sec:postrows} of the present work.

\subsubsection{Other numerical problems}

The literature contains a variety of other randomized algorithms for solving
standard problems in and around numerical linear algebra.  We list some of
the basic references.

\lsp

\begin{description}
\item[Tensor skeletons.]
Randomized column selection methods can be used to produce CUR-type
decompositions of higher-order
tensors~\cite{DM07:Randomized-Algorithm}.

\item[Matrix multiplication.]
Column selection and dimension reduction techniques can be used to
accelerate the multiplication of rank-deficient
matrices~\cite{DKM06:Fast-Monte-Carlo-I,Sar06:Improved-Approximation}.
See also~\cite{BW08:Sparse-Representation}.

\item[Overdetermined linear systems.]
The randomized Kaczmarz algorithm is a linearly convergent iterative
method that can be used to solve overdetermined linear
systems~\cite{Nee09:Randomized-Kaczmarz,SV08:Randomized-Kaczmarz}.

\item[Overdetermined least squares.]
Fast dimension-reduction maps can sometimes accelerate the solution
of overdetermined least-squares
problems~\cite{DMMS09:Faster-Least,Sar06:Improved-Approximation}.

\item[Nonnegative least squares.]
Fast dimension reduction maps can be used to reduce the size of
nonnegative least-squares problems~\cite{BD09:Random-Projections}.

\item[Preconditioned least squares.]
Randomized matrix approximations can be used to precondition
conjugate gradient to solve least-squares
problems~\cite{2008_rokhlin_leastsquares}.

\item[Other regression problems.]
Randomized algorithms for $\ell_1$ regression are described
in~\cite{Cla05:Subgradient-Sampling}.  Regression in $\ell_p$ for $p
\in [1, \infty)$ has also been
considered~\cite{DDHKM09:Sampling-Algorithms}.

\item[Facility location.]
The Fermat--Weber facility location problem can be viewed as matrix
approximation with respect to a different discrepancy measure.
Randomized algorithms for this type of problem appear
in~\cite{SV07:Efficient-Subspace}.
\end{description}


\subsubsection{Compressive sampling}

Although randomized matrix approximation and compressive sampling
are based on some common intuitions, it is facile to consider either
one as a subspecies of the other. We offer a short overview of the
field of compressive sampling---especially the part connected with
matrices---so we can highlight some of the differences.

The theory of compressive sampling starts with the observation that
many types of vector-space data are \emph{compressible}.  That is,
the data are approximated well using a short linear combination of
basis functions drawn from a 
fixed collection~\cite{DVDD98:Data-Compression}. For example, natural
images are well approximated in a wavelet basis; numerically
low-rank matrices are well approximated as a sum of rank-one
matrices.  The idea behind compressive sampling is that suitably
chosen random samples from this type of compressible object carry a
large amount of information. Furthermore, it is possible to
reconstruct the compressible object from a small set of these random
samples, often by solving a convex optimization problem. The initial
discovery works of
Cand{\`e}s--Romberg--Tao~\cite{CRT06:Robust-Uncertainty} and
Donoho~\cite{Don06:Compressed-Sensing} were written in 2004.

The earliest work in compressive sampling focused on vector-valued
data; soon after, researchers began to study compressive sampling
for matrices. In 2007, Recht--Fazel--Parillo demonstrated that it is
possible to reconstruct a rank-deficient matrix from
Gaussian measurements~\cite{RFP09:Guaranteed-Minimum}.
More recently, Cand{\`e}s--Recht~\cite{CR08:Exact-Matrix} and Cand{\`e}s--Tao~\cite{CT09:Power-Convex}
considered the problem of completing a low-rank matrix from a random sample of its entries.

The usual goals of compressive sampling are
(i) to design a method for collecting informative, nonadaptive data about
a compressible object and (ii) to reconstruct a compressible object given
some measured data.  In both cases, there is an implicit assumption that we
have limited---if any---access to the underlying data.

In the problem of matrix approximation, we typically have a complete representation
of the matrix at our disposal.  The point is to compute a simpler representation
as efficiently as possible under some operational constraints.  In particular,
we would like to perform as little computation as we can, but we are usually
allowed to revisit the input matrix.  Because of the
different focus, randomized matrix approximation algorithms require fewer random
samples from the matrix and use fewer computational resources than
compressive sampling reconstruction algorithms.





\subsection{Origins}

This section attempts to identify some
of the major threads of research that ultimately led to the development
of the randomized techniques we discuss in this paper.




\subsubsection{Random embeddings}

The field of random embeddings is a major
precursor to randomized matrix approximation.
In a celebrated 1984
paper~\cite{JL84:Extensions-Lipschitz}, Johnson and Lindenstrauss
showed that the pairwise distances among a collection of $N$ points
in a Euclidean space are approximately maintained when the points
are mapped randomly to a Euclidean space of dimension $\bigO(\log
N)$. In other words, random embeddings preserve Euclidean geometry.
Shortly afterward, Bourgain showed that appropriate random
low-dimensional embeddings preserve the geometry of point sets in
finite-dimensional $\ell_1$ spaces~\cite{Bou85:Lipschitz-Embedding}.

These observations suggest that we might be able to solve some
computational problems of a geometric nature more efficiently by
translating them into a lower-dimensional space and solving them
there.  This idea was cultivated by the theoretical computer science
community beginning in the late 1980s, with research flowering in
the late 1990s. In particular, nearest-neighbor search can benefit
from dimension-reduction
techniques~\cite{IM98:Approximate-Nearest,Kle97:Two-Algorithms,KOR00:Efficient-Search}.
The papers~\cite{FKV98:Fast-Monte-Carlo,PRTV98:Latent-Semantic}
were apparently the first to apply this approach to linear algebra.

Around the same time, researchers became interested in simplifying
the form of dimension reduction maps and improving the computational
cost of applying the map.  Several researchers developed refined
results on the performance of a Gaussian matrix as a linear
dimension reduction
map~\cite{DG99:Elementary-Proof,IM98:Approximate-Nearest,Mat02:Lectures-Discrete}.
Achlioptas demonstrated that discrete random matrices would serve
nearly as well~\cite{Ach03:Database-Friendly}.
%
%
In 2006, Ailon and Chazelle proposed the \term{fast
Johnson--Lindenstrauss transform}~\cite{AC06:Approximate-Nearest},
which combines the speed of the FFT with the favorable embedding
properties of a Gaussian matrix.
%
%
Subsequent refinements appear in~\cite{AL08:Fast-Dimension,%
LAS08:Dense-Fast}.  Sarl{\'o}s then imported these techniques to
study several problems in numerical linear algebra, which
has led to some of the fastest algorithms currently
available~\cite{2007_PNAS,random2}.

\subsubsection{Data streams}

Muthukrishnan argues that a distinguishing feature of modern data is
the manner in which it is \emph{presented} to us. The sheer volume
of information and the speed at which it must be processed tax our
ability to \emph{transmit} the data elsewhere, to \emph{compute}
complicated functions on the data, or to \emph{store} a substantial
part of the data~\cite[\S3]{Mut05:Data-Streams}.  As a result,
computer scientists have started to develop algorithms that can
address familiar computational problems under these novel
constraints. The data stream phenomenon is one of the primary
justifications cited by~\cite{DKM06:Fast-Monte-Carlo-I} for
developing pass-efficient methods for numerical linear algebra
problems, and it is also the focus of the recent
treatment~\cite{2009_clarkson_woodruff}.

One of the methods for dealing with massive data sets is to maintain \emph{sketches},
which are small summaries that allow functions of interest to be calculated.
In the simplest case, a sketch is simply a random projection of the data,
but it might be a more sophisticated object~\cite[\S5.1]{Mut05:Data-Streams}.
The idea of sketching can be traced to the work of
Alon et al.~\cite{AMS96:Space-Complexity,AGMS99:Tracking-Join}.



%




%




\subsubsection{Numerical linear algebra}

Classically, the field of numerical linear algebra has focused on
developing deterministic algorithms that produce highly accurate
matrix approximations with provable guarantees. Nevertheless,
randomized techniques have appeared in several environments.

One of the original examples is the use of
random models for arithmetical errors, which was pioneered by von Neumann
and Goldstine.
Their papers~\cite{NG47:Numerical-Inverting,NG51:Numerical-Inverting-II} stand among the first works
to study the properties of random matrices.
The earliest numerical linear algebra algorithm that
depends essentially on randomized techniques is probably Dixon's
method for estimating norms and condition
numbers~\cite{Dix83:Estimating-Extremal}.

Another situation where randomness commonly arises is the initialization of
iterative methods for computing invariant subspaces.  For example,
most numerical linear algebra texts advocate random selection of the
starting vector for the power method because it ensures that the
vector has a nonzero component in the direction of a dominant
eigenvector.  Wo{\' z}niakowski and coauthors have analyzed the
performance of the power method and the Lanczos iteration given a
random starting
vector~\cite{KW92:Estimating-Largest,LW98:Estimating-Largest}.

Among other interesting applications of randomness, we mention
the work by Parker and Pierce, which applies a randomized FFT to
eliminate pivoting in Gaussian elimination~\cite{PP95:Randomizing-FFT},
work by Demmel et al.~who have studied randomization
in connection with the stability of fast methods for
linear algebra~\cite{DDH07:Fast-Linear}, and work by Le and Parker
utilizing randomized methods for stabilizing fast linear algebraic
computations based on recursive algorithms, such as Strassen's matrix
multiplication~\cite{Le_Parker_1999}.



\subsubsection{Scientific computing}

One of the first algorithmic applications of randomness
is the method of Monte Carlo integration
introduced by Von Neumann and Ulam \cite{metropolis_ulam}, and its
extensions, such as the Metropolis algorithm for simulations in
statistical physics. (See~\cite{Bei00:Metropolis-Algorithm} for an introduction.)
The most basic technique is to estimate an integral by sampling $m$ points from the
measure and computing an empirical mean of the integrand evaluated
at the sample locations:
$$
\int f(x) \idiff{\mu}(x) \approx \frac{1}{m} \sum_{i=1}^m f(X_i),
$$
where $X_i$ are independent and identically distributed according to
the probability measure $\mu$. The law of large numbers (usually)
ensures that this approach produces the correct result in the limit
as $m \to \infty$.  Unfortunately, the approximation error typically
has a standard deviation of $m^{-1/2}$, and the method provides
no certificate of success.

The disappointing computational profile of Monte Carlo integration
seems to have inspired a distaste for randomized approaches
within the scientific computing community.
Fortunately, there are many other types of
randomized algorithms---such as the ones in this paper---that do not
suffer from the same shortcomings.

\subsubsection{Geometric functional analysis}

There is one more character that plays a central role in our story:
the probabilistic method in geometric analysis.
Many of the algorithms and proof techniques
ultimately come from work in this beautiful but recondite corner of
mathematics.

Dvoretsky's theorem~\cite{Dvo61:Some-Results} states (roughly) that
every infinite-dimensional Banach space contains an $n$-dimensional
subspace whose geometry is essentially the same as an
$n$-dimensional Hilbert space, where $n$ is an arbitrary natural
number. In 1971, V.~D.~Milman developed a striking proof of
this result by showing that a \emph{random} $n$-dimensional subspace
of an $N$-dimensional Banach space has this property with
exceedingly high probability, provided that $N$ is large
enough~\cite{Mil71:New-Proof}. Milman's article debuted
the \emph{concentration of measure phenomenon},
which is a geometric interpretation of the classical idea that
regular functions of independent random variables rarely deviate far
from their mean.  This work opened a new era in geometric analysis
where the probabilistic method became a basic instrument.


Another prominent example of measure concentration is Kashin's
computation of the Gel'fand widths of the $\ell_1$
ball~\cite{Kas77:Widths-Certain}, subsequently refined
in~\cite{GG84:Widths-Euclidean}. This work showed that a
\emph{random} $(N - n)$-dimensional projection of the $N$-dimensional
$\ell_1$ ball has an astonishingly small Euclidean diameter:
approximately $\sqrt{(1 + \log(N/n))/n}$.  In contrast, a nonzero projection
of the $\ell_2$ ball always has Euclidean diameter one.
This basic geometric fact undergirds recent developments in
compressive sampling~\cite{Can06:Compressive-Sampling}.

We have already described a third class of examples: the randomized
embeddings of Johnson--Lindenstrauss~\cite{JL84:Extensions-Lipschitz} and
of Bourgain~\cite{Bou85:Lipschitz-Embedding}.

Finally, we mention Maurey's technique of empirical approximation.
The original work was unpublished; one of the earliest applications
appears in~\cite[\S1]{Car85:Inequalities-Bernstein-Jackson}.
Although Maurey's idea has not received as much press as the examples above,
it can lead to simple and efficient algorithms for sparse approximation.
For some examples in machine learning,
consider~\cite{Bar93:Universal-Approximation,LBW96:Efficient-Agnostic,SS08:Low-l1-Norm,RR08:Random-Features}

The importance of random constructions in the geometric analysis
community has led to the development of powerful techniques for
studying random matrices. Classical random matrix theory focuses on
a detailed asymptotic analysis of the spectral properties of special
classes of random matrices.  In contrast, geometric analysts know
methods for determining the approximate behavior of rather complicated
finite-dimensional random matrices. See~\cite{DS02:Local-Operator}
for a fairly current survey article. We also mention
the works of Rudelson~\cite{Rud99:Random-Vectors} and
Rudelson--Vershynin~\cite{RV07:Sampling-Large}, which describe
powerful tools for studying random matrices drawn from certain
discrete distributions.  Their papers are rooted deeply in the field of
geometric functional analysis, but they reach out toward
computational applications.







\section{Linear algebraic preliminaries}
\label{sec:LA_prel}

This section summarizes the background we need
for the detailed description of randomized algorithms
in~\S\S\ref{sec:algorithm}--\ref{sec:costs} and
the analysis in~\S\S\ref{sec:lin-alg-prelim}--\ref{sec:SRFTs}.
We introduce notation in~\S\ref{sec:basic_def}, describe some standard matrix decompositions
in~\S\ref{sec:standard_factorizations}, and briefly review
standard techniques for computing matrix factorizations
in~\S\ref{sec:standard_techniques}.



\subsection{Basic definitions}
\label{sec:basic_def}
The standard Hermitian geometry for $\Cspace{n}$ is induced by the inner product
$$
\ip{ \vct{x} }{ \vct{y} } = \vct{x} \cdot \vct{y} =  \sum\nolimits_{j} x_j\, \overline{y}_j.
$$
The associated norm is
$$
\normsq{ \vct{x} } = \ip{ \vct{x} }{ \vct{x} }
    = \sum\nolimits_{j} \abssq{x_j}.
$$
We usually measure the magnitude of a matrix $\mtx{A}$
with the operator norm
$$
\norm{\mtx{A}} = \max_{\vct{x} \neq \vct{0}}
\frac{\norm{\mtx{A}\vct{x}}}{\norm{\vct{x}}},
$$
which is often referred to as the \term{spectral norm}.
The Frobenius norm is given by
$$
\fnorm{\mtx{A}} = \left[ \sum\nolimits_{jk} \abssq{a_{jk}} \right]^{1/2}.
$$
The conjugate transpose, or \term{adjoint}, of a matrix $\mtx{A}$ is
denoted $\mtx{A}^{\adj}$.  The important identities
$$
\norm{\mtx{A}}^2 = \norm{\mtx{A}^\adj \mtx{A}} = \norm{\mtx{AA}^\adj}
$$
hold for each matrix $\mtx{A}$.

We say that a matrix $\mtx{U}$ is \term{orthonormal} if its columns
form an orthonormal set with respect to the Hermitian inner product.
An orthonormal matrix $\mtx{U}$ preserves geometry in the sense that
$\norm{ \mtx{U} \vct{x}} = \norm{\vct{x}}$ for every vector $\vct{x}$.
A \emph{unitary} matrix is a square orthonormal matrix, and
an \emph{orthogonal} matrix is a real unitary matrix.
Unitary matrices satisfy the relations $\mtx{UU}^\adj = \mtx{U}^\adj \mtx{U} = \Id$.
Both the operator norm and the Frobenius norm are \term{unitarily invariant},
which means that
$$
\norm{ \mtx{UAV}^\adj } = \norm{ \mtx{A} }
\qquad\mbox{and}\qquad
\fnorm{ \mtx{UAV}^\adj } = \fnorm{ \mtx{A} }
$$
for every matrix $\mtx{A}$ and all orthonormal matrices $\mtx{U}$ and $\mtx{V}$

We use the notation of \cite{golub} to denote submatrices.
If $\mtx{A}$ is a matrix with entries $a_{ij}$, and if
$I = [i_{1},\,i_{2},\,\dots,\,i_{p}]$ and $J = [j_{1},\,j_{2},\,\dots,\,j_{q}]$
are two index vectors, then the associated $p\times q$ submatrix is expressed as
$$
\mtx{A}_{(I,J)} = \begin{bmatrix}
a_{i_{1},j_{1}} & \cdots & a_{i_{1},j_{q}} \\
\vdots && \vdots \\
a_{i_{p},j_{1}} & \cdots & a_{i_{p},j_{q}}
\end{bmatrix}.
$$
For column- and row-submatrices, we use the standard abbreviations
$$\mtx{A}_{(\colon,J)} = \mtx{A}_{([1,\,2,\,\dots,\,m],J)},
\qquad\mbox{and}\qquad
\mtx{A}_{(I,\colon)} = \mtx{A}_{(I,[1,\,2,\,\dots,\,n])}.
$$

\subsection{Standard matrix factorizations}
\label{sec:standard_factorizations}

This section defines three basic matrix decompositions.
Methods for computing them are described in \S\ref{sec:standard_techniques}.

\subsubsection{The pivoted QR factorization}
\label{sec:QR}
Each $m \times n$ matrix $\mtx{A}$ of rank $k$ admits a decomposition
$$
\mtx{A} = \mtx{Q}\mtx{R},
$$
where $\mtx{Q}$ is an $m\times k$ orthonormal matrix, and $\mtx{R}$ is
a $k \times n$ \term{weakly upper-triangular} matrix.  That is,
there exists a permutation
$J$ of the numbers $\{1,\,2,\,\dots,\,n\}$ such that $\mtx{R}_{(\colon,J)}$ is upper
triangular. Moreover, the diagonal entries of $\mtx{R}_{(\colon,J)}$ are
weakly decreasing.
See \cite[\S5.4.1]{golub} for details.

\subsubsection{The singular value decomposition (SVD)}
\label{sec:SVD}
Each $m \times n$ matrix $\mtx{A}$ of rank $k$ admits a factorization
$$
\mtx{A} = \mtx{U}\mtx{\Sigma}\mtx{V}^\adj,
$$
where $\mtx{U}$ is an $m\times k$ orthonormal matrix, $\mtx{V}$ is an
$n\times k$ orthonormal matrix, and $\mtx{\Sigma}$ is a $k\times k$ nonnegative,
diagonal matrix
$$
\mtx{\Sigma} = \begin{bmatrix}
\sigma_{1} &  &  &  \\
 & \sigma_{2} &  &  \\
 &  & \ddots  &  \\
 &  &  & \sigma_{k}
\end{bmatrix}.
$$
The numbers $\sigma_{j}$ are called the \term{singular values}
of $\mtx{A}$.  They are arranged in weakly decreasing order:
$$
\sigma_{1} \geq \sigma_{2} \geq \cdots \geq \sigma_{k} \geq 0.
$$
The columns of $\mtx{U}$ and $\mtx{V}$ are called \term{left singular
vectors} and \term{right singular vectors}, respectively.

Singular values are connected with the approximability of
matrices.  For each $j$, the number $\sigma_{j+1}$ equals
the spectral-norm discrepancy between $\mtx{A}$ and an optimal rank-$j$
approximation~\cite{Mir60:Symmetric-Gauge}.  That is,
\begin{equation}
\label{eq:svd_property}
\sigma_{j+1} = \min\{\norm{\mtx{A} - \mtx{B}}\,\colon\,\mtx{B}\mbox{ has rank }j\}.
\end{equation}
In particular, $\sigma_{1} = \norm{\mtx{A}}$.
See \cite[\S 2.5.3 and \S5.4.5]{golub} for additional details.

\subsubsection{The interpolative decomposition (ID)}
\label{sec:ID}


Our final factorization 
identifies a collection of $k$ columns from a rank-$k$ matrix $\mtx{A}$ that
span the range of $\mtx{A}$.
To be precise, we can compute an index set $J = [j_{1},\,\dots,\,j_{k}]$
such that
$$
\mtx{A} = \mtx{A}_{(\colon,J)}\,\mtx{X},
$$
where $\mtx{X}$ is a $k\times n$ matrix that satisfies $\mtx{X}_{(\colon,J)} = \Id_{k}$.
Furthermore, no entry of $\mtx{X}$ has magnitude larger than two.
In other words, this decomposition expresses each column of $\mtx{A}$
using a linear combination of $k$ fixed columns with \emph{bounded} coefficients.
Stable and efficient algorithms for computing the ID appear in the papers~\cite{mskel,gu_rrqr}.

It is also possible to compute a two-sided ID
$$
\mtx{A} = \mtx{W} \,\mtx{A}_{(J',J)}\,\mtx{X},
$$
where $J'$ is an index set identifying $k$ of the rows of $\mtx{A}$,
and $\mtx{W}$ is an $m\times k$ matrix that satisfies $\mtx{W}_{(J',\colon)} = \Id_{k}$
and whose entries are all bounded by two.

\lsp

\begin{remark} \rm
There always exists an ID where the entries in the factor $\mtx{X}$
have magnitude bounded by one.  Known proofs of this fact
are constructive, e.g.,~\cite[Lem.~3.3]{Pan00:Existence-Computation},
but they require us to find a collection of $k$ columns that has ``maximum volume.''
It is {\sf NP}-hard to identify a subset of columns
with this type of extremal property~\cite{CM09:Selecting-Maximum}.
We find it remarkable that ID computations are possible as soon as
the bound on $\mtx{X}$ is relaxed.
\end{remark}

\lsp

\subsection{Techniques for computing standard factorizations}
\label{sec:standard_techniques}

This section discusses some established deterministic techniques
for computing the factorizations presented
in~\S\ref{sec:standard_factorizations}.  The material
on pivoted QR and SVD can be located in any major text on
numerical linear algebra, such as~\cite{golub,trefethen_bau}.
References for the ID include~\cite{gu_rrqr,mskel}.


\subsubsection{Computing the full decomposition}

It is possible to compute the full QR factorization or the full SVD
of an $m \times n$ matrix to double-precision accuracy with
$\bigO(mn\min\{m,n\})$ flops.  Techniques for computing
the SVD are iterative by necessity, but they converge so fast
that we can treat them as finite for practical purposes.


\subsubsection{Computing partial decompositions}
\label{sec:partial_decomp}

Suppose that an $m \times n$ matrix has numerical rank $k$,
where $k$ is substantially smaller than $m$ and $n$.  In this case,
it is possible to produce a structured low-rank decomposition that
approximates the matrix well.
Sections~\ref{sec:algorithm} and~\ref{sec:otherfactorizations}
describe a set of randomized techniques for obtaining these
partial decompositions.  This section briefly reviews the classical
techniques, which also play a role in developing randomized methods.

To compute a partial QR decomposition, the classical device is the
Businger--Golub algorithm, which performs successive orthogonalization
with pivoting on the columns of the matrix. The procedure halts
when the Frobenius norm of the remaining columns is
less than a computational tolerance $\eps$. Letting $\ell$ denote
the number of steps required, the process results in a partial factorization
\begin{equation} \label{eq:QRE}
\mtx{A} = \mtx{QR} + \mtx{E},
\end{equation}
where $\mtx{Q}$ is an $m \times \ell$ orthonormal matrix, $\mtx{R}$
is a $\ell \times n$ weakly upper-triangular matrix, and $\mtx{E}$
is a residual that satisfies $\fnorm{\mtx{E}} \leq \eps$.
The computational cost is $\bigO(\ell mn)$, and the number
$\ell$ of steps taken is typically close to the minimal rank $k$ for
which precision $\varepsilon$ (in the Frobenius norm) is achievable.
The Businger--Golub algorithm can in principle significantly
overpredict the rank, but in practice this problem is very rare
provided that orthonormality is maintained scrupulously.

Subsequent research has led to strong rank-revealing QR algorithms
that succeed for all matrices.  For example, the Gu--Eisenstat algorithm \cite{gu_rrqr}
(setting their parameter $f = 2$)
produces an QR decomposition of the form~\eqref{eq:QRE}, where
$$
\norm{\mtx{E}} \leq \sqrt{1 + 4k(n-k)} \cdot \sigma_{k+1}.
$$
Recall that $\sigma_{k+1}$ is the minimal error possible in a rank-$k$
approximation~\cite{Mir60:Symmetric-Gauge}.  The cost of the Gu--Eisenstat
algorithm is typically $O(kmn)$, but it can be slightly higher in rare cases.
The algorithm can also be used to obtain an approximate ID~\cite{mskel}.

To compute an approximate SVD of a general $m \times n$ matrix,
the most straightforward technique is to compute the full SVD and
truncate it.  This procedure is stable and accurate, but it
requires $\bigO(mn\min\{m,n\})$ flops.  A more efficient approach
is to compute a partial QR factorization and postprocess the
factors to obtain a partial SVD using the methods described below
in~\S\ref{sec:converting}.  This scheme takes only $\bigO(kmn)$ flops.
Krylov subspace methods can also compute partial SVDs at a comparable
cost of $\bigO(kmn)$, but they are less robust.

Note that all the techniques described in this section require
extensive random access to the matrix, and they can be very slow when
the matrix is stored out-of-core.

\subsubsection{Converting from one partial factorization to another}
\label{sec:converting}

Suppose that we have obtained a partial decomposition of a matrix $\mtx{A}$
by some means:
$$
\norm{ \mtx{A} - \mtx{CB} } \leq \eps,
$$
where $\mtx{B}$ and $\mtx{C}$ have rank $k$. Given this information, we can
efficiently compute any of the basic factorizations.

We construct a partial QR factorization using the following three steps:
\lsp
\begin{remunerate}
\item Compute a QR factorization of $\mtx{C}$ so that $\mtx{C} = \mtx{Q}_{1}\mtx{R}_{1}$.
\item Form the product $\mtx{D} = \mtx{R}_{1}\mtx{B}$, and compute a QR factorization: $\mtx{D} = \mtx{Q}_{2}\mtx{R}$.
\item Form the product $\mtx{Q} = \mtx{Q}_{1}\mtx{Q}_{2}$.
\end{remunerate}
\lsp
The result is an orthonormal matrix  $\mtx{Q}$ and a weakly upper-triangular matrix $\mtx{R}$ such
that $\norm{\mtx{A} - \mtx{Q}\mtx{R}} \leq \varepsilon$.

An analogous technique yields a partial SVD:
\lsp
\begin{remunerate}
\item Compute a QR factorization of $\mtx{C}$ so that $\mtx{C} = \mtx{Q}_{1}\mtx{R}_{1}$.
\item Form the product $\mtx{D} = \mtx{R}_{1}\mtx{B}$, and compute an SVD:
      $\mtx{D} = \mtx{U}_{2}\mtx{\Sigma}\mtx{V}^{\adj}$.
\item Form the product $\mtx{U} = \mtx{Q}_{1}\mtx{U}_{2}$.
\end{remunerate}
\lsp
The result is a diagonal matrix $\mtx{\Sigma}$ and orthonormal matrices $\mtx{U}$ and $\mtx{V}$
such that $\norm{\mtx{A} - \mtx{U}\mtx{\Sigma}\mtx{V}^{\adj}} \leq \varepsilon$.

Converting $\mtx{B}$ and $\mtx{C}$ into a partial ID is a one-step process:
\lsp
\begin{remunerate}
\item Compute $J$ and $\mtx{X}$ such that $\mtx{B} = \mtx{B}_{(\colon,J)}\mtx{X}$.
\end{remunerate}
\lsp
Then $\mtx{A} \approx \mtx{A}_{(\colon,J)}\mtx{X}$, but the approximation error may
deteriorate from the initial estimate.  For example, if we compute the ID using the Gu--Eisenstat algorithm~\cite{gu_rrqr} with the parameter $f = 2$, then the error
$\smnorm{}{\mtx{A} - \mtx{A}_{(\colon,J)}\mtx{X}} \leq (1 + \sqrt{1 + 4k(n-k)}) \cdot \eps$.
Compare this bound with Lemma~\ref{lem:convert_ID} below.

\subsubsection{Krylov-subspace methods}
\label{sec:krylov}

Suppose that the matrix $\mtx{A}$ can be applied rapidly to vectors,
as happens when $\mtx{A}$ is sparse or structured.
Then Krylov subspace techniques can very effectively
and accurately compute partial spectral decompositions.
For concreteness, assume that $\mtx{A}$ is Hermitian.
The idea of these techniques is to fix a starting vector $\vct{\omega}$
and to seek approximations to the eigenvectors within the corresponding \term{Krylov subspace}
$$
{\cal V}_q(\vct{\omega}) = \lspan\{ \vct{\omega}, \mtx{A}\vct{\omega},
\mtx{A}^2 \vct{\omega}, \dots, \mtx{A}^{q-1} \vct{\omega} \}.
$$
Krylov methods also come in blocked versions, in which the starting \textit{vector}
$\vct{\omega}$ is replaced by a starting \textit{matrix} $\mtx{\Omega}$.
A common recommendation is to draw a starting vector $\vct{\omega}$
(or starting matrix $\mtx{\Omega}$) from a standardized Gaussian
distribution, which indicates a significant overlap
between Krylov methods and the methods in this paper.

The most basic versions of Krylov methods for computing spectral
decompositions are numerically unstable.  High-quality
implementations require that we incorporate restarting
strategies, techniques for maintaining high-quality bases for the
Krylov subspaces, etc. The diversity and complexity of
such methods make it hard to state a precise computational
cost, but in the environment we consider in this paper,
a typical cost for a fully stable implementation would be
\begin{equation}
\label{eq:Tkrylov}
T_{\rm Krylov} \sim k \, T_{\rm mult} + k^2 (m+n),
\end{equation}
where $T_{\rm mult}$ is the cost of a matrix--vector multiplication.

\vspace{0.5in}

\begin{center}
{\bf Part II: Algorithms}
\end{center}

\lsp

This part of the paper, \S\S\ref{sec:algorithm}--\ref{sec:numerics},
provides detailed descriptions of randomized algorithms for constructing
low-rank approximations to matrices.  As discussed in \S\ref{sec:framework},
we split the problem into two stages.  In Stage A, we construct a subspace that captures
the action of the input matrix.  In Stage B, we use this subspace to obtain an
approximate factorization of the matrix.

Section~\ref{sec:algorithm} develops randomized methods for completing Stage A,
and \S\ref{sec:otherfactorizations} describes deterministic methods for Stage B.
Section~\ref{sec:costs} compares the computational costs of the resulting two-stage
algorithm with the classical approaches outlined in~\S\ref{sec:LA_prel}.
Finally, \S\ref{sec:numerics} illustrates the performance of the randomized schemes
via numerical examples.

\section{Stage A: Randomized schemes for approximating the range}
\label{sec:algorithm}

This section outlines techniques for constructing a subspace that captures
most of the action of a matrix.  We begin with a recapitulation of the
proto-algorithm that we introduced in \S\ref{sec:sketchofalgorithm}.
We discuss how it can be implemented in practice (\S\ref{sec:proto_revisited})
and then consider the question of how many random samples to acquire
(\S\ref{sec:howmuchoversampling}).
Afterward, we present several ways in which the basic scheme can be improved.
Sections~\ref{sec:aposteriori} and~\ref{sec:algorithm1} explain how to address
the situation where the numerical rank of the input matrix is not known in advance.
Section~\ref{sec:powerscheme} shows how to modify the scheme to improve its
accuracy when the singular spectrum of the input matrix decays slowly.
Finally, \S\ref{sec:ailonchazelle} describes how the scheme can be accelerated
by using a structured random matrix.

\subsection{The proto-algorithm revisited}
\label{sec:proto_revisited}

The most natural way to implement the proto-algorithm from \S\ref{sec:sketchofalgorithm}
is to draw a random test matrix $\mtx{\Omega}$ from the standard Gaussian distribution.
That is, each entry of $\mtx{\Omega}$ is an independent Gaussian random variable with
mean zero and variance one.  For reference, we formulate the resulting scheme as
Algorithm~\ref{alg:basic}.

\begin{figure}
\begin{center}
\fbox{
\begin{minipage}{.9\textwidth}

\begin{center}
\textsc{Algorithm \ref{alg:basic}: Randomized Range Finder}
\end{center}

\lsp

\textit{Given an $m\times n$ matrix $\mtx{A}$, and an integer $\ell$,
this scheme computes an $m\times \ell$ orthonormal matrix
$\mtx{Q}$ whose range approximates the range of $\mtx{A}$.}

\lsp

\begin{tabbing}
\hspace{5mm} \= \hspace{5mm} \= \hspace{5mm} \= \hspace{5mm} \= \kill
\anum{1} \>Draw an $n\times \ell$ Gaussian random matrix $\mtx{\Omega}$.\\
\anum{2} \>Form the $m\times \ell$ matrix $\mtx{Y} = \mtx{A}\mtx{\Omega}$.\\
\anum{3} \>Construct an $m \times \ell$ matrix $\mtx{Q}$ whose columns form an orthonormal\\
         \> basis for the range of $\mtx{Y}$, e.g., using the QR factorization $\mtx{Y} = \mtx{Q}\mtx{R}$.
\end{tabbing}
\end{minipage}}
\end{center}
\end{figure}

The number $T_{\rm basic}$ of flops required by Algorithm \ref{alg:basic} satisfies
\begin{equation}
\label{eq:cost_basic}
T_{\rm basic} \sim \ell n \, T_{\rm rand} + \ell\,T_{\rm mult} + \ell^{2}m 
\end{equation}
where $T_{\rm rand}$ is the cost of generating a Gaussian random number
and $T_{\rm mult}$ is the cost of multiplying $\mtx{A}$ by a vector.
The three terms in~\eqref{eq:cost_basic} correspond directly with the
three steps of Algorithm~\ref{alg:basic}.

Empirically, we have found that the performance of Algorithm~\ref{alg:basic}
depends very little on the quality of the random number generator used in Step 1.

The actual cost of Step 2 depends substantially on the matrix $\mtx{A}$ and the
computational environment that we are working in.
The estimate~\eqref{eq:cost_basic} suggests that Algorithm~\ref{alg:basic}
is especially efficient when the matrix--vector product $\vct{x} \mapsto \mtx{A}\vct{x}$
can be evaluated rapidly.  In particular, the scheme is appropriate for
approximating sparse or structured matrices.  Turn to~\S\ref{sec:costs} for more details.

The most important implementation issue arises when performing the basis calculation
in Step 3.  Typically, the columns of the sample matrix $\mtx{Y}$ are almost linearly
dependent, so it is imperative to use stable methods for performing
the orthonormalization.  We have found that the Gram--Schmidt procedure,
augmented with the \term{double orthogonalization} described in~\cite{bjorck94},
is both convenient and reliable.  Methods based on Householder reflectors or Givens
rotations also work very well.  Note that very little is gained by pivoting because
the columns of the random matrix $\mtx{Y}$ are independent samples drawn from
the same distribution.

\subsection{The number of samples required}
\label{sec:howmuchoversampling}

The goal of Algorithm~\ref{alg:basic} is to produce an orthonormal matrix $\mtx{Q}$
with few columns that achieves
\begin{equation} \label{eq:numsamperr}
\norm{ (\Id - \mtx{QQ}^\adj) \mtx{A} } \leq \eps,
\end{equation}
where $\eps$ is a specified tolerance.  The number of columns $\ell$
that the algorithm needs to reach this threshold is usually slightly larger than
the minimal rank $k$ of the smallest basis that verifies~\eqref{eq:numsamperr}.
We refer to this discrepancy $p = \ell - k$ as the \term{oversampling parameter}.
The size of the oversampling parameter depends on several factors:

\lsp

\begin{description}
\item[The matrix dimensions.] Very large matrices may require more oversampling.

\item[The singular spectrum.] The more rapid the decay of the singular values,
the less oversampling is needed.  In the extreme case that the matrix has exact
rank $k$, it is not necessary to oversample.

\item[The random test matrix.]  Gaussian matrices succeed with very little oversampling,
but are not always the most cost-effective option.
The structured random matrices discussed in~\S\ref{sec:ailonchazelle}
may require substantial oversampling, but they still yield computational
gains in certain settings.
\end{description}

\lsp

The theoretical results in Part~III provide detailed information about how
the behavior of randomized schemes depends on these factors.  For the moment,
we limit ourselves to some general remarks on implementation issues.

For Gaussian test matrices, it is adequate to choose the oversampling parameter
to be a small constant, such as $p = 5$ or $p = 10$.  There is rarely any
advantage to select $p > k$.  This observation, first presented in~\cite{random1},
demonstrates that a Gaussian test matrix results in a negligible amount
of extra computation.

In practice, the target rank $k$ is rarely known in advance.  Randomized
algorithms are usually implemented in an adaptive fashion where the number
of samples is increased until the error satisfies the desired tolerance.
In other words, the user never \emph{chooses} the oversampling parameter.
Theoretical results that bound the amount of oversampling are valuable
primarily as aids for designing algorithms.  We develop an adaptive approach
in~\S\S\ref{sec:aposteriori}--\ref{sec:algorithm1}.

The computational bottleneck in Algorithm~\ref{alg:basic} is usually the
formation of the product $\mtx{A\Omega}$.  As a result, it often pays
to draw a larger number $\ell$ of samples than necessary because the
user can minimize the cost of the matrix multiplication with tools
such as blocking of operations, high-level linear algebra subroutines,
parallel processors, etc.  This approach may lead to an
ill-conditioned sample matrix $\mtx{Y}$, but the orthogonalization
in Step 3 of Algorithm~\ref{alg:basic} can easily identify the numerical
rank of the sample matrix and ignore the excess samples.  Furthermore,
Stage B of the matrix approximation process succeeds even when
the basis matrix $\mtx{Q}$ has a larger dimension than necessary.

\subsection{A posteriori error estimation}
\label{sec:aposteriori}

Algorithm~\ref{alg:basic} is designed for solving the fixed-rank
problem, where the target rank of the input matrix is specified
in advance.  To handle the fixed-precision problem, where the
parameter is the computational tolerance, we need a scheme for
estimating how well a putative basis matrix $\mtx{Q}$ captures
the action of the matrix $\mtx{A}$.
To do so, we develop a probabilistic error estimator.
These methods are inspired by work of
Dixon~\cite{Dix83:Estimating-Extremal}; our treatment
follows~\cite{2007_PNAS,random2}.

The exact approximation error is $\norm{ (\Id - \mtx{QQ}^\adj) \mtx{A} }$.
It is intuitively plausible that we can obtain some information
about this quantity by computing $\norm{ (\Id - \mtx{QQ}^\adj) \mtx{A}\vct{\omega} }$,
where $\vct{\omega}$ is a standard Gaussian vector.
This notion leads to the following method.
Draw a sequence $\{ \vct{\omega}^{(i)} : i = 1, 2, \dots, r \}$
of standard Gaussian vectors, where $r$ is a small integer
that balances computational cost and reliability.
Then
\begin{equation}
\label{eq:errorest}
\norm{ (\Id - \mtx{Q}\mtx{Q}^{\adj})\mtx{A}}
    \leq 10 \sqrt{\frac{2}{\pi}} \max_{i = 1, \dots, r}
    \smnorm{}{ (\Id - \mtx{Q}\mtx{Q}^{\adj}) \mtx{A}\vct{\omega}^{(i)} }
\end{equation}
with probability at least $1 - 10^{-r}$. This statement follows by setting
$\mtx{B} = (\Id - \mtx{Q}\mtx{Q}^{\adj})\mtx{A}$ and $\alpha =
10$ in the following lemma, whose proof appears in
\cite[\S3.4]{random2}.

\lsp

\begin{lemma}
\label{thm:aposteriori}
Let $\mtx{B}$ be a real $m\times n$ matrix.
Fix a positive integer $r$ and a real number $\alpha > 1$.
Draw an independent family $\{ \vct{\omega}^{(i)} : i = 1, 2, \dots, r \}$
of standard Gaussian vectors.  Then
\begin{equation*}
\norm{\mtx{B}}
    \leq \alpha \sqrt{\frac{2}{\pi}} \max_{i = 1, \dots, r}
    \smnorm{}{\mtx{B}\vct{\omega}^{(i)} }
\end{equation*}
except with probability $\alpha^{-r}$.
\end{lemma}

\lsp

The critical point is that the error estimate~\eqref{eq:errorest}
is computationally inexpensive because it requires only a
small number of matrix--vector products.  Therefore, we can make
a lowball guess for the numerical rank of $\mtx{A}$ and add more
samples if the error estimate is too large.  The asymptotic cost
of Algorithm~\ref{alg:basic} is preserved if we double our guess
for the rank at each step.  For example, we can start with 32 samples,
compute another 32, then another 64, etc.

\lsp

\begin{remark}\rm
\label{remark:better_errorestimate}
The estimate~\eqref{eq:errorest} is actually somewhat crude.
We can obtain a better estimate at a similar computational cost by
initializing a power iteration with a random vector and repeating the process
several times~\cite{2007_PNAS}.
\end{remark}

\subsection{Error estimation (almost) for free}
\label{sec:algorithm1}

The error estimate described in~\S\ref{sec:aposteriori} can be combined
with any method for constructing an approximate basis for the range of
a matrix.  In this section, we explain how the error estimator
can be incorporated into Algorithm~\ref{alg:basic} at almost no
additional cost.

To be precise, let us suppose that $\mtx{A}$ is an $m \times n$ matrix
and $\eps$ is a computational tolerance.  We seek an integer $\ell$ and
an $m \times \ell$ orthonormal matrix $\mtx{Q}^{(\ell)}$ such that
\begin{equation} \label{eqn:err_est_err_bd}
\smnorm{}{ \big(\Id - \mtx{Q}^{(\ell)} (\mtx{Q}^{(\ell)})^{\adj} \big)\mtx{A} } \leq \eps.
\end{equation}
The size $\ell$ of the basis will typically be slightly larger than
the size $k$ of the smallest basis that achieves this error.

The basic observation behind the adaptive scheme is that we can
generate the basis in Step 3 of Algorithm~\ref{alg:basic} incrementally.
Starting with an empty basis matrix $\mtx{Q}^{(0)}$, the following
scheme generates an orthonormal matrix whose range captures the action
of $\mtx{A}$:
\lsp
\begin{tabbing}
\hspace{8mm} \= \hspace{8mm} \= \hspace{8mm} \= \hspace{8mm} \= \kill
\>\textbf{for} $i = 1,\,2,\,3,\dots$\\
\>\> Draw an $n\times 1$ Gaussian random vector $\vct{\omega}^{(i)}$, and set
      $\vct{y}^{(i)} = \mtx{A}\vct{\omega}^{(i)}$.\\
\>\> Compute $\tilde{\vct{q}}^{(i)} = \left(\Id - \mtx{Q}^{(i-1)}(\mtx{Q}^{(i-1)})^{\adj}\right)\vct{y}^{(i)}$.\\
\>\> Normalize $\vct{q}^{(i)} = \tilde{\vct{q}}^{(i)}/\smnorm{}{\tilde{\vct{q}}^{(i)}}$,
and form $\mtx{Q}^{(i)} = [\mtx{Q}^{(i-1)}\ \vct{q}^{(i)}]$.\\
\>\textbf{end for}
\end{tabbing}
\lsp
How do we know when we have reached a basis $\mtx{Q}^{(\ell)}$
that verifies~\eqref{eqn:err_est_err_bd}?  The answer becomes
apparent once we observe that the vectors $\tilde{\vct{q}}^{(i)}$
are precisely the vectors that appear in the error bound~\eqref{eq:errorest}.
The resulting rule is that we break the loop once we observe $r$
consecutive vectors $\tilde{\vct{q}}^{(i)}$ whose norms are smaller
than $\eps/(10\sqrt{2/\pi})$.

A formal description of the resulting algorithm appears as Algorithm~\ref{alg:adaptive2}.
A potential complication of the method is that the vectors $\tilde{\vct{q}}^{(i)}$ become
small as the basis starts to capture most of the action of $\mtx{A}$.
In finite-precision arithmetic, their direction is extremely
unreliable.  To address this problem, we simply reproject
the normalized vector $\vct{q}^{(i)}$ onto $\range(\mtx{Q}^{(i-1)})^{\perp}$
in steps 7 and 8 of Algorithm~\ref{alg:adaptive2}.

The CPU time requirements of Algorithms \ref{alg:adaptive2} and \ref{alg:basic}
are essentially identical.
Although Algorithm \ref{alg:adaptive2} computes the last few samples purely to obtain
the error estimate, this apparent extra cost is offset by the fact that
Algorithm~\ref{alg:basic} always includes an oversampling factor.
The failure probability stated for Algorithm \ref{alg:adaptive2} is pessimistic because it
is derived from a simple union bound argument.  In practice,
the error estimator is reliable in a range of circumstances when we take $r = 10$.

\begin{figure}[h]
\begin{center}
\fbox{
\begin{minipage}{.9\textwidth}

\begin{center}
\textsc{Algorithm \ref{alg:adaptive2}: Adaptive Randomized Range Finder}
\end{center}

\lsp

\textit{Given an $m\times n$ matrix $\mtx{A}$, a tolerance $\varepsilon$,
and an integer $r$ (e.g.~$r=10$), the following scheme computes an orthonormal matrix
$\mtx{Q}$ such that \eqref{eq:numsamperr} holds with probability at
least $1 - \min\{m,n\} 10^{-r}$.} 

\lsp

\begin{tabbing}
\hspace{5mm} \= \hspace{5mm} \= \hspace{5mm} \= \hspace{5mm} \= \kill
\anum   {1} \> Draw standard Gaussian vectors $\vct{\omega}^{(1)}, \dots, \vct{\omega}^{(r)}$ of length $n$.\\
\anum{2} \> For $i = 1,2,\dots,r$, compute $\vct{y}^{(i)} = \mtx{A}\vct{\omega}^{(i)}$.\\
\anum{3} \> $j=0$.\\
\anum{4} \> $\mtx{Q}^{(0)} = [\ ]$, the $m\times 0$ empty matrix. \\
\anum{5} \> \textbf{while} $\displaystyle
         \max\left\{\smnorm{}{\vct{y}^{(j+1)}},\smnorm{}{\vct{y}^{(j+2)}},\dots,\smnorm{}{\vct{y}^{(j+r)}} \right\} >
\varepsilon/(10\sqrt{2/\pi})$,\\
\anum{6} \> \> $j = j + 1$.\\
\anum{7} \> \> Overwrite $\vct{y}^{(j)}$ by $\bigl(\Id - \mtx{Q}^{(j-1)}(\mtx{Q}^{(j-1)})^{\adj}\bigr)\vct{y}^{(j)}$.\\
\anum{8} \> \> $\vct{q}^{(j)} = \vct{y}^{(j)}/\norm{\vct{y}^{(j)}}$.\\
\anum{9} \> \> $\mtx{Q}^{(j)} = [\mtx{Q}^{(j-1)}\ \vct{q}^{(j)}]$.\\
\anum{10} \> \> Draw a standard Gaussian vector $\vct{\omega}^{(j+r)}$ of length $n$.\\
\anum{11} \> \> $\vct{y}^{(j+r)} = \left(\Id - \mtx{Q}^{(j)}(\mtx{Q}^{(j)})^{\adj}\right)\mtx{A}\vct{\omega}^{(j+r)}$.\\
\anum{12} \> \> \textbf{for} $i = (j+1),(j+2),\dots,(j+r-1)$,\\
\anum{13} \> \> \> Overwrite $\vct{y}^{(i)}$ by $\vct{y}^{(i)} - \vct{q}^{(j)}\ip{\vct{q}^{(j)}}{\vct{y}^{(i)}}$.\\
\anum{14} \> \> \textbf{end for}\\
\anum{15} \> \textbf{end while}\\
\anum{16} \> $\mtx{Q} = \mtx{Q}^{(j)}$.
\end{tabbing}
\end{minipage}}
\end{center}
\end{figure}

\lsp

\begin{remark}\rm
The calculations in Algorithm~\ref{alg:adaptive2} can be organized so that
each iteration processes a block of samples simultaneously.  This revision
can lead to dramatic improvements in speed because it allows us to exploit
higher-level linear algebra subroutines (e.g., BLAS3) or parallel processors.
Although blocking can lead to the generation of unnecessary samples, this
outcome is generally harmless, as noted in~\S\ref{sec:howmuchoversampling}.
\end{remark}

\subsection{A modified scheme for matrices whose singular values decay slowly}
\label{sec:powerscheme}

The techniques described in~\S\ref{sec:proto_revisited} and~\S\ref{sec:algorithm1}
work well for matrices whose singular values exhibit some decay,
but they may produce a poor basis when the input matrix has a flat singular spectrum
or when the input matrix is very large. In this section, we describe techniques,
originally proposed in~\cite{Gu-personal,tygert_szlam}, for improving the accuracy
of randomized algorithms in these situations.
Related earlier work includes~\cite{roweis} and the literature on
classical orthogonal iteration methods \cite[p.~332]{golub}.

The intuition behind these techniques is that
the singular vectors associated with small singular values
interfere with the calculation, so we reduce their weight relative to
the dominant singular vectors by taking powers of the matrix to be analyzed.
More precisely, we wish to apply the randomized sampling scheme
to the matrix $\mtx{B} = (\mtx{A}\mtx{A}^\adj)^q \mtx{A}$, where $q$ is a
small integer.  The matrix $\mtx{B}$ has the same singular vectors as
the input matrix $\mtx{A}$, but its singular values decay much more quickly:
\begin{equation}
\label{eq:svds_of_B}
\sigma_j( \mtx{B} ) = \sigma_j( \mtx{A} )^{2q+1},
\quad j = 1, 2, 3, \dots.
\end{equation}
We modify Algorithm \ref{alg:basic} by replacing the formula $\mtx{Y} = \mtx{A}\mtx{\Omega}$
in Step 2 by the formula
$\mtx{Y} = \mtx{B}\mtx{\Omega} = \bigl(\mtx{A}\mtx{A}^{*}\bigr)^{q}\mtx{A}\mtx{\Omega}$,
and we obtain Algorithm~\ref{alg:poweriteration}.

\begin{figure}
\begin{center}
\fbox{
\begin{minipage}{.9\textwidth}
\begin{center}
\textsc{Algorithm \ref{alg:poweriteration}: Randomized Power Iteration}
\end{center}

\lsp

\textit{Given an $m\times n$ matrix $\mtx{A}$ and integers $\ell$ and $q$,
this algorithm computes an $m \times \ell$ orthonormal matrix $\mtx{Q}$
whose range approximates the range of $\mtx{A}$.}

\lsp

\begin{tabbing}
\hspace{5mm} \= \hspace{5mm} \= \hspace{5mm} \= \hspace{5mm} \= \kill
\anum{1} \>Draw an $n\times \ell$ Gaussian random matrix $\mtx{\Omega}$.\\
\anum{2} \>Form the $m\times \ell$ matrix $\mtx{Y} = (\mtx{A}\mtx{A}^{\adj})^{q}\mtx{A}\mtx{\Omega}$ via alternating application\\
         \>of $\mtx{A}$ and $\mtx{A}^{\adj}$.\\
\anum{3} \>Construct an $m \times \ell$ matrix $\mtx{Q}$ whose columns form an orthonormal\\
         \> basis for the range of $\mtx{Y}$, e.g., via the QR factorization $\mtx{Y} = \mtx{Q}\mtx{R}$.
\end{tabbing}

\lsp

{\bf Note:} This procedure is vulnerable to round-off errors; see Remark \ref{remark:roundoff_in_powerscheme}.
The recommended implementation appears as Algorithm~\ref{alg:subspaceiteration}.
\end{minipage}}
\end{center}
\end{figure}

Algorithm~\ref{alg:poweriteration} requires $2q+1$ times as many
matrix--vector multiplies as Algorithm~\ref{alg:basic}, but is far
more accurate in situations where the singular values of $\mtx{A}$
decay slowly.  A good
heuristic is that when the original scheme produces a basis whose
approximation error is within a factor $C$ of the optimum, the power
scheme produces an approximation error within $C^{1/(2q+1)}$ of the
optimum.  In other words, the power iteration drives the approximation
gap to one exponentially fast.  See Theorem~\ref{thm:power-method}
and~\S\ref{sec:avg-power-method} for the details.

Algorithm~\ref{alg:poweriteration} targets the fixed-rank problem.  To address the fixed-precision
problem, we can incorporate the error estimators described in~\S\ref{sec:aposteriori}
to obtain an adaptive scheme analogous with Algorithm~\ref{alg:adaptive2}.
In situations where it is critical to achieve near-optimal approximation errors,
one can increase the oversampling beyond our standard recommendation
$\ell = k + 5$ all the way to $\ell = 2k$ without changing the
scaling of the asymptotic computational cost.  
A supporting analysis appears
in Corollary~\ref{cor:power-method-spec-gauss}.

\lsp

\begin{remark} \rm
\label{remark:roundoff_in_powerscheme}
Unfortunately, when Algorithm \ref{alg:poweriteration} is executed in floating point arithmetic,
rounding errors will extinguish all information pertaining to singular modes associated with
singular values that are small compared with $\norm{\mtx{A}}$. (Roughly, if machine
precision is $\mu$, then all information associated with singular values smaller than
$\mu^{1/(2q+1)} \norm{\mtx{A}}$ is lost.) This problem can easily be remedied
by orthonormalizing
the columns of the sample matrix between each application of $\mtx{A}$ and $\mtx{A}^{\adj}$.
The resulting scheme, summarized as Algorithm~\ref{alg:subspaceiteration}, is algebraically
equivalent to Algorithm~\ref{alg:poweriteration} when executed in exact arithmetic~\cite{stewart1969,Szlam10}.
We recommend Algorithm~\ref{alg:subspaceiteration} because its computational costs are similar to
Algorithm~\ref{alg:poweriteration}, even though the former is substantially more accurate in
floating-point arithmetic.

\end{remark}

\lsp

\begin{figure}
\begin{center}
\fbox{
\begin{minipage}{.9\textwidth}
\begin{center}
\textsc{Algorithm \ref{alg:subspaceiteration}: Randomized Subspace Iteration}
\end{center}

\lsp

\textit{Given an $m\times n$ matrix $\mtx{A}$ and integers $\ell$ and $q$,
this algorithm computes an $m \times \ell$ orthonormal matrix $\mtx{Q}$
whose range approximates the range of $\mtx{A}$.}

\lsp

\begin{tabbing}
\hspace{5mm} \= \hspace{5mm} \= \hspace{5mm} \= \hspace{5mm} \= \kill
\anum{1} \> Draw an $n\times \ell$ standard Gaussian matrix $\mtx{\Omega}$.\\
\anum{2} \> Form $\mtx{Y}_{0} = \mtx{A}\mtx{\Omega}$ and compute its QR factorization $\mtx{Y}_{0} = \mtx{Q}_{0}\mtx{R}_{0}$.\\
\anum{3} \> \textbf{for} $j = 1,\,2,\,\dots,\,q$\\
\anum{4} \> \> Form $\widetilde{\mtx{Y}}_{j} = \mtx{A}^{\adj}\mtx{Q}_{j-1}$ and compute its QR factorization
               $\widetilde{\mtx{Y}}_{j} = \widetilde{\mtx{Q}}_{j}\widetilde{\mtx{R}}_{j}$.\\
\anum{5} \> \> Form $\mtx{Y}_{j} = \mtx{A}\widetilde{\mtx{Q}}_{j}$ and compute its QR factorization
               $\mtx{Y}_{j} = \mtx{Q}_{j}\mtx{R}_{j}$.\\
\anum{6} \> \textbf{end}\\
\anum{7} \> $\mtx{Q} = \mtx{Q}_{q}$.
\end{tabbing}
\end{minipage}}
\end{center}
\end{figure}

\subsection{An accelerated technique for general dense matrices}
\label{sec:ailonchazelle}

This section describes a set of techniques that allow us to compute
an approximate rank-$\ell$ factorization of a general dense $m \times n$
matrix in roughly $\bigO(mn\log(\ell))$ flops, in contrast to the
asymptotic cost $\bigO(mn\ell)$ required by earlier methods.
We can tailor this scheme for the real or complex case, but we focus on
the conceptually simpler complex case.
These algorithms were introduced in~\cite{random2};
similar techniques were proposed in~\cite{Sar06:Improved-Approximation}.


The first step toward this accelerated technique is to observe that
the bottleneck in Algorithm~\ref{alg:basic} is the computation of
the matrix product $\mtx{A\Omega}$.  When the test matrix $\mtx{\Omega}$
is standard Gaussian, the cost of this multiplication is $\bigO(mn\ell)$,
the same as a rank-revealing QR algorithm~\cite{gu_rrqr}.  The key
idea is to use a \emph{structured} random matrix that allows us to
compute the product in $\bigO(mn \log(\ell))$ flops.

The \term{subsampled random Fourier transform}, or SRFT, is perhaps the
simplest example of a structured random matrix that meets our goals.
An SRFT is an $n \times \ell$ matrix of the form
\begin{equation}
\label{eq:def_srft}
\mtx{\Omega} = \sqrt{\frac{n}{\ell}} \, \mtx{DFR},
\end{equation}
where
\lsp
\begin{itemize}
\item   $\mtx{D}$ is an $n \times n$ diagonal matrix whose entries are
independent random variables uniformly distributed on the complex unit circle,

\item   $\mtx{F}$ is the $n \times n$ unitary discrete Fourier transform (DFT),
whose entries take the values $f_{pq} = n^{-1/2} \, \econst^{-2\pi\iunit (p-1)(q-1)/n}$ for $p, q = 1, 2, \dots, n$, and

\item   $\mtx{R}$ is an $n \times \ell$ matrix that samples $\ell$ coordinates
from $n$ uniformly at random, i.e., its $\ell$ columns are drawn randomly
without replacement from the columns of the $n \times n$ identity matrix.
\end{itemize}
\lsp

When $\mtx{\Omega}$ is defined by~\eqref{eq:def_srft}, we can compute the sample
matrix $\mtx{Y} = \mtx{A\Omega}$ using $\bigO(mn\log(\ell))$ flops via a
subsampled FFT~\cite{random2}.
Then we form the basis $\mtx{Q}$ by orthonormalizing
the columns of $\mtx{Y}$, as described in~\S\ref{sec:proto_revisited}.
This scheme appears as Algorithm~\ref{alg:fastbasic}.
The total number $T_{\rm struct}$ of flops required by this procedure is
\begin{equation}
\label{eq:cost_SRFT}
T_{\rm struct} \sim mn \log(\ell) + \ell^2 n
\end{equation}
%
Note that if $\ell$ is substantially larger than the numerical rank $k$
of the input matrix, we can perform the orthogonalization with $\bigO( k \ell n)$
flops because the columns of the sample matrix are almost linearly dependent.

The test matrix~\eqref{eq:def_srft} is just one choice among many
possibilities.  Other suggestions that appear in the literature include
subsampled Hadamard transforms, chains of Givens rotations acting on
randomly chosen coordinates, and many more.  See~\cite{liberty_diss}
and its bibliography.  Empirically, we have found that the transform
summarized in Remark~\ref{remark:random_givens} below performs very
well in a variety of environments~\cite{2008_rokhlin_leastsquares}.

At this point, it is not well understood how to quantify and compare
the behavior of structured random transforms.  One reason for this
uncertainty is that it has been difficult to analyze the amount
of oversampling that various transforms require.  Section~\ref{sec:SRFTs}
establishes that the random matrix~\eqref{eq:def_srft} can be used to
identify a near-optimal basis for a rank-$k$ matrix using
$\ell \sim (k + \log(n)) \log(k)$ samples.
%
%
In practice, the transforms \eqref{eq:def_srft} and \eqref{eq:random_Givens}
typically require no more oversampling than a Gaussian test matrix requires.
(For a numerical example, see \S\ref{sec:num_SRFT}.)
As a consequence, setting $\ell = k + 10$ or $\ell = k + 20$ is typically more than
adequate. Further research on these questions would be valuable.

\lsp

\begin{figure}
\begin{center}
\fbox{
\begin{minipage}{.9\textwidth}

\begin{center}
\textsc{Algorithm \ref{alg:fastbasic}: Fast Randomized Range Finder}
\end{center}

\lsp

\textit{Given an $m\times n$ matrix $\mtx{A}$, and an integer $\ell$,
this scheme computes an $m\times \ell$ orthonormal matrix
$\mtx{Q}$ whose range approximates the range of $\mtx{A}$.}

\lsp

\begin{tabbing}
\hspace{5mm} \= \hspace{5mm} \= \hspace{5mm} \= \hspace{5mm} \= \kill
\anum{1} \>Draw an $n\times \ell$ SRFT test matrix $\mtx{\Omega}$, as defined by \eqref{eq:def_srft}. \\
\anum{2} \>Form the $m\times \ell$ matrix $\mtx{Y} = \mtx{A}\mtx{\Omega}$ using a (subsampled) FFT.\\
\anum{3} \>Construct an $m \times \ell$ matrix $\mtx{Q}$ whose columns form an orthonormal\\
         \> basis for the range of $\mtx{Y}$, e.g., using the QR factorization $\mtx{Y} = \mtx{Q}\mtx{R}$.
\end{tabbing}
\end{minipage}}
\end{center}
\end{figure}

\lsp

\begin{remark}\rm
\label{remark:SRFT_fixedaccuracy}
The structured random matrices discussed in this section
do not adapt readily to the fixed-precision problem, where
the computational tolerance is specified, because the
samples from the range are usually computed in bulk.
Fortunately, these schemes are sufficiently inexpensive that
we can progressively increase the number of samples
computed starting with $\ell = 32$, say, and then
proceeding to $\ell = 64, 128, 256, \dots$ until we achieve
the desired tolerance.
\end{remark}

\lsp

\begin{remark}\rm
When using the SRFT~\eqref{eq:def_srft} for matrix approximation,
we have a choice whether to use a subsampled FFT or a full FFT.
The complete FFT is so inexpensive that it often pays to construct
an extended sample matrix $\mtx{Y}_{\rm large} = \mtx{ADF}$
and then generate the actual samples by drawing columns at random
from $\mtx{Y}_{\rm large}$ and rescaling as needed.
The asymptotic cost increases to $\bigO(mn\log(n))$
flops, but the full FFT is actually faster for moderate problem
sizes because the constant suppressed by the big-O notation is so
small.  Adaptive rank determination is easy because we just examine
extra samples as needed.
\end{remark}
%


\lsp

\begin{remark}
\label{remark:random_givens}\rm
Among the structured random matrices that
we have tried, one of the strongest candidates involves sequences
of random Givens rotations~\cite{2008_rokhlin_leastsquares}.
This matrix takes the form
\begin{equation}
\label{eq:random_Givens}
\mtx{\Omega} = \mtx{D}''\,\mtx{\Theta}'\,\mtx{D}'\,\mtx{\Theta}\,\mtx{D}\,\mtx{F}\,\mtx{R},
\end{equation}
where the prime symbol $'$ indicates an independent realization
of a random matrix.
The matrices $\mtx{R}$, $\mtx{F}$, and $\mtx{D}$ are defined after~\eqref{eq:def_srft}.
The matrix $\mtx{\Theta}$ is a chain of random Givens rotations:
$$
\mtx{\Theta} = \mtx{\Pi} \,\mtx{G}(1, 2; \theta_1) \,\mtx{G}(2, 3; \theta_2) \, \cdots \,
\mtx{G}(n-1, n; \theta_{n-1})
$$
where $\mtx{\Pi}$ is a random $n \times n$ permutation matrix;
where $\theta_1, \dots, \theta_{n-1}$ are independent random
variables uniformly distributed on the interval $[0, 2\pi]$;
and where $\mtx{G}(i, j; \theta)$ denotes a rotation on $\Cspace{n}$ by the angle $\theta$
in the $(i, j)$ coordinate plane~\cite[\S5.1.8]{golub}.
\end{remark}

\lsp

\begin{remark} \rm
When the singular values of the input matrix $\mtx{A}$ decay slowly,
Algorithm \ref{alg:fastbasic} may perform poorly in terms of accuracy.
When randomized sampling is used with a Gaussian random matrix, the
recourse is to take a couple of steps of a power iteration;
see~Algorithm \ref{alg:subspaceiteration}. However, it is not currently
known whether such an iterative scheme can be accelerated to $\bigO(mn\log(k))$
complexity using ``fast'' random transforms such as the SRFT.
\end{remark}

\lsp

\section{Stage B: Construction of standard factorizations}
\label{sec:otherfactorizations}

The algorithms for Stage A described in~\S\ref{sec:algorithm} produce an
orthonormal matrix $\mtx{Q}$ whose range captures the action of an input
matrix $\mtx{A}$:
\begin{equation}
\label{eq:fixed_precision2}
\norm{\mtx{A} - \mtx{Q}\mtx{Q}^{\adj}\mtx{A}} \leq \eps,
\end{equation}
where $\eps$ is a computational tolerance.  This section describes
methods for approximating standard factorizations of $\mtx{A}$ using
the information in the basis $\mtx{Q}$.

To accomplish this task, we pursue the idea from~\S\ref{sec:converting}
that any low-rank factorization $\mtx{A} \approx \mtx{CB}$ can be manipulated
to produce a standard decomposition.  When the bound \eqref{eq:fixed_precision2}
holds, the low-rank factors are simply $\mtx{C} = \mtx{Q}$ and $\mtx{B} = \mtx{Q}^\adj \mtx{A}$.
The simplest scheme (\S\ref{sec:postSVD}) computes the factor $\mtx{B}$ directly
with a matrix--matrix product to ensure a minimal error in the final approximation.
An alternative approach (\S\ref{sec:postrows}) constructs factors $\mtx{B}$ and $\mtx{C}$
without forming any matrix--matrix product. The approach of \S\ref{sec:postrows} is often
faster than the approach of \S\ref{sec:postSVD} but typically results in larger errors.
Both schemes can be streamlined for an Hermitian input matrix~(\S\ref{sec:postsym})
and a positive semidefinite input matrix~(\S\ref{sec:postpsd}).
Finally, we develop single-pass algorithms that exploit other information
generated in Stage~A to avoid revisiting the input matrix~(\S\ref{sec:onepass}).

Throughout this section, $\mtx{A}$ denotes an $m\times n$
matrix, and $\mtx{Q}$ is an $m\times k$ orthonormal matrix
that verifies~\eqref{eq:fixed_precision2}.  For purposes
of exposition, we concentrate on methods for constructing
the partial SVD.

\subsection{Factorizations based on forming $\mtx{Q}^{\adj}\mtx{A}$ directly}
\label{sec:postSVD}

The relation~\eqref{eq:fixed_precision2} implies that
$\norm{ \mtx{A} - \mtx{Q}\mtx{B}} \leq \eps$, where
$\mtx{B} = \mtx{Q}^\adj \mtx{A}$.  Once we have computed
$\mtx{B}$, we can produce any standard factorization using
the methods of~\S\ref{sec:converting}.
Algorithm~\ref{alg:Atranspose} illustrates how to build
an approximate SVD.


\begin{figure}[tb]
\begin{center}
\fbox{\begin{minipage}{.9\textwidth}
\begin{center}
\textsc{Algorithm \ref{alg:Atranspose}: Direct SVD}
\end{center}

\lsp

\textit{Given matrices $\mtx{A}$ and $\mtx{Q}$ such that \eqref{eq:fixed_precision2} holds,
this procedure computes an approximate factorization
$\mtx{A} \approx \mtx{U}\mtx{\Sigma}\mtx{V}^{\adj}$,
where $\mtx{U}$ and $\mtx{V}$ are orthonormal, and $\mtx{\Sigma}$ is a nonnegative diagonal matrix.}

\lsp

\anum{1}    Form the matrix
$
\mtx{B} = \mtx{Q}^{\adj}\mtx{A}.
$

\anum{2}    Compute an SVD of the small matrix:
$
\mtx{B} = \widetilde{\mtx{U}}\mtx{\Sigma}\mtx{V}^{\adj}.
$

\noindent
\anum{3}    Form the orthonormal matrix
$
\mtx{U} = \mtx{Q}\widetilde{\mtx{U}}.
$
\end{minipage}}
\end{center}
\end{figure}

The factors produced by Algorithm~\ref{alg:Atranspose}
satisfy
\begin{equation} \label{eq:noworse}
\norm{ \mtx{A} - \mtx{U\Sigma V}^\adj } \leq \eps.
\end{equation}
In other words, the approximation error does not degrade.

The cost of Algorithm~\ref{alg:Atranspose} is generally dominated
by the cost of the product $\mtx{Q}^\adj \mtx{A}$ in Step 1, which
takes $\bigO(kmn)$ flops for a general dense matrix.  Note that this
scheme is particularly well suited to environments where we have a
fast method for computing the matrix--vector product
$\vct{x} \mapsto \mtx{A}^\adj \vct{x}$, for example when $\mtx{A}$ is
sparse or structured.  This approach retains a strong advantage
over Krylov-subspace methods and rank-revealing QR because Step 1
can be accelerated using BLAS3, parallel processors, and so forth.
Steps 2 and 3 require $\bigO(k^2 n)$ and
$\bigO(k^2 m)$ flops respectively.


\lsp

\begin{remark} \label{rem:truncation} \rm
Algorithm~\ref{alg:Atranspose} produces an approximate SVD
with the same rank as the basis matrix $\mtx{Q}$.  When the
size of the basis exceeds the desired rank $k$ of the SVD,
it may be preferable to 
retain only the dominant $k$ singular values and singular vectors.
Equivalently, we replace the diagonal matrix $\mtx{\Sigma}$ of
computed singular values with the matrix $\mtx{\Sigma}_{(k)}$
formed by zeroing out all but the largest $k$ entries of $\mtx{\Sigma}$.
In the worst case, this truncation step can increase the approximation error
by $\sigma_{k+1}$; see~\S\ref{sec:truncation-analysis} for an analysis.
Our numerical experience suggests that this error analysis is pessimistic,
and the term $\sigma_{k+1}$ often does not appear in practice.
\end{remark}

\lsp

\subsection{Postprocessing via row extraction}
\label{sec:postrows}

Given a matrix $\mtx{Q}$ such that~\eqref{eq:fixed_precision2} holds,
we can obtain a rank-$k$ factorization
\begin{equation}
\label{eq:A_ID} \mtx{A} \approx \mtx{XB},
\end{equation}
where $\mtx{B}$ is a $k \times n$ matrix consisting of $k$ rows
extracted from $\mtx{A}$.  The approximation~\eqref{eq:A_ID} can
be produced without computing any matrix--matrix products, which
makes this approach to postprocessing very fast.  The drawback comes
because the error $\norm{\mtx{A} - \mtx{XB}}$ is usually larger than
the initial error $\norm{\mtx{A} - \mtx{QQ}^\adj \mtx{A}}$,
especially when the dimensions of $\mtx{A}$ are large.
See Remark~\ref{remark:errors} for more discussion.

To obtain the factorization~\eqref{eq:A_ID}, we simply construct
the interpolative decomposition~(\S\ref{sec:ID}) of the matrix $\mtx{Q}$:
\begin{equation}
\label{eq:ID}
\mtx{Q} = \mtx{X}\mtx{Q}_{(J,\colon)}.
\end{equation}
The index set $J$ marks $k$ rows of $\mtx{Q}$ that span the row
space of $\mtx{Q}$, and $\mtx{X}$ is an $m\times k$ matrix whose
entries are bounded in magnitude by two and contains the $k\times k$
identity as a submatrix: $\mtx{X}_{(J,\colon)} =
\Id_{k}$.  Combining \eqref{eq:ID} and~\eqref{eq:fixed_precision2},
we reach
\begin{equation}
\label{eq:saathoff}
\mtx{A} \approx \mtx{Q}\mtx{Q}^{\adj}\mtx{A} =
\mtx{X}\mtx{Q}_{(J,\colon)}\mtx{Q}^{\adj}\mtx{A}.
\end{equation}
Since $\mtx{X}_{(J,\colon)} = \Id_{k}$, equation
\eqref{eq:saathoff} implies that $\mtx{A}_{(J,\colon)} \approx
\mtx{Q}_{(J,\colon)}\mtx{Q}^{\adj}\mtx{A}$.
Therefore, \eqref{eq:A_ID} follows when we put $\mtx{B} = \mtx{A}_{(J,\colon)}$.

Provided with the factorization \eqref{eq:A_ID}, we can
obtain any standard factorization using the techniques of \S\ref{sec:converting}.
Algorithm \ref{alg:extractrows} illustrates an SVD calculation.
This procedure requires $\bigO(k^{2} (m+n))$ flops.  The following lemma
guarantees the accuracy of the computed factors.

\begin{figure}[tb]
\begin{center}
\fbox{
\begin{minipage}{.9\textwidth}
\begin{center}
\textsc{Algorithm \ref{alg:extractrows}: SVD via Row Extraction}
\end{center}

\lsp

\textit{Given matrices $\mtx{A}$ and $\mtx{Q}$ such that \eqref{eq:fixed_precision2} holds,
this procedure computes an approximate factorization
$\mtx{A} \approx \mtx{U}\mtx{\Sigma}\mtx{V}^{\adj}$,
where $\mtx{U}$ and $\mtx{V}$ are orthonormal, and $\mtx{\Sigma}$ is a nonnegative diagonal matrix.}

\lsp

\anum{1}    Compute an ID
$
\mtx{Q} = \mtx{XQ}_{(J,\colon)}.
$
(The ID is defined in \S\ref{sec:ID}.) 

\anum{2}    Extract $\mtx{A}_{(J,\colon)}$, and compute
a QR factorization
$
\mtx{A}_{(J,\colon)} = \mtx{R}^{\adj}\mtx{W}^{\adj}.
$

\anum{3}    Form the product
$
\mtx{Z} = \mtx{XR}^{\adj}.
$

\anum{4}    Compute an SVD
$
\mtx{Z} = \mtx{U\Sigma}\widetilde{\mtx{V}}^{\adj}.
$

\anum{5}    Form the orthonormal matrix
$
\mtx{V} = \mtx{W} \widetilde{\mtx{V}}.
$

\lsp

{\bf Note:} Algorithm~\ref{alg:extractrows} is faster than Algorithm~\ref{alg:Atranspose} but less accurate.\\
{\bf Note:} It is advantageous to replace the basis $\mtx{Q}$ by the sample matrix $\mtx{Y}$
produced in Stage A, cf.~Remark \ref{remark:use_Y_in rowextraction}.
\end{minipage}}
\end{center}
\end{figure}

\lsp

\begin{lemma}
\label{lem:convert_ID}
Let $\mtx{A}$ be an $m\times n$ matrix and let $\mtx{Q}$ be an
$m\times k$ matrix that satisfy \eqref{eq:fixed_precision2}.
Suppose that $\mtx{U}$, $\mtx{\Sigma}$, and $\mtx{V}$ are the matrices
constructed by Algorithm~\ref{alg:extractrows}. Then
\begin{equation}
\label{eq:IPA}
\norm{\mtx{A} - \mtx{U\Sigma V}^{\adj}} \leq
\left[1 + \sqrt{1 + 4k(n-k)} \right] \eps.
\end{equation}
\end{lemma}


\begin{proof}
The factors $\mtx{U}$, $\mtx{\Sigma}$, $\mtx{V}$ constructed by the algorithm satisfy
$$
\mtx{U\Sigma V}^{\adj} =
\mtx{U\Sigma}\widetilde{\mtx{V}}^{\adj}\mtx{W}^{\adj} =
\mtx{Z}\mtx{W}^\adj =
\mtx{XR}^{\adj}\mtx{W}^{\adj} =
\mtx{XA}_{(J,\colon)}.
$$
Define the approximation
\begin{equation}
\label{eq:desch}
\widehat{\mtx{A}} = \mtx{QQ}^{\adj}\mtx{A}.
\end{equation}
Since $\widehat{\mtx{A}} = \mtx{XQ}_{(J,\colon)}\mtx{Q}^{\adj}\mtx{A}$
and since $\mtx{X}_{(J,\colon)} = \Id_{k}$, it must be that
$\widehat{\mtx{A}}_{(J,\colon)} = \mtx{Q}_{(J,\colon)}\mtx{Q}^{\adj}\mtx{A}$.
Consequently,
$$
\widehat{\mtx{A}} = \mtx{X}\widehat{\mtx{A}}_{(J,\colon)}.
$$
We have the chain of relations
\begin{align}
\label{eq:utes}
\norm{\mtx{A} - \mtx{U\Sigma V}^{\adj}}
    &= \norm{\mtx{A} - \mtx{XA}_{(J,\colon)}} \notag \\
    &= \smnorm{}{\big(\mtx{A} - \mtx{X}\widehat{\mtx{A}}_{(J,\colon)} \big) +
        \big(\mtx{X}\widehat{\mtx{A}}_{(J,\colon)} - \mtx{XA}_{(J,\colon)}\big)} \notag \\
    &\leq \smnorm{}{\mtx{A} -
     \widehat{\mtx{A}}} + \smnorm{}{\mtx{X}\widehat{\mtx{A}}_{(J,\colon)} - \mtx{XA}_{(J,\colon)}} \notag \\
    &\leq \smnorm{}{\mtx{A} -
     \widehat{\mtx{A}}} + \norm{\mtx{X}}\smnorm{}{ \mtx{A}_{(J,\colon)} - \widehat{\mtx{A}}_{(J,\colon)}}.
\end{align}
Inequality \eqref{eq:fixed_precision2} ensures that $\smnorm{}{\mtx{A} -
\widehat{\mtx{A}}} \leq \eps$.  Since $\mtx{A}_{(J,\colon)} -
\widehat{\mtx{A}}_{(J,\colon)}$ is a submatrix of $\mtx{A} -
\widehat{\mtx{A}}$, we must also have $\smnorm{}{\mtx{A}_{(J,\colon)} -
\widehat{\mtx{A}}_{(J,\colon)}} \leq \eps$.
Thus, \eqref{eq:utes} reduces to
\begin{equation}
\label{eq:utes2}
\norm{\mtx{A} - \mtx{U}\mtx{\Sigma}\mtx{V}^{\adj}} \leq \left(1 + \norm{\mtx{X}}\right) \eps.
\end{equation}
The bound \eqref{eq:IPA} follows from \eqref{eq:utes2} after we observe that
$\mtx{X}$ contains a $k\times k$ identity matrix and that the entries of the
remaining $(n-k) \times k$ submatrix are bounded in magnitude by two.
\end{proof}

\lsp

\begin{remark}\rm
\label{remark:use_Y_in rowextraction}
To maintain a unified presentation, we have formulated all the
postprocessing techniques so they take an orthonormal matrix
$\mtx{Q}$ as input.  Recall that, in Stage A of our framework,
we construct the matrix $\mtx{Q}$ by orthonormalizing the columns
of the sample matrix $\mtx{Y}$.  With finite-precision
arithmetic, it is preferable to adapt Algorithm~\ref{alg:extractrows}
to start directly from the sample matrix $\mtx{Y}$.
To be precise, we modify Step 1 to compute $\mtx{X}$ and $J$
so that $\mtx{Y} = \mtx{XY}_{(J,:)}$.  This revision is
recommended even when $\mtx{Q}$ is available from the
adaptive rank determination of Algorithm~\ref{alg:adaptive2}.
%
\end{remark}

\lsp

\begin{remark}\rm
\label{remark:errors}
As the inequality~\eqref{eq:IPA} suggests, the factorization
produced by Algorithm~\ref{alg:extractrows} is potentially
less accurate than the basis that it uses as input.  This
loss of accuracy is problematic when $\eps$ is not so small
or when $kn$ is large.
In such cases, we recommend Algorithm~\ref{alg:Atranspose}
over Algorithm~\ref{alg:extractrows}; the former is more costly, but
it does not amplify  the error, as shown in~\eqref{eq:noworse}.
\end{remark}


\lsp

\subsection{Postprocessing an Hermitian matrix}
\label{sec:postsym}

When $\mtx{A}$ is Hermitian, the postprocessing becomes particularly elegant.
In this case, the columns of $\mtx{Q}$ form a good basis for both the
column space \emph{and} the row space of $\mtx{A}$ so that we have
$\mtx{A} \approx \mtx{QQ}^\adj \mtx{A} \mtx{QQ}^\adj$.
More precisely, when~\eqref{eq:fixed_precision2} is in force,
we have
\begin{multline}
\label{eq:twoepsilon}
\norm{\mtx{A} - \mtx{QQ}^{\adj}\mtx{A}\mtx{QQ}^{\adj}} =
\norm{\mtx{A} - \mtx{QQ}^{\adj}\mtx{A} + \mtx{QQ}^{\adj}\mtx{A} -
\mtx{QQ}^{\adj}\mtx{A}\mtx{QQ}^{\adj}} \\
\leq \norm{\mtx{A} - \mtx{QQ}^{\adj}\mtx{A}} +
\norm{\mtx{QQ}^{\adj}
\bigl(\mtx{A} - \mtx{AQQ}^{\adj} \bigr)} \leq 2\eps.
\end{multline}
The last inequality relies on the facts that $\norm{\mtx{QQ}^\adj} = 1$ and
that
$$
\norm{ \mtx{A} - \mtx{AQQ}^\adj } = \smnorm{}{ (\mtx{A} - \mtx{AQQ}^\adj)^\adj }
    = \norm{ \mtx{A} - \mtx{QQ}^\adj \mtx{A} }.
$$
Since $\mtx{A} \approx \mtx{Q}\bigl(\mtx{Q}^\adj\mtx{A} \mtx{Q}\bigr)\mtx{Q}^\adj$
is a low-rank approximation
of $\mtx{A}$, we can form any standard factorization using the techniques from~\S\ref{sec:converting}.

For Hermitian $\mtx{A}$, it is more common to compute an eigenvalue decomposition
than an SVD.  We can accomplish this goal using Algorithm~\ref{alg:hermeig}, which adapts the scheme
from~\S\ref{sec:postSVD}.
This procedure delivers a factorization that satisfies the error bound
$\norm{ \mtx{A} - \mtx{U\Lambda U}^\adj } \leq 2\eps$.
The calculation requires $\bigO( k n^2 )$ flops.


We can also pursue the row extraction approach from~\S\ref{sec:postrows}, which is faster but less accurate.
See Algorithm~\ref{alg:hermeigrows} for the details.  The total cost is $\bigO( k^2 n )$ flops.

\begin{figure}[tb]
\begin{center}
\fbox{\begin{minipage}{.9\textwidth}
\begin{center}
\textsc{Algorithm \ref{alg:hermeig}: Direct Eigenvalue Decomposition}
\end{center}

\lsp

\textit{Given an Hermitian matrix $\mtx{A}$ and a basis $\mtx{Q}$ such that \eqref{eq:fixed_precision2} holds,
this procedure computes an approximate eigenvalue decomposition
$\mtx{A} \approx \mtx{U}\mtx{\Lambda}\mtx{U}^{\adj}$,
where $\mtx{U}$ is orthonormal, and $\mtx{\Lambda}$ is a real diagonal matrix.}

\lsp

\anum{1}    Form the small matrix
$
\mtx{B} = \mtx{Q}^{\adj}\mtx{AQ}.
$

\anum{2}    Compute an eigenvalue decomposition
$
\mtx{B} = \mtx{V}\mtx{\Lambda}\mtx{V}^{\adj}.
$

\noindent
\anum{3}    Form the orthonormal matrix
$
\mtx{U} = \mtx{QV}.
$
\end{minipage}}
\end{center}
\end{figure}

\begin{figure}[tb]
\begin{center}
\fbox{
\begin{minipage}{.9\textwidth}
\begin{center}
\textsc{Algorithm \ref{alg:hermeigrows}: Eigenvalue Decomposition via Row Extraction}
\end{center}

\lsp

\textit{Given an Hermitian matrix $\mtx{A}$ and a basis $\mtx{Q}$ such that \eqref{eq:fixed_precision2} holds, this procedure computes an approximate eigenvalue decomposition
$\mtx{A} \approx \mtx{U}\mtx{\Lambda}\mtx{U}^{\adj}$,
where $\mtx{U}$ is orthonormal, and $\mtx{\Lambda}$ is a real diagonal matrix.}

\lsp

\anum{1}    Compute an ID
$
\mtx{Q} = \mtx{XQ}_{(J,\colon)}.
$

\anum{2}    Perform a QR factorization
$
\mtx{X} = \mtx{VR}.
$

\anum{3}	Form the product
$
\mtx{Z} = \mtx{R} \mtx{A}_{(J, J)} \mtx{R}^\adj.
$

\anum{4}	Compute an eigenvalue decomposition
$
\mtx{Z} = \mtx{W \Lambda W}^\adj.
$

\anum{5}	Form the orthonormal matrix
$
\mtx{U} = \mtx{VW}.
$

\lsp

{\bf Note:} Algorithm~\ref{alg:hermeigrows} is faster than Algorithm~\ref{alg:hermeig} but less accurate.\\
{\bf Note:} It is advantageous to replace the basis $\mtx{Q}$ by the sample matrix $\mtx{Y}$
produced in Stage A, cf.~Remark \ref{remark:use_Y_in rowextraction}.
\end{minipage}}
\end{center}
\end{figure}

\subsection{Postprocessing a positive semidefinite matrix} \label{sec:postpsd}

When the input matrix $\mtx{A}$ is positive semidefinite, the
\term{Nystr\"om method} can be used to improve
the quality of standard factorizations at almost no additional cost;
see~\cite{DM05:Nystrom-Method} and its bibliography.
To describe the main idea, we first recall that the direct method presented
in~\S\ref{sec:postsym} manipulates the approximate rank-$k$ factorization
\begin{equation}
\label{eq:psd_standard}
\mtx{A} \approx \mtx{Q}\,\bigl(\mtx{Q}^{\adj}\mtx{A}\mtx{Q}\bigr)\,\mtx{Q}^{\adj}.
\end{equation}
In contrast, the Nystr\"{o}m scheme builds a more sophisticated rank-$k$ approximation, namely
\begin{multline}
\label{eq:psd_nystrom}
\mtx{A}
\approx (\mtx{A}\mtx{Q})\,\bigl(\mtx{Q}^{\adj}\mtx{A}\mtx{Q}\bigr)^{-1}\,(\mtx{A}\mtx{Q})^{\adj}\\
= \left[(\mtx{A}\mtx{Q})\,\bigl(\mtx{Q}^{\adj}\mtx{A}\mtx{Q}\bigr)^{-1/2}\right]\,
  \left[(\mtx{A}\mtx{Q})\,\bigl(\mtx{Q}^{\adj}\mtx{A}\mtx{Q}\bigr)^{-1/2}\right]^{\adj}
= \mtx{FF}^\adj,
\end{multline}
where $\mtx{F}$ is an approximate Cholesky factor of $\mtx{A}$ with dimension $n\times k$.
To compute the factor $\mtx{F}$ numerically, first form the matrices
$\mtx{B}_1 = \mtx{AQ}$ and $\mtx{B}_2 = \mtx{Q}^\adj \mtx{B}_1$.
Then decompose the psd matrix $\mtx{B}_2 = \mtx{C}^\adj \mtx{C}$
into its Cholesky factors.  Finally compute the factor $\mtx{F} = \mtx{B}_1 \mtx{C}^{-1}$
by performing a triangular solve. The low-rank factorization~\eqref{eq:psd_nystrom} can be
converted to a standard decomposition using the techniques from~\S\ref{sec:converting}.


The literature contains an explicit expression~\cite[Lem.~4]{DM05:Nystrom-Method} for the approximation error in~\eqref{eq:psd_nystrom}.  This result implies that, in the spectral norm, the Nystr{\"o}m approximation error never exceeds $\norm{ \mtx{A} - \mtx{QQ}^\adj \mtx{A} }$, and it is often substantially smaller.  We omit a detailed discussion.

%

For an example of the Nystr{\"o}m technique, consider Algorithm~\ref{alg:nystrom}, which computes an approximate eigenvalue decomposition of a positive semidefinite matrix.  This method should be compared with the scheme for Hermitian matrices, Algorithm~\ref{alg:hermeig}.
In both cases, the dominant cost occurs when we form $\mtx{AQ}$, so the two procedures have roughly the same running time.  On the other hand, Algorithm~\ref{alg:nystrom} is typically much more accurate than Algorithm~\ref{alg:hermeig}.  In a sense, we are exploiting the fact that $\mtx{A}$ is positive semidefinite to take one step of subspace iteration (Algorithm~\ref{alg:subspaceiteration}) for free.

\begin{figure}[tb]
\begin{center}
\fbox{
\begin{minipage}{.9\textwidth}
\begin{center}
\textsc{Algorithm \ref{alg:nystrom}: Eigenvalue Decomposition via Nystr{\"o}m Method}
\end{center}

\lsp

\textit{Given a positive semidefinite matrix $\mtx{A}$ and a basis $\mtx{Q}$ such that \eqref{eq:fixed_precision2} holds, this procedure computes an approximate eigenvalue decomposition
$\mtx{A} \approx \mtx{U}\mtx{\Lambda}\mtx{U}^{\adj}$,
where $\mtx{U}$ is orthonormal, and $\mtx{\Lambda}$ is nonnegative and diagonal.}

\lsp

\anum{1}    Form the matrices $\mtx{B}_{1} = \mtx{A}\mtx{Q}$ and $\mtx{B}_{2} = \mtx{Q}^{\adj}\mtx{B}_{1}$.

\anum{2}    Perform a Cholesky factorization $\mtx{B}_{2} = \mtx{C}^\adj \mtx{C}$.

\anum{3}	Form $\mtx{F} = \mtx{B}_{1}\mtx{C}^{-1}$ using a triangular solve.

\anum{4}	Compute an SVD $\mtx{F} = \mtx{U}\mtx{\Sigma}\mtx{V}^\adj$ and set $\mtx{\Lambda} = \mtx{\Sigma}^{2}$.


\vspace{.5pc}

\end{minipage}}
\end{center}
\end{figure}

\subsection{Single-pass algorithms}
\label{sec:onepass}

The techniques described in~\S\S\ref{sec:postSVD}--\ref{sec:postpsd}
all require us to revisit the input matrix.  This may not be
feasible in environments where the matrix is too large to be
stored.  In this section, we develop a method that requires just
one pass over the matrix to construct not only an approximate
basis but also a complete factorization.  Similar techniques
appear in~\cite{random2} and~\cite{2009_clarkson_woodruff}.

For motivation, we begin with the case where $\mtx{A}$ is Hermitian.
Let us recall the proto-algorithm from~\S\ref{sec:proto-algorithm}:
Draw a random test matrix $\mtx{\Omega}$; form the sample matrix
$\mtx{Y} = \mtx{A\Omega}$; then construct a basis $\mtx{Q}$ for the
range of $\mtx{Y}$.  It turns out that the matrices $\mtx{\Omega}$, $\mtx{Y}$,
and $\mtx{Q}$ contain all the information we need to approximate
$\mtx{A}$.

To see why, define the (currently unknown) matrix $\mtx{B}$
via $\mtx{B} = \mtx{Q}^\adj \mtx{AQ}$.
Postmultiplying the definition by $\mtx{Q}^\adj \mtx{\Omega}$, we obtain the identity
$\mtx{BQ}^\adj \mtx{\Omega} = \mtx{Q}^\adj \mtx{AQQ}^\adj \mtx{\Omega}$.
The relationships $\mtx{AQQ}^\adj \approx \mtx{A}$ and $\mtx{A\Omega} = \mtx{Y}$
show that $\mtx{B}$ must satisfy
\begin{equation}
\label{eq:tuss}
\mtx{BQ}^{\adj}\mtx{\Omega} \approx \mtx{Q}^{\adj}\mtx{Y}.
\end{equation}
All three matrices $\mtx{\Omega}$, $\mtx{Y}$, and $\mtx{Q}$
are available, so we can solve~\eqref{eq:tuss} to obtain the matrix $\mtx{B}$.
Then the low-rank factorization $\mtx{A} \approx \mtx{QBQ}^\adj$ can
be converted to an eigenvalue decomposition via familiar techniques.
The entire procedure requires $\bigO(k^2 n)$ flops, and it is summarized
as Algorithm~\ref{alg:postsym}.

\lsp

\begin{figure}[htb]
\begin{center}
\fbox{
\begin{minipage}{.9\textwidth}
\begin{center}
\textsc{Algorithm \ref{alg:postsym}: Eigenvalue Decomposition in One Pass}
\end{center}

\lsp

\textit{Given an Hermitian matrix $\mtx{A}$,
a random test matrix $\mtx{\Omega}$,
a sample matrix $\mtx{Y} = \mtx{A}\mtx{\Omega}$,
and an orthonormal matrix $\mtx{Q}$ that verifies
\eqref{eq:fixed_precision2} and $\mtx{Y} = \mtx{Q}\mtx{Q}^{\adj}\mtx{Y}$,
this algorithm computes an approximate eigenvalue decomposition
$\mtx{A} \approx \mtx{U\Lambda U}^{\adj}$.}

\lsp

\begin{tabbing}
\hspace{5mm} \= \hspace{5mm} \= \hspace{5mm} \= \hspace{5mm} \= \kill

\anum{1}  \>  Use a standard least-squares solver to find an Hermitian matrix $\mtx{B}_{\rm approx}$\\
\>that approximately satisfies the equation
$\mtx{B}_{\rm approx} (\mtx{Q}^{\adj}\mtx{\Omega}) \approx \mtx{Q}^{\adj}\mtx{Y}$.\\

\anum{2} \>   Compute the eigenvalue decomposition
$
\mtx{B}_{\rm approx} = \mtx{V\Lambda V}^{\adj}.
$\\

\anum{3} \>   Form the product
$
\mtx{U} = \mtx{QV}.
$
\end{tabbing}
\end{minipage}}
\end{center}
\end{figure}

When $\mtx{A}$ is not Hermitian, it is still possible to devise
single-pass algorithms, but we must modify the initial Stage A
of the approximation framework to simultaneously construct bases
for the ranges of $\mtx{A}$ and $\mtx{A}^{\adj}$:
\lsp
\begin{remunerate}
\item   Generate random matrices $\mtx{\Omega}$ and $\widetilde{\mtx{\Omega}}$.
\item   Compute $\mtx{Y} = \mtx{A\Omega}$ and
$\widetilde{\mtx{Y}} = \mtx{A}^{\adj}\widetilde{\mtx{\Omega}}$ in a single pass over $\mtx{A}$.
\item   Compute QR factorizations $\mtx{Y} = \mtx{QR}$ and
$\widetilde{\mtx{Y}} = \widetilde{\mtx{Q}}\widetilde{\mtx{R}}$.
\end{remunerate}
\lsp
This procedure results in matrices $\mtx{Q}$ and $\widetilde{\mtx{Q}}$ such that
$\mtx{A} \approx \mtx{QQ}^{\adj}\mtx{A}\widetilde{\mtx{Q}}\widetilde{\mtx{Q}}^{\adj}$.
The reduced matrix we must approximate is
$\mtx{B} = \mtx{Q}^{\adj}\mtx{A}\widetilde{\mtx{Q}}$.
In analogy with~\eqref{eq:tuss}, we find that
\begin{equation}
\label{eq:milk1}
\mtx{Q}^{\adj}\mtx{Y} =
\mtx{Q}^{\adj}\mtx{A\Omega} \approx
\mtx{Q}^{\adj}\mtx{A}\widetilde{\mtx{Q}}\widetilde{\mtx{Q}}^{\adj}\mtx{\Omega} =
\mtx{B}\widetilde{\mtx{Q}}^{\adj}\mtx{\Omega}.
\end{equation}
An analogous calculation shows that $\mtx{B}$ should also satisfy
\begin{equation}
\label{eq:milk2}
\widetilde{\mtx{Q}}^{\adj}\widetilde{\mtx{Y}} \approx \mtx{B}^{\adj}\mtx{Q}^{\adj}\widetilde{\mtx{\Omega}}.
\end{equation}
Now, the reduced matrix $\mtx{B}_{\rm approx}$ can be determined by finding a
minimum-residual solution to the system of relations~\eqref{eq:milk1} and~\eqref{eq:milk2}.

\lsp

\begin{remark} \rm
The single-pass approaches described in this
section can degrade the approximation error
in the final decomposition significantly.
To explain the issue, we focus on
the Hermitian case.  It turns out that the coefficient matrix
$\mtx{Q}^{\adj}\mtx{\Omega}$ in the linear
system \eqref{eq:tuss} is usually ill-conditioned.
In a worst-case scenario, the error
$\norm{\mtx{A} - \mtx{U}\mtx{\Lambda}\mtx{U}^{\adj}}$
in the factorization produced by Algorithm \ref{alg:postsym}
could be larger than the error resulting from the two-pass
method of Section \ref{sec:postsym} by a
factor of $1/\tau_{\rm min}$, where $\tau_{\rm min}$ is the
minimal singular value of the matrix $\mtx{Q}^{\adj}\mtx{\Omega}$.

The situation can be improved by oversampling.  Suppose that
we seek a rank-$k$ approximate eigenvalue decomposition.
Pick a small oversampling parameter $p$.  Draw
an $n\times (k+p)$ random matrix $\mtx{\Omega}$,
and form the sample matrix $\mtx{Y} = \mtx{A}\mtx{\Omega}$.
Let $\mtx{Q}$ denote the $n\times k$ matrix formed by
the $k$ leading left singular vectors
of $\mtx{Y}$. Now, the linear system~\eqref{eq:tuss} has a coefficient matrix $\mtx{Q}^{\adj}\mtx{\Omega}$
of size $k \times (k+p)$, so it is overdetermined.  An
approximate solution of this system yields a $k \times k$
matrix $\mtx{B}$.
\end{remark}

\lsp

\section{Computational costs}
\label{sec:costs}
So far, we have postponed a detailed discussion of the computational
cost of randomized matrix approximation algorithms because it is necessary
to account for both the first stage, where we compute an approximate basis for
the range~(\S\ref{sec:algorithm}), and the second stage, where we postprocess
the basis to complete the factorization~(\S\ref{sec:otherfactorizations}).
We are now prepared to compare the cost of the two-stage scheme with the cost of
traditional techniques.

Choosing an appropriate algorithm, whether classical or randomized,
requires us to consider the properties of the input matrix.
To draw a nuanced picture, we discuss three
representative computational environments
in~\S\ref{sec:generalmat}--\ref{sec:outofcore}.
We close with some comments on parallel implementations in~\S\ref{sec:parallel}.

For concreteness, we focus on the problem of computing an approximate
SVD of an $m \times n$ matrix $\mtx{A}$ with numerical rank $k$.
The costs for other factorizations are similar.


\subsection{General matrices that fit in core memory}
\label{sec:generalmat}

Suppose that $\mtx{A}$ is a general matrix presented as an array
of numbers that fits in core memory.  In this case, the appropriate
method for Stage A is to use a structured random matrix~(\S\ref{sec:ailonchazelle}),
which allows us to find a basis that captures the action of the matrix
using $\bigO(mn\log(k) + k^2 m)$ flops.  For Stage B, we apply
the row-extraction technique~(\S\ref{sec:postrows}), which costs
an additional $\bigO(k^2(m+n))$ flops.  The total number of operations
$T_{\rm random}$ for this approach satisfies
$$
T_{\rm random} \sim mn \log(k) + k^2 (m+n).
$$
As a rule of thumb, the approximation error of this procedure satisfies
\begin{equation} \label{eq:fastalg-err-heur}
\norm{ \mtx{A} - \mtx{U\Sigma V}^\adj }
    \lesssim n \cdot \sigma_{k+1},
\end{equation}
where $\sigma_{k+1}$ is the $(k+1)$th singular value of $\mtx{A}$.
%
%
The estimate \eqref{eq:fastalg-err-heur}, which follows from Theorem~\ref{thm:SRFT}
and Lemma~\ref{lem:convert_ID}, reflects the worst-case scenario;
actual errors are usually smaller.

This algorithm should be compared with modern deterministic techniques,
such as rank-revealing QR followed by postprocessing (\S\ref{sec:partial_decomp})
which typically require
$$
T_{\rm RRQR} \sim kmn
$$
operations to achieve a comparable error.

In this setting, the randomized algorithm can be several times faster than classical
techniques even for problems of moderate size, say $m, n \sim 10^3$ and $k \sim 10^2$.
See \S\ref{sec:num_SRFT} for numerical evidence.

\lsp

\begin{remark} \rm
In case row extraction is impractical, there is an alternative
$\bigO(mn\log(k))$ technique described in~\cite[\S5.2]{random2}.
When the error~\eqref{eq:fastalg-err-heur} is unacceptably large,
we can use the direct method (\S\ref{sec:postSVD}) for Stage B,
which brings the total cost to $\bigO( kmn )$ flops.
\end{remark}

\subsection{Matrices for which matrix--vector products can be rapidly evaluated}
\label{sec:fastmatvec}

In many problems in data mining and scientific computing, the cost
$T_{\rm mult}$ of performing the matrix--vector multiplication
$\vct{x} \mapsto \mtx{A}\vct{x}$ is substantially smaller
than the nominal cost $\bigO(mn)$ for the dense case.
It is not uncommon that $\bigO(m + n)$ flops suffice.
Standard examples include (i) very sparse matrices;
(ii) structured matrices, such as T{\"o}plitz operators, that can be applied
using the FFT or other means; and
(iii) matrices that arise from physical problems, such as
discretized integral operators, that can be applied via, e.g.,
the fast multipole method \cite{rokhlin1997}.

Suppose that both $\mtx{A}$ and $\mtx{A}^\adj$ admit fast multiplies.
The appropriate randomized approach for this scenario completes
Stage A using Algorithm~\ref{alg:basic} with $p$ constant
(for the fixed-rank problem)
or Algorithm~\ref{alg:adaptive2} (for the fixed-precision problem)
at a cost of  $(k+p) \, T_{\rm mult} + \bigO(k^2 m)$ flops.
For Stage B, we invoke Algorithm~\ref{alg:Atranspose}, which
requires $(k+p) \, T_{\rm mult} + \bigO(k^2 (m + n))$ flops.
The total cost $T_{\rm sparse}$ satisfies
\begin{equation}
\label{eq:sparsealg-cost}
T_{\rm sparse} = 2\,(k + p)\, T_{\rm mult} + \bigO(k^2 (m + n)).
\end{equation}
As a rule of thumb, the approximation error of this procedure satisfies
\begin{equation}
\label{eq:sparsealg-error}
\norm{ \mtx{A} - \mtx{U\Sigma V}^\adj } \lesssim \sqrt{kn} \cdot \sigma_{k+1}.
\end{equation}
The estimate \eqref{eq:sparsealg-error} follows from Corollary~\ref{cor:tail-spec-error-gauss}
and the discussion in~\S\ref{sec:postSVD}.  Actual errors are usually smaller.

When the singular spectrum of $\mtx{A}$ decays slowly,
we can incorporate $q$ iterations of the power method
(Algorithm~\ref{alg:poweriteration}) to obtain superior
solutions to the fixed-rank problem.
The computational cost increases to, cf.~\eqref{eq:sparsealg-cost},
\begin{equation}
\label{eq:sparsealg-cost2}
T_{\rm sparse} = (2q+2)\,(k + p) \, T_{\rm mult} + \bigO(k^2 (m + n)),
\end{equation}
while the error \eqref{eq:sparsealg-error} improves to
\begin{equation}
\label{eq:sparsealg-error2}
\norm{ \mtx{A} - \mtx{U\Sigma V}^\adj } \lesssim (kn)^{1/2(2q+1)} \cdot \sigma_{k+1}.
\end{equation}
The estimate \eqref{eq:sparsealg-error2}
takes into account the discussion in~\S\ref{sec:avg-power-method}.
The power scheme can also be adapted for the fixed-precision
problem (\S\ref{sec:powerscheme}).

In this setting, the classical prescription for obtaining a partial SVD
is some variation of a Krylov-subspace method; see \S\ref{sec:krylov}.
These methods exhibit great diversity, so it is hard to specify
a ``typical'' computational cost.
To a first approximation, it is fair to say
that in order to obtain an approximate SVD of rank $k$, the cost of
a numerically stable implementation of a Krylov method is no less than
the cost \eqref{eq:sparsealg-cost} with $p$ set to zero. At this price,
the Krylov method often obtains better accuracy than the basic
randomized method obtained by combining Algorithms \ref{alg:basic} and
\ref{alg:Atranspose}, especially for matrices whose singular values decay
slowly. On the other hand, the randomized schemes are inherently more robust
and allow much more freedom in organizing the computation to suit a particular
application or a particular hardware architecture. The latter point is in
practice of crucial importance because it is usually much faster to apply a
matrix to $k$ vectors simultaneously than it is to execute $k$ matrix--vector
multiplications consecutively.  In practice, blocking and parallelism can lead
to enough gain that a few steps of the power method (Algorithm~\ref{alg:poweriteration})
can be performed more quickly than $k$ steps of a Krylov method.

\lsp

\begin{remark} \rm
Any comparison between randomized sampling schemes and Krylov variants
becomes complicated because of the fact that ``basic'' Krylov schemes such
as Lanczos \cite[p.~473]{golub} or Arnoldi \cite[p.~499]{golub}
are inherently unstable. To obtain numerical robustness, we must incorporate
sophisticated modifications such as restarts, reorthogonalization procedures, etc.
Constructing a high-quality implementation is sufficiently hard that the authors
of a popular book on ``numerical recipes''
qualify their treatment of spectral computations as follows
\cite[p.~567]{2007_numerical_recipes}:

\lsp

\begin{quotation}
\noindent
You have probably gathered by now that the solution of eigensystems is a fairly complicated business.
It is. It is one of the few subjects covered in this book for which we do \emph{not} recommend that you avoid
canned routines. On the contrary, the purpose of this chapter is precisely to give you some appreciation of what
is going on inside such canned routines, so that you can make intelligent choices about using them, and intelligent
diagnoses when something goes wrong.
\end{quotation}

\lsp

\noindent
Randomized sampling does not eliminate the difficulties referred to in this quotation;
however it reduces the task of computing a \emph{partial} spectral decomposition
of a very large matrix to the task of computing a \emph{full} decomposition of
a small dense matrix.  (For example,~in Algorithm \ref{alg:Atranspose}, the input matrix
$\mtx{A}$ is large and $\mtx{B}$ is small.) The latter task is much better understood
and is eminently suitable for using canned routines.  Random sampling schemes
interact with the large matrix only through matrix--matrix products, which can
easily be implemented by a user in a manner appropriate to the application
and to the available hardware.


The comparison is further complicated by the fact that there is significant
overlap between the two sets of ideas. Algorithm~\ref{alg:poweriteration} is
conceptually similar to a ``block Lanczos method'' \cite[p.~485]{golub} with a random starting matrix.
Indeed, we believe that there are significant opportunities for cross-fertilization
in this area.  Hybrid schemes that combine the best ideas from both
fields may perform very well.
\end{remark}

\subsection{General matrices stored in slow memory or streamed}
\label{sec:outofcore}

The traditional metric for numerical algorithms is the number
of floating-point operations they require.  When the data does
not fit in fast memory, however, the computational time is
often dominated by the cost of memory access.  In this setting,
a more appropriate measure of algorithmic performance is
\term{pass-efficiency}, which counts how many times the data
needs to be cycled through fast memory.  Flop counts become
largely irrelevant.

All the classical matrix factorization techniques that we
discuss in~\S\ref{sec:standard_factorizations}---including
dense SVD, rank-revealing QR, Krylov methods, and so
forth---require at least $k$ passes over the the matrix, which is
prohibitively expensive for huge data matrices.
A desire to reduce the pass count of matrix approximation
algorithms served as one of the early motivations for
developing randomized schemes~\cite{papadimitriou,kannan_vempala,drineas_kannan_mahoney}.
Detailed recent work appears in~\cite{2009_clarkson_woodruff}.

For many matrices, randomized techniques can produce an accurate
approximation using just one pass over the data.  For Hermitian
matrices, we obtain a single-pass algorithm by combining
Algorithm~\ref{alg:basic}, which constructs an approximate basis,
with Algorithm~\ref{alg:postsym}, which produces an eigenvalue
decomposition without any additional access to the matrix.
Section~\ref{sec:onepass} describes the analogous technique
for general matrices.

For the huge matrices that arise in applications such as data mining,
it is common that the singular spectrum
decays slowly.  Relevant applications include image processing
(see \S\S\ref{sec:graph_laplacian}--\ref{sec:eigenfaces}
for numerical examples), statistical data analysis,
and network monitoring.  To compute approximate
factorizations in these environments, it is crucial
to enhance the accuracy of the randomized approach
using the power scheme, Algorithm~\ref{alg:poweriteration},
or some other device.  This approach increases the
pass count somewhat, but in our experience it is very rare
that more than five passes are required.

\subsection{Gains from parallelization}
\label{sec:parallel}

As mentioned in~\S\S\ref{sec:fastmatvec}--\ref{sec:outofcore},
randomized methods often outperform classical techniques not
because they involve fewer floating-point operations but
rather because they allow us to reorganize the calculations
to exploit the matrix properties and the computer architecture more fully.
In addition, these methods are well suited for parallel implementation.
For example, in Algorithm~\ref{alg:basic}, the computational
bottleneck is the evaluation of the matrix product $\mtx{A\Omega}$,
which is embarrassingly parallelizable.


\lsp

\section{Numerical examples}
\label{sec:numerics}

%

By this time, the reader has surely formulated a pointed question:
Do these randomized matrix approximation algorithms actually work in practice?
In this section, we attempt to address this concern by illustrating
how the algorithms perform on a diverse collection of test cases.

Section \ref{sec:example1} starts with two examples from the physical sciences
involving discrete approximations to operators with exponentially decaying
spectra. Sections~\ref{sec:graph_laplacian} and \ref{sec:eigenfaces} continue
with two examples of matrices arising in ``data mining.'' These are large matrices
whose singular spectra decay slowly; one is sparse and fits in RAM, one is dense
and is stored out-of-core.
Finally, \S\ref{sec:num_SRFT} investigates the performance of randomized methods based
on structured random matrices.

Sections \ref{sec:example1}--\ref{sec:eigenfaces} focus on the algorithms
for Stage A that we presented in~\S\ref{sec:algorithm} because we wish to isolate the performance
of the randomized step.

Computational examples illustrating truly large data matrices have been reported
elsewhere, for instance in \cite{2010_outofcore}.


\subsection{Two matrices with rapidly decaying singular values}
\label{sec:example1}


We first illustrate the behavior of the adaptive range approximation
method, Algorithm~\ref{alg:adaptive2}. We apply it to two matrices
associated with the numerical analysis of differential and integral
operators. The matrices in question have rapidly decaying singular
values and our intent is to demonstrate that in this environment,
the approximation error of a bare-bones randomized method
such as Algorithm \ref{alg:adaptive2} is \emph{very} close to
the minimal error achievable by any method. We observe that the
approximation error of a randomized method is itself a random variable
(it is a function of the random matrix $\mtx{\Omega}$) so what we
need to demonstrate is not only that the error is small in a typical
realization, but also that it clusters tightly around the
mean value.

We first consider a $200 \times 200$ matrix $\mtx{A}$ that results
from discretizing the following single-layer operator
associated with the Laplace equation:
\begin{equation} \label{eq:int_op_laplace}
[S \sigma](x) = {\rm const} \cdot \int_{\Gamma_1} \log \abs{x-y} \, \sigma(y) \idiff{A(y)},
\qquad x \in \Gamma_2,
\end{equation}
where $\Gamma_1$ and $\Gamma_2$ are the two contours in $\mathbb{R}^2$
illustrated in Figure~\ref{fig:contours}(a).  We approximate
the integral with the trapezoidal rule, which converges superalgebraically
because the kernel is smooth.  In the absence of floating-point errors,
we estimate that the discretization error would be less than $10^{-20}$
for a smooth source $\sigma$.  The leading constant is selected so the
matrix $\mtx{A}$ has unit operator norm.

We implement Algorithm~\ref{alg:adaptive2} in Matlab v6.5.
Gaussian test matrices are generated using the {\tt randn} command.
For each number $\ell$ of samples, we compare the following three quantities:
\lsp
\begin{enumerate}
\item   The minimum rank-$\ell$ approximation error $\sigma_{\ell+1}$ is determined using {\tt svd}.

\item   The actual error
$e_{\ell} = \norm{\big(\Id - \mtx{Q}^{(\ell)}(\mtx{Q}^{(\ell)})^{\adj} \big)\mtx{A}}$
is computed with {\tt norm}.

\item   A random estimator $f_\ell$ for the actual error $e_{\ell}$ is obtained from~\eqref{eq:errorest},
        with the parameter $r$ set to $5$.
\end{enumerate}

\lsp

\noindent
Note that any values less than $10^{-15}$ should be considered numerical artifacts.

Figure~\ref{fig:laplace_fullrun} tracks a characteristic execution
of Algorithm~\ref{alg:adaptive2}.  We make three observations:
(i) The error $e_{\ell}$ incurred by the algorithm is remarkably
close to the theoretical minimum $\sigma_{\ell+1}$.
(ii) The error estimate always produces an upper bound for the actual error.
Without the built-in $10\times$ safety margin, the estimate would track
the actual error almost exactly.
(iii) The basis constructed by the algorithm essentially reaches
full double-precision accuracy.

How typical is the trial documented in Figure~\ref{fig:laplace_fullrun}?
To answer this question, we examine the empirical performance
of the algorithm over 2000 independent trials.
Figure~\ref{fig:laplace_stats} charts the error estimate versus
the actual error at four points during the course of execution:
$\ell = 25, 50, 75, 100$.  We offer four observations:
(i) The initial run detailed in Figure~\ref{fig:laplace_fullrun}
is entirely typical.
(ii) Both the actual and estimated error concentrate about their
mean value.
(iii) The actual error drifts slowly away from the optimal error
as the number $\ell$ of samples increases.
(iv) The error estimator is \emph{always} pessimistic by a factor
of about ten, which means that the algorithm \emph{never} produces
a basis with lower accuracy than requested.  The only effect of
selecting an unlucky sample matrix $\mtx{\Omega}$ is that the
algorithm proceeds for a few additional steps.


\begin{figure}
\begin{center}
\begin{tabular}{ccc}
\setlength{\unitlength}{1mm}
\begin{picture}(40,40)
\put(-10,00){\includegraphics[width=58mm]{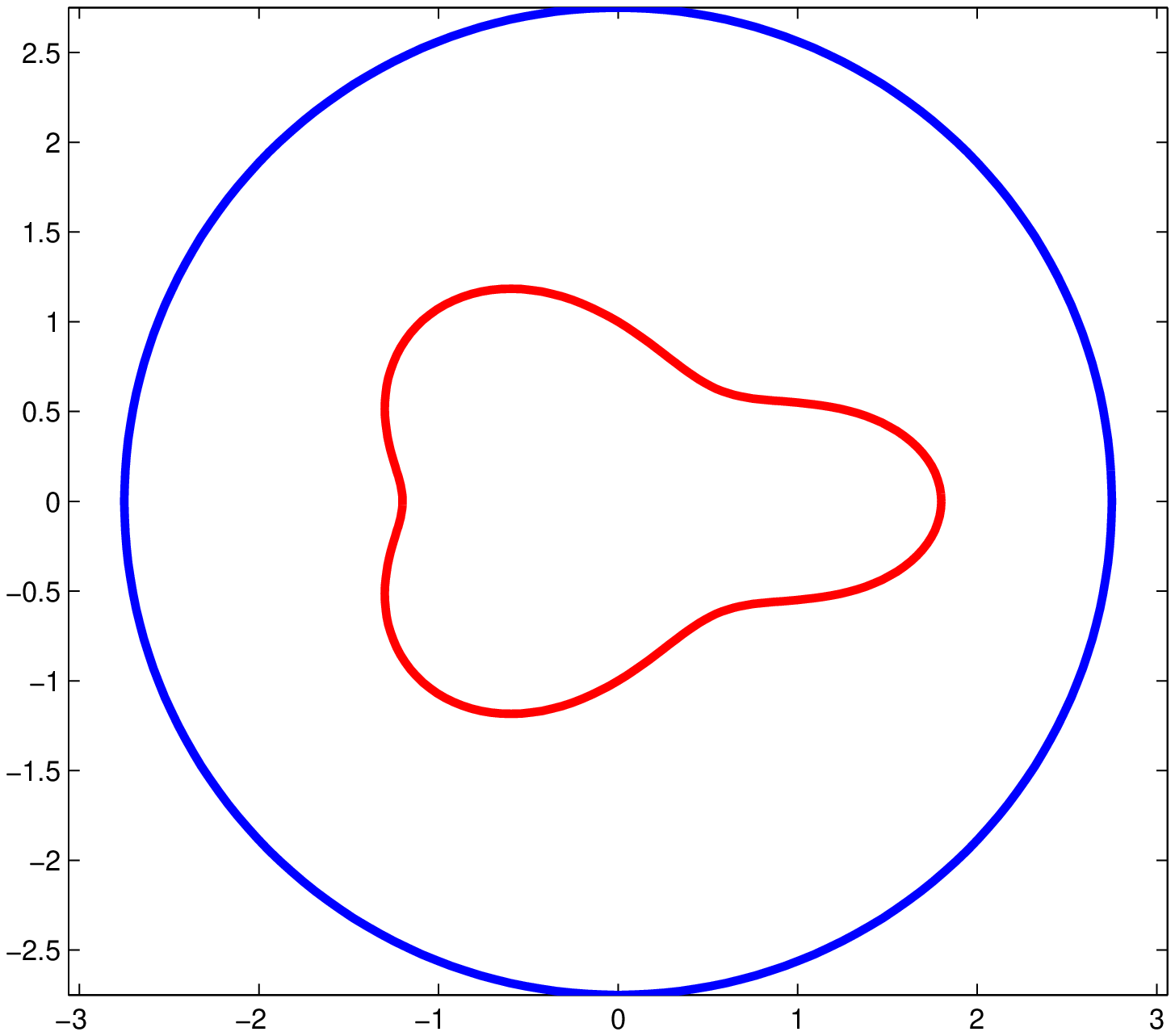}}
\put(11,29){\color{red}$\Gamma_{1}$}
\put(31,09){\color{blue}$\Gamma_{2}$}
\end{picture}
&\mbox{}\hspace{8mm}\mbox{}&
\includegraphics[width=0.43\textwidth]{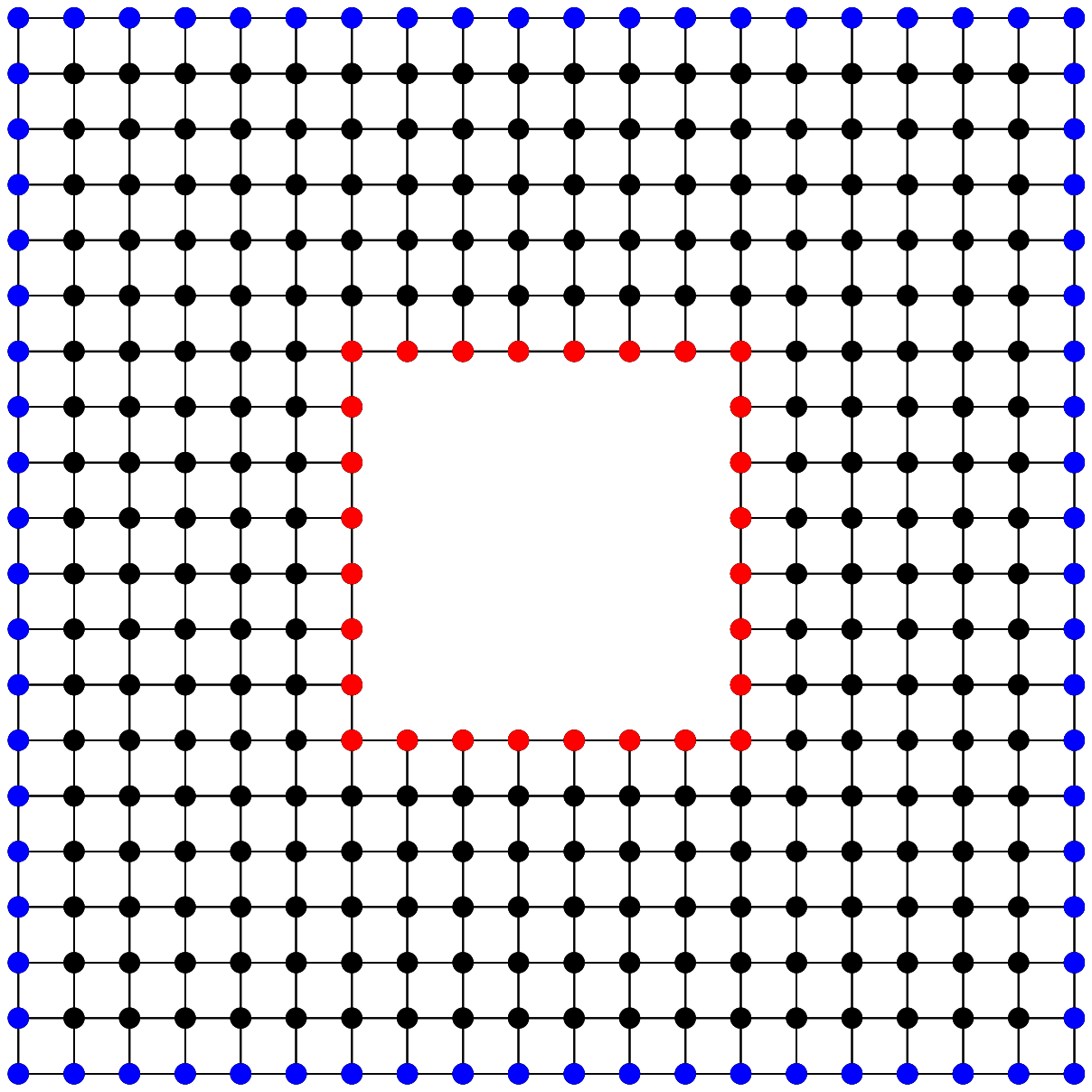}\\
(a) && (b)
\end{tabular}
\end{center}
\caption{{\rm Configurations for physical problems.}
{\rm (a)}  The contours $\Gamma_{1}$ (red) and $\Gamma_{2}$ (blue) for
the integral operator~\eqref{eq:int_op_laplace}.
{\rm (b)} Geometry of the lattice problem associated with matrix $\mtx{B}$ in~\S\ref{sec:example1}.}
\label{fig:contours}
\end{figure}

\begin{figure}
\begin{center}
\setlength{\unitlength}{1mm}
\begin{picture}(125,87)
\put(09,00){\includegraphics[width=108mm]{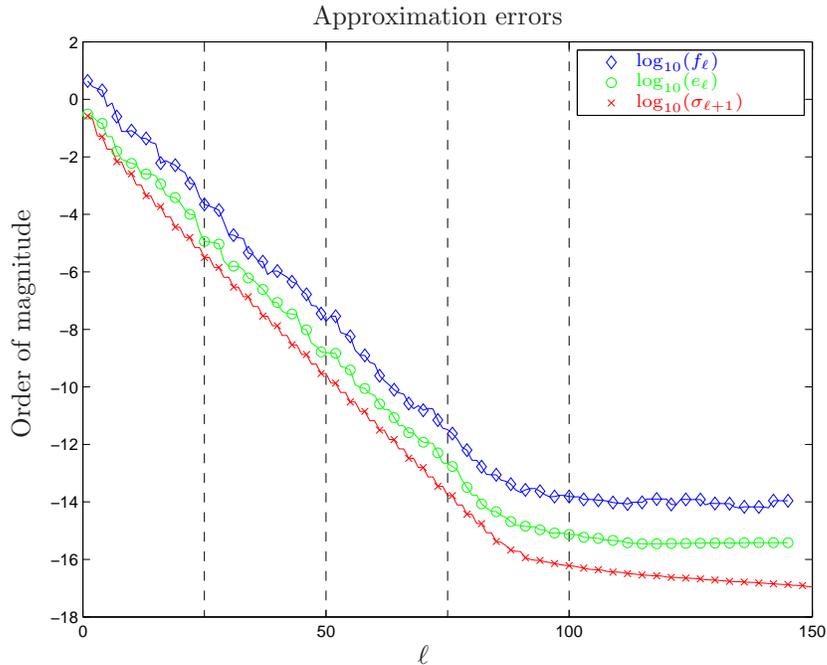}}
\put(62,00){$\ell$}
\put(91,79.5){\scriptsize\color{blue}$\log_{10}(f_{\ell})$}
\put(91,76.75){\scriptsize\color{green}$\log_{10}(e_{\ell})$}
\put(91,74){\scriptsize\color{red}$\log_{10}(\sigma_{\ell+1})$}
\put(48,85){Approximation errors}
\put(8,30){\rotatebox{90}{Order of magnitude}}
\end{picture}
\end{center}
\caption{{\rm Approximating a Laplace integral operator.}
One execution of Algorithm \ref{alg:adaptive2} for the
$200\times 200$ input matrix $\mtx{A}$ described in \S\ref{sec:example1}.
The number $\ell$ of random samples varies along the horizontal axis;
the vertical axis measures the base-10 logarithm of error magnitudes.
The dashed vertical lines mark the points during execution at which
Figure~\ref{fig:laplace_stats} provides additional statistics.}
\label{fig:laplace_fullrun}
\end{figure}

\begin{figure}
\begin{center}
\setlength{\unitlength}{1mm}
\begin{picture}(110,75)
\put(10,03){\includegraphics[width=90mm]{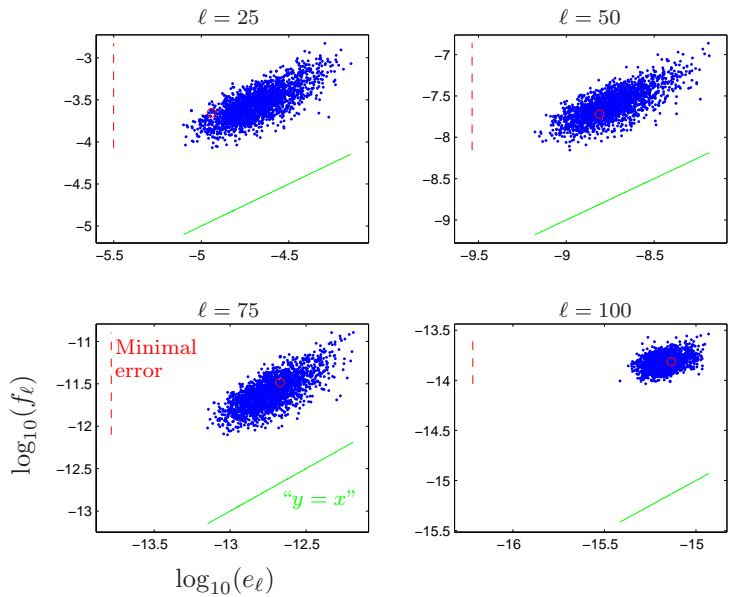}}
\put(26,-2){$\footnotesize \log_{10}(e_{\ell})$}
\put(4,13){\rotatebox{90}{$\footnotesize \log_{10}(f_{\ell})$}}
\put(40,09){\color{green}\footnotesize``$y=x$''}
\put(18,29){\color{red}\footnotesize Minimal}
\put(18,26){\color{red}\footnotesize error}
\put(29,73){\footnotesize $\ell = 25$}
\put(77,73){\footnotesize $\ell = 50$}
\put(29,34){\footnotesize $\ell = 75$}
\put(77,34){\footnotesize $\ell = 100$}
\end{picture}
\end{center}
\caption{{\rm Error statistics for approximating a Laplace integral operator.}
2,000 trials of Algorithm~\ref{alg:adaptive2} applied to a $200 \times 200$ matrix
approximating the integral operator \eqref{eq:int_op_laplace}.
The panels isolate the moments at which $\ell = 25, 50, 75, 100$ random samples have been drawn.
Each solid point compares the estimated error $f_{\ell}$ versus the actual error $e_{\ell}$ in one trial;
the open circle indicates the trial detailed in Figure~\ref{fig:laplace_fullrun}.
The dashed line identifies the minimal error $\sigma_{\ell+1}$,
and the solid line marks the contour where the error estimator would equal the actual error.}
\label{fig:laplace_stats}
\end{figure}

We next consider a matrix $\mtx{B}$ which is defined implicitly in the
sense that we cannot access its elements directly; we can only evaluate
the map $\vct{x} \mapsto \mtx{B}\vct{x}$ for a given vector
$\vct{x}$. To be precise, $\mtx{B}$ represents a transfer matrix for a
network of resistors like the one shown in Figure \ref{fig:contours}(b).
The vector $\vct{x}$ represents a set of electric potentials specified
on the red nodes in the figure. These potentials induce a unique equilibrium
field on the network in which the potential of each black and blue node is
the average of the potentials of its three or four neighbors. The vector
$\mtx{B}\vct{x}$ is then the restriction of the potential to the blue
exterior nodes. Given a vector $\vct{x}$, the vector $\mtx{B}\vct{x}$
can be obtained by solving a large sparse linear system whose coefficient
matrix is the classical five-point stencil approximating the 2D Laplace
operator.

We applied Algorithm~\ref{alg:adaptive2} to the $1596 \times 532$ matrix $\mtx{B}$
associated with a lattice in which there were $532$ nodes (red) on the ``inner ring''
and $1596$ nodes on the (blue) ``outer ring.'' Each application of $\mtx{B}$ to
a vector requires the solution of a sparse linear system of size roughly
$140\,000 \times 140\,000$. We implemented the scheme in Matlab using the
``backslash'' operator for the linear solve.
The results of a typical trial appear in Figure~\ref{fig:lattice_fullrun}.
Qualitatively, the performance matches the results in~Figure \ref{fig:laplace_stats}.

\begin{figure}[h]
\begin{center}
\setlength{\unitlength}{1mm}
\begin{picture}(125,86)
\put(09,00){\includegraphics[width=108mm]{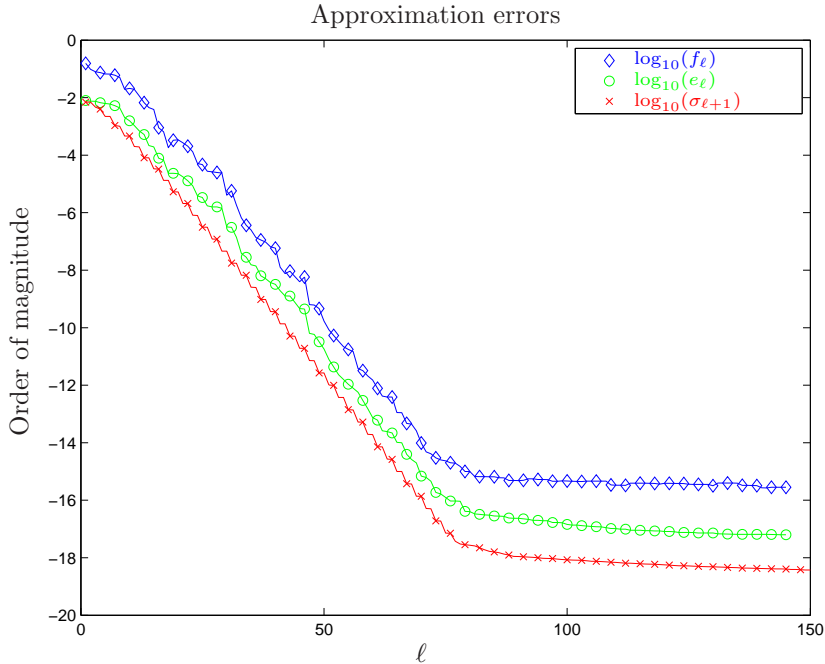}}
\put(62,00){$\ell$}
\put(91,79.5){\scriptsize\color{blue}$\log_{10}(f_{\ell})$}
\put(91,76.75){\scriptsize\color{green}$\log_{10}(e_{\ell})$}
\put(91,74){\scriptsize\color{red}$\log_{10}(\sigma_{\ell+1})$}
\put(48,85){Approximation errors}
\put(8,30){\rotatebox{90}{Order of magnitude}}
\end{picture}
\end{center}
\caption{{\rm Approximating the inverse of a discrete Laplacian.}
One execution of Algorithm \ref{alg:adaptive2} for the
$1596\times 532$ input matrix $\mtx{B}$ described in \S\ref{sec:example1}.
See Figure~\ref{fig:laplace_fullrun} for notations.}
\label{fig:lattice_fullrun}
\end{figure}

\subsection{A large, sparse, noisy matrix arising in image processing}
\label{sec:graph_laplacian}

Our next example involves a matrix that arises in image processing.
A recent line of work uses information about the local geometry
of an image to develop promising new algorithms for standard tasks,
such as denoising, inpainting, and so forth.  These
methods are based on approximating a \term{graph Laplacian}
associated with the image. The dominant eigenvectors of this
matrix provide ``coordinates'' that help us smooth out noisy
image patches \cite{Szlam08,Shen08}.

We begin with a $95 \times 95$ pixel grayscale image. The intensity
of each pixel is represented as an integer in the range $0$ to $4095$.
We form for each pixel $i$ a vector $\vct{x}^{(i)} \in \mathbb{R}^{25}$
by gathering the $25$ intensities of the pixels in a $5 \times 5$ neighborhood
centered at pixel $i$ (with appropriate modifications near the edges).
Next, we form the $9025 \times 9025$ \term{weight matrix}
$\widetilde{\mtx{W}}$ that reflects the similarities between patches:
$$
\widetilde{w}_{ij} = \exp\big\{-\smnorm{}{ \vct{x}^{(i)} - \vct{x}^{(j)} }^2 / \sigma^2 \big\},
$$
where the parameter $\sigma = 50$ controls the level
of sensitivity. We obtain a sparse weight matrix $\mtx{W}$ by
zeroing out all entries in $\widetilde{\mtx{W}}$ except the seven
largest ones in each row.  The object is then to construct the low
frequency eigenvectors of the graph Laplacian matrix
$$
\mtx{L} = \Id - \mtx{D}^{-1/2} \mtx{W} \mtx{D}^{-1/2},
$$
where $\mtx{D}$ is the diagonal matrix with entries $d_{ii} = \sum_j w_{ij}$.
These are the eigenvectors associated with the dominant eigenvalues of
the auxiliary matrix $\mtx{A} = \mtx{D}^{-1/2} \mtx{W} \mtx{D}^{-1/2}$.

The matrix $\mtx{A}$ is large, and its eigenvalues decay slowly,
so we use the power scheme summarized in Algorithm~\ref{alg:poweriteration}
to approximate it.
Figure~\ref{fig:Meyer}[left] illustrates how the approximation error $e_{\ell}$
declines as the number $\ell$ of samples increases.
When we set the exponent $q = 0$, which corresponds with the basic Algorithm~\ref{alg:basic},
the approximation is rather poor.
The graph illustrates that increasing the exponent $q$ slightly results
in a tremendous improvement in the accuracy of the power scheme.

Next, we illustrate the results of using the two-stage approach to
approximate the eigenvalues of $\mtx{A}$.
In Stage A, we construct a basis for $\mtx{A}$ using
Algorithm~\ref{alg:poweriteration} with $\ell = 100$ samples
for different values of $q$.
In Stage B, we apply the Hermitian variant of Algorithm~\ref{alg:Atranspose}
described in~\S\ref{sec:postsym}
to compute an approximate eigenvalue decomposition.
Figure~\ref{fig:Meyer}[right] shows the approximate eigenvalues
and the actual eigenvalues of $\mtx{A}$.
Once again, we see that the minimal exponent $q = 0$
produces miserable results, but
the largest eigenvalues are quite accurate
even for $q=1$.

\begin{figure}[h]
\begin{center}
\setlength{\unitlength}{1mm}
\begin{picture}(125,86)
\put(10,00){\includegraphics[width=105mm]{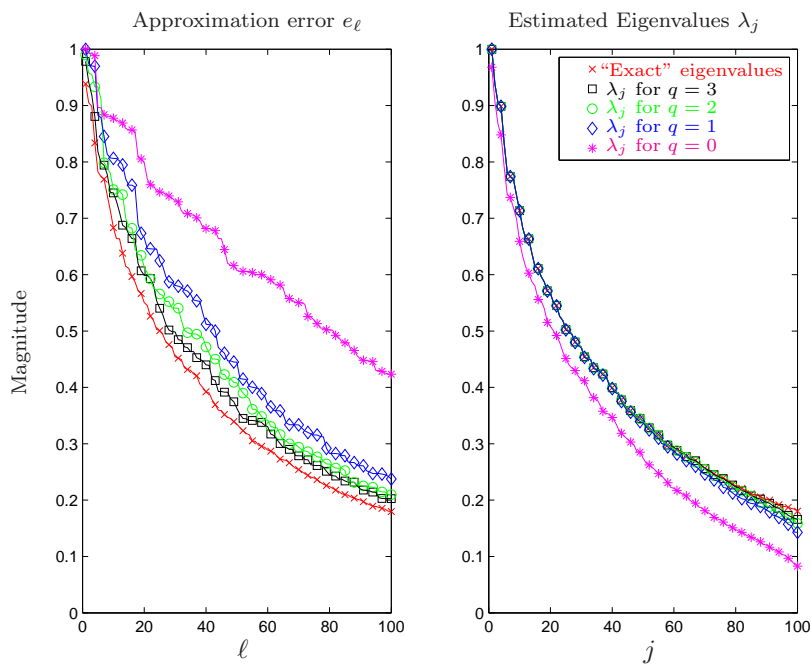}}
\put(38,00){$\ell$}
\put(92,00){$j$}
\put(24,84){\footnotesize Approximation error $e_{\ell}$}
\put(74,84){\footnotesize Estimated Eigenvalues $\lambda_{j}$}
\put(8,35){\rotatebox{90}{\footnotesize Magnitude}}
\put(86,77.5){\scriptsize \color{red}``Exact'' eigenvalues}
\put(87,75.0){\scriptsize \color{black}  $\lambda_{j}$ for $q=3$}
\put(87,72.5){\scriptsize \color{green}  $\lambda_{j}$ for $q=2$}
\put(87,70.0){\scriptsize \color{blue}   $\lambda_{j}$ for $q=1$}
\put(87,67.5){\scriptsize \color{magenta}$\lambda_{j}$ for $q=0$}
\end{picture}
\end{center}
\caption{{\rm Approximating a graph Laplacian.}
For varying exponent $q$, one trial of the power scheme, Algorithm~\ref{alg:poweriteration},
applied to the $9025 \times 9025$ matrix $\mtx{A}$
described in~\S\ref{sec:graph_laplacian}.
{\rm [Left]} Approximation errors as a function of the number $\ell$ of random samples.
{\rm [Right]} Estimates for the 100 largest eigenvalues given $\ell = 100$ random samples
compared with the 100 largest eigenvalues of $\mtx{A}$.}
\label{fig:Meyer}
\end{figure}

\subsection{Eigenfaces}
\label{sec:eigenfaces}

Our next example involves a large, dense matrix derived from
the FERET databank of face images~\cite{feret1,feret2}.  A simple
method for performing face recognition is to identify the principal
directions of the image data, which are called \term{eigenfaces}.
Each of the original photographs can be summarized by its components
along these principal directions.  To identify the subject in a new
picture, we compute its decomposition in this basis and use a classification
technique, such as nearest neighbors, to select the closest image in the
database \cite{1987_eigenfaces}.

We construct a data matrix $\mtx{A}$ as follows:
The FERET database contains $7254$ images, and each
$384 \times 256$ image contains $98\,304$ pixels.  First,
we build a $98\,304 \times 7254$ matrix
$\widetilde{\mtx{A}}$ whose columns are the images.
We form $\mtx{A}$ by centering each column of $\widetilde{\mtx{A}}$
and scaling it to unit norm, so that the images are
roughly comparable.  The eigenfaces are the dominant
left singular vectors of this matrix.

Our goal then is to compute an approximate SVD of the matrix $\mtx{A}$.
Represented as an array of double-precision real numbers, $\mtx{A}$
would require $5.4$\,GB of storage, which does not fit within the fast
memory of many machines. It is possible to compress the database
down to at $57$\,MB or less (in JPEG format), but then the data would
have to be uncompressed with each sweep over the matrix.
%
%
Furthermore, the matrix $\mtx{A}$ has slowly
decaying singular values, so we need to use
the power scheme, Algorithm~\ref{alg:poweriteration},
to capture the range of the matrix accurately.

To address these concerns, we implemented the
power scheme to run in a pass-efficient manner.
An additional difficulty arises because the size of the data
makes it expensive to calculate the
actual error $e_{\ell}$ incurred by the approximation
or to determine the minimal error $\sigma_{\ell+1}$.
To estimate the errors, we use the technique described
in Remark~\ref{remark:better_errorestimate}.

Figure~\ref{fig:eigenfaces} describes the behavior
of the power scheme, which is similar to its
performance for the graph Laplacian in~\S\ref{sec:graph_laplacian}.
When the exponent $q = 0$, the approximation of
the data matrix is very poor, but it improves quickly
as $q$ increases.  Likewise, the estimate for the
spectrum of $\mtx{A}$ appears to converge
rapidly; the largest singular values are
already quite accurate when $q = 1$.
We see essentially no improvement in the estimates
after the first 3--5 passes over the matrix.



\begin{figure}[h]
\begin{center}
\setlength{\unitlength}{1mm}
\begin{picture}(125,88)
\put(08,00){\includegraphics[width=110mm]{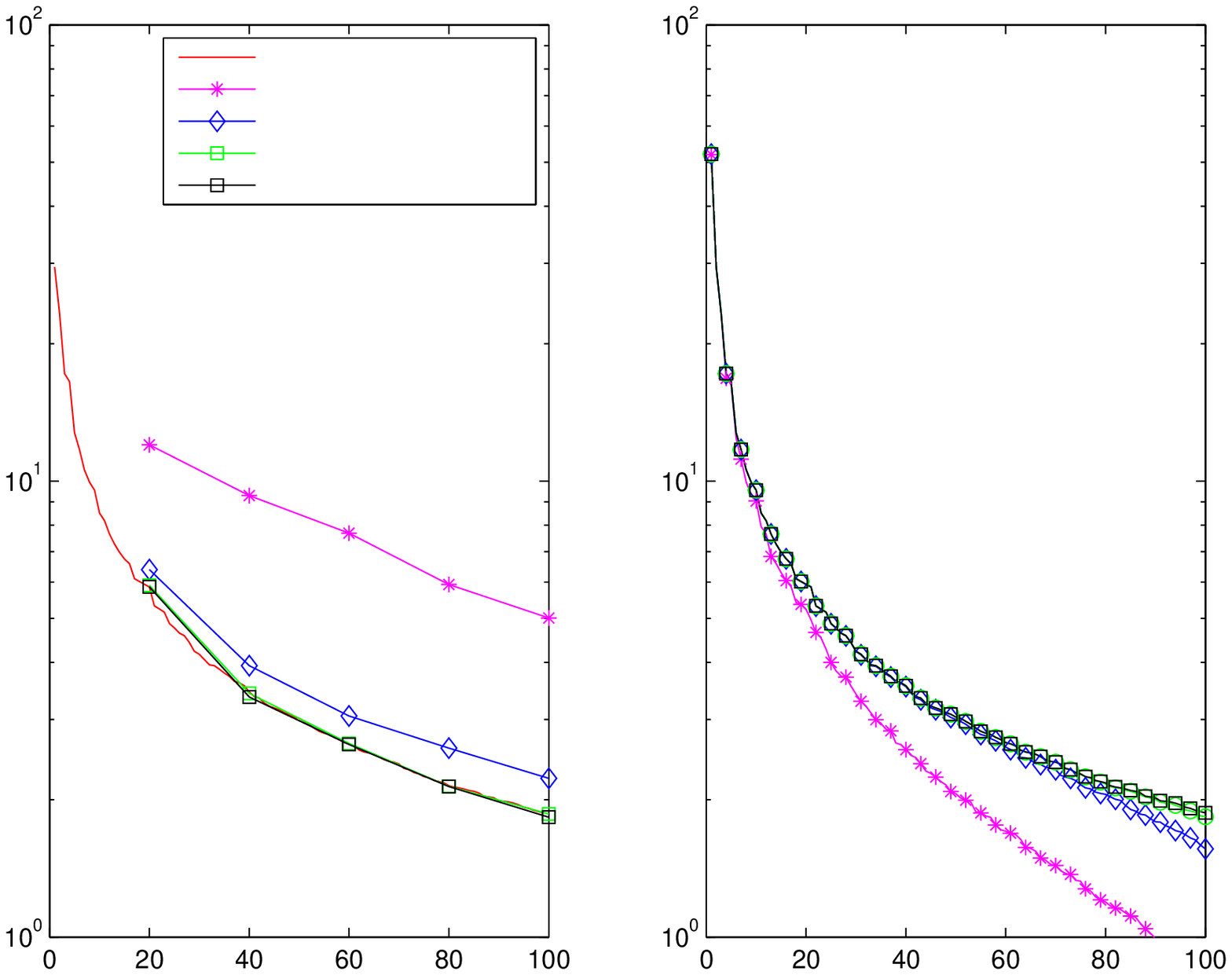}}
\put(24,86){\footnotesize Approximation error $e_{\ell}$}
\put(74,86){\footnotesize Estimated Singular Values $\sigma_{j}$}
\put(8,40){\rotatebox{90}{\footnotesize Magnitude}}
\put(34,81.50){\scriptsize \color{red}Minimal error (est)}
\put(35,78.75){\scriptsize \color{magenta}$q=0$}
\put(35,76.00){\scriptsize \color{blue}   $q=1$}
\put(35,73.25){\scriptsize \color{red}    $q=2$}
\put(35,70.50){\scriptsize \color{black}  $q=3$}
\put(37,00){$\ell$}
\put(93,00){$j$}
\end{picture}
\end{center}
\caption{{\rm Computing eigenfaces.}
For varying exponent $q$, one trial of the power scheme, Algorithm~\ref{alg:poweriteration},
applied to the $98\,304 \times 7254$ matrix $\mtx{A}$ described in~\S\ref{sec:eigenfaces}.
{\rm (Left)} Approximation errors as a function of the number $\ell$ of random samples.
The red line indicates the minimal errors as estimated by the singular values computed using
$\ell = 100$ and $q=3$.
{\rm (Right)} Estimates for the 100 largest eigenvalues given $\ell = 100$ random samples.}
\label{fig:eigenfaces}
\end{figure}

\subsection{Performance of structured random matrices}
\label{sec:num_SRFT}

Our final set of experiments illustrates that the structured random matrices
described in~\S\ref{sec:ailonchazelle} lead to matrix approximation algorithms
that are both fast and accurate.

First, we compare the computational speeds of four methods for computing an
approximation to the $\ell$ dominant terms in the SVD of an $n \times n$ matrix $\mtx{A}$.
For now, we are interested in execution time only (not accuracy), so the
choice of matrix is irrelevant and we have selected $\mtx{A}$ to be a Gaussian matrix.
The four methods are summarized in the following table; Remark~\ref{rem:fortran}
provides more details on the implementation.

\lsp

\begin{center}
\renewcommand{\arraystretch}{1.2}
\begin{tabular}{|l||l|l|}
\hline
\textbf{Method} & \textbf{Stage A} & \textbf{Stage B} \\
\hline \hline
\texttt{direct} & Rank-revealing QR executed using column & Algorithm \ref{alg:Atranspose} \\
          & pivoting and Householder reflectors & \\ \hline
\texttt{gauss} & Algorithm \ref{alg:basic} with a Gaussian random matrix & Algorithm \ref{alg:Atranspose}\\ \hline
\texttt{srft} & Algorithm \ref{alg:basic} with the modified SRFT \eqref{eq:random_Givens} & Algorithm \ref{alg:extractrows} \\
\hline
\texttt{svd} & 
Full SVD with LAPACK routine {\tt dgesdd} & Truncate to $\ell$ terms \\
\hline
\end{tabular}
\end{center}

\lsp

Table~\ref{tab:runtimes} lists the measured runtime of a single execution
of each algorithm for various choices of the dimension $n$ of the input
matrix and the rank $\ell$ of the approximation.  Of course, the cost of
the full SVD does not depend on the number $\ell$ of components required.
A more informative way to look at the runtime data is to compare the
\emph{relative} cost of the algorithms.  The {\tt direct} method is
the best deterministic approach for dense matrices, so we calculate
the factor by which the randomized methods improve on this benchmark.
Figure~\ref{fig:speedup} displays the results.
We make two observations:
(i) Using an SRFT often leads to a dramatic
speed-up over classical techniques, even for moderate problem sizes.
(ii) Using a standard Gaussian test matrix typically leads to a moderate
speed-up over classical methods, primarily because performing a
matrix--matrix multiplication is faster than a QR factorization.

Second, we investigate how the choice of random test matrix influences the
error in approximating an input matrix.  For these experiments, we return to
the $200 \times 200$ matrix $\mtx{A}$ defined in Section~\ref{sec:example1}.
Consider variations of Algorithm \ref{alg:basic} obtained when
the random test matrix $\mtx{\Omega}$ is drawn from the following four distributions:

\lsp

\begin{tabbing}
\mbox{}\hspace{5mm}\= \hspace{17mm} \= \kill
\> \textbf{Gauss:} \> The standard Gaussian distribution.\\[1mm]
\> \textbf{Ortho:} \> The uniform distribution on $n \times \ell$ orthonormal matrices.\\[1mm]
\> \textbf{SRFT:}  \> The SRFT distribution defined in~\eqref{eq:def_srft}.\\[1mm]
\> \textbf{GSRFT:} \>  The modified SRFT distribution defined in~\eqref{eq:random_Givens}.
\end{tabbing}

\lsp

\noindent
Intuitively, we expect that {\bf Ortho} should provide the best performance.


For each distribution, we perform 100\,000 trials of the following experiment.
Apply the corresponding version of Algorithm~\ref{alg:basic} to the matrix $\mtx{A}$,
and calculate the approximation error $e_{\ell} = \norm{\mtx{A} - \mtx{Q}_{\ell}\mtx{Q}_{\ell}^{\adj}\mtx{A}}$.
Figure~\ref{fig:SRFT_errors} displays the empirical probability density function for the error $e_{\ell}$
obtained with each algorithm.  We offer three observations:
(i) The SRFT actually performs slightly \textit{better} than a Gaussian random matrix for
this example.
(ii) The standard SRFT and the modified SRFT have essentially identical errors.
(iii) There is almost no difference between the Gaussian random
matrix and the random orthonormal matrix in the first three plots,
while the fourth plot shows that the random orthonormal matrix performs better.
This behavior occurs because, with high probability, a tall Gaussian matrix
is well conditioned and a (nearly) square Gaussian matrix is not.

\lsp

\begin{remark} \label{rem:fortran} \rm
The running times reported in Table~\ref{tab:runtimes} and in Figure~\ref{fig:speedup} depend
strongly on both the computer hardware and the coding of the algorithms.  The experiments
reported here were performed on a standard office desktop with a 3.2\,GHz Pentium IV processor and  2\,GB of RAM.
The algorithms were implemented in Fortran 90 and compiled with the Lahey compiler. The Lahey versions
of BLAS and LAPACK were used to accelerate all matrix--matrix multiplications, as well as the SVD computations in
Algorithms~\ref{alg:Atranspose} and~\ref{alg:extractrows}.  We used the code for the modified SRFT
\eqref{eq:random_Givens} provided in the publicly available software package \texttt{id$\underline{\mbox{ }}$dist} \cite{id_dist}.
\end{remark}

\begin{table}
\begin{center}
\parbox{5in}{
\caption{{\rm Computational times for a partial SVD.} The time, in seconds, required to compute the $\ell$
leading components in the SVD of an $n\times n$ matrix using each of the methods from~\S\ref{sec:num_SRFT}.
The last row indicates the time needed to obtain a full SVD.}
\label{tab:runtimes}}

\renewcommand{\arraystretch}{1.1}

{\scriptsize
\begin{tabular}{|r||r@{e}r|r@{e}r|r@{e}r||r@{e}r|r@{e}r|r@{e}r||r@{e}r|r@{e}r|r@{e}r|} \hline
& \multicolumn{6}{|c||}{$n = 1024$} & \multicolumn{6}{|c||}{$n = 2048$} & \multicolumn{6}{|c|}{$n = 4096$}\\
\multicolumn{1}{|c||}{$\ell$} &
    \multicolumn{2}{c}{\tt direct} & \multicolumn{2}{c}{\tt gauss} & \multicolumn{2}{c||}{\tt srft} &
    \multicolumn{2}{c}{\tt direct} & \multicolumn{2}{c}{\tt gauss} & \multicolumn{2}{c||}{\tt srft} &
    \multicolumn{2}{c}{\tt direct} & \multicolumn{2}{c}{\tt gauss} & \multicolumn{2}{c|}{\tt srft} \\ \hline
10  &
1.08&-1  &  5.63&-2  &  9.06&-2  &  4.22&-1  &  2.16&-1  &  3.56&-1  &  1.70&\phantom{-}0  &  8.94&-1  &  1.45&\phantom{-}0  \\
20  &
1.97&-1  &  9.69&-2  &  1.03&-1  &  7.67&-1  &  3.69&-1  &  3.89&-1  &  3.07&0  &  1.44&0  &  1.53&0  \\
40  &
3.91&-1  &  1.84&-1  &  1.27&-1  &  1.50&0  &  6.69&-1  &  4.33&-1  &  6.03&0  &  2.64&0  &  1.63&0  \\
80  &
7.84&-1  &  4.00&-1  &  2.19&-1  &  3.04&0  &  1.43&0  &  6.64&-1  &  1.20&1  &  5.43&0  &  2.08&0  \\
160  &
1.70&0  &  9.92&-1  &  6.92&-1  &  6.36&0  &  3.36&0  &  1.61&0  &  2.46&1  &  1.16&1  &  3.94&0  \\
320  &
3.89&0  &  2.65&0  &  2.98&0  &  1.34&1  &  7.45&0  &  5.87&0  &  5.00&1  &  2.41&1  &  1.21&1  \\
640  &
1.03&1  &  8.75&0  &  1.81&1  &  3.14&1  &  2.13&1  &  2.99&1  &  1.06&2  &  5.80&1  &  5.35&1  \\
1280  &
\multicolumn{2}{c|}{---} & \multicolumn{2}{c|}{---} & \multicolumn{2}{c||}{---} &
7.97&1  &  6.69&1  &  3.13&2  &  2.40&2  &  1.68&2  &  4.03&2  \\ \hline \hline
{\tt svd} &
    \multicolumn{2}{|c}{} & \multicolumn{2}{c}{1.19e\phantom{-}1} & \multicolumn{2}{c||}{} &
    \multicolumn{2}{|c}{} & \multicolumn{2}{c}{8.77e\phantom{-}1} & \multicolumn{2}{c||}{} &
    \multicolumn{2}{|c}{} & \multicolumn{2}{c}{6.90e\phantom{-}2} & \multicolumn{2}{c|}{} \\ \hline
\end{tabular}
}
\end{center}
\end{table}

\begin{figure}
\begin{center}
\setlength{\unitlength}{1mm}
\begin{picture}(130,105)
\put(10,04){\includegraphics[width=120mm]{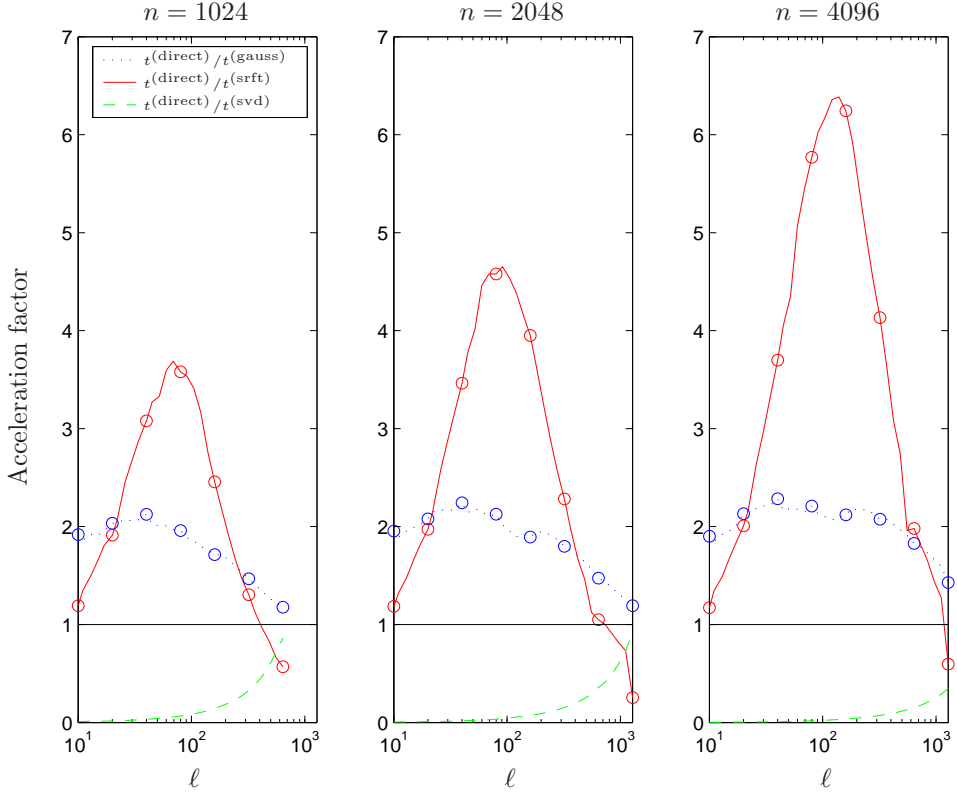}}
\put(27,00){$\ell$}
\put(69,00){$\ell$}
\put(111,00){$\ell$}
\put(21, 102){$n=1024$}
\put(63, 102){$n=2048$}
\put(105,102){$n=4096$}
\put(21,96){\tiny$t^{\rm(direct)}/t^{\rm(gauss)}$}
\put(21,93){\tiny$t^{\rm(direct)}/t^{\rm(srft)}$}
\put(21,90){\tiny$t^{\rm(direct)}/t^{\rm(svd)}$}
\put(3,40){\rotatebox{90}{Acceleration factor}}
\end{picture}
\end{center}
\caption{{\rm Acceleration factor.}
The relative cost of computing an $\ell$-term partial SVD of an $n \times n$ Gaussian
matrix using {\tt direct}, a benchmark classical algorithm, versus each of the three
competitors described in~\S\ref{sec:num_SRFT}.  The solid red curve shows the speedup using
an SRFT test matrix, and the dotted blue curve shows the speedup with a Gaussian test matrix.
The dashed green curve indicates that a full SVD computation using classical methods
is substantially {\em slower}.  Table~\ref{tab:runtimes} reports the absolute runtimes that yield the
circled data points.}
\label{fig:speedup}
\end{figure}

\begin{figure}
\begin{center}
\setlength{\unitlength}{1mm}
\begin{picture}(130,102)
\put(08,00){\includegraphics[width=120mm]{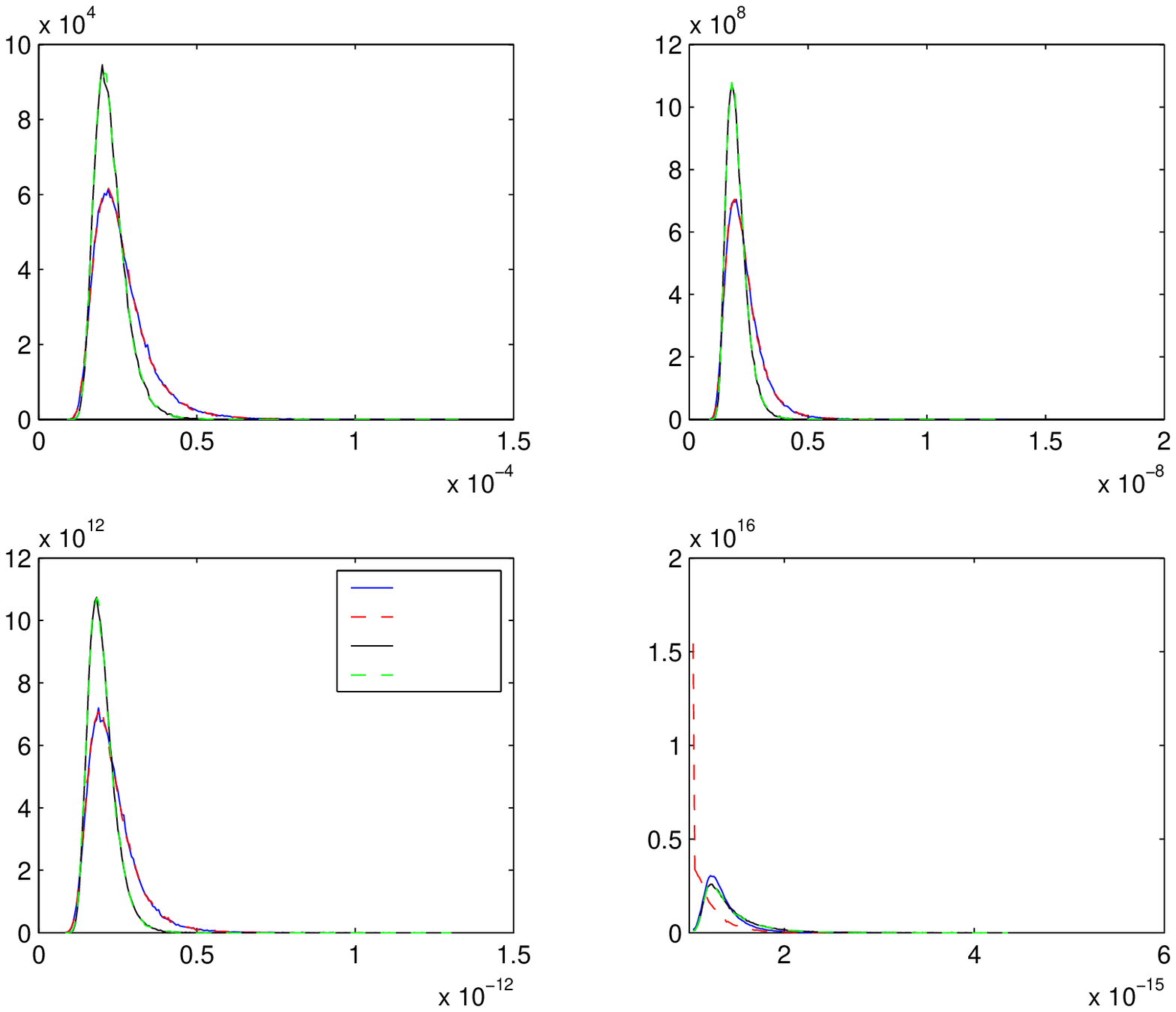}}
\put(02,13){\rotatebox{90}{Empirical density}}
\put(36,55){$e_{25}$}
\put(101,55){$e_{50}$}
\put(36,03){$e_{75}$}
\put(101,03){$e_{100}$}
\put(33,100){$\ell = 25$}
\put(33,48){$\ell = 75$}
\put(98,100){$\ell = 50$}
\put(98,48){$\ell = 100$}
\put(48,42.5){\footnotesize Gauss}
\put(48,39.5){\footnotesize Ortho}
\put(48,36.5){\footnotesize SRFT}
\put(48,33.5){\footnotesize GSRFT}
\end{picture}
\end{center}
\caption{{\rm Empirical probability density functions for the error in Algorithm \ref{alg:basic}.}
As described in~\S\ref{sec:num_SRFT}, the algorithm is implemented with four
distributions for the random test matrix and used to approximate the $200\times 200$ input matrix obtained by
discretizing the integral operator \eqref{eq:int_op_laplace}.
The four panels capture the empirical error distribution for each version of the algorithm
at the moment when $\ell = 25, 50, 75, 100$ random samples have been drawn.}
\label{fig:SRFT_errors}
\end{figure}


\vspace{5mm}

\begin{center}
{\bf Part III: Theory}
\end{center}

\lsp

This part of the paper, \S\S\ref{sec:lin-alg-prelim}--\ref{sec:SRFTs},
provides a detailed analysis of randomized sampling schemes
for constructing an approximate basis for the range of a matrix,
the task we refer to as Stage A in the framework of~\S\ref{sec:framework}.
More precisely, we assess the quality of the basis $\mtx{Q}$
that the proto-algorithm of~\S\ref{sec:sketchofalgorithm} produces
by establishing rigorous bounds for the approximation error
\begin{equation}
\label{eq:basicerror}
\triplenorm{ \mtx{A} - \mtx{QQ}^\adj \mtx{A} },
\end{equation}
where $\triplenorm{\cdot}$ denotes either the spectral norm or the
Frobenius norm. The difficulty in developing these bounds is that
the matrix $\mtx{Q}$ is random, and its distribution is a
complicated nonlinear function of the input matrix $\mtx{A}$ and
the random test matrix $\mtx{\Omega}$.  Naturally, any estimate for
the approximation error must depend on the properties of the input
matrix and the distribution of the test matrix.

To address these challenges, we split the argument into two
pieces.  The first part exploits techniques from linear algebra
to deliver a generic error bound that depends
on the interaction between the test matrix $\mtx{\Omega}$
and the right singular vectors of the input matrix $\mtx{A}$,
as well as the tail singular values of $\mtx{A}$.
In the second part of the argument, we take into
account the distribution of the random matrix to estimate the
error for specific instantiations of the proto-algorithm.
This bipartite proof is common in the literature on randomized
linear algebra, but our argument is most similar in spirit
to~\cite{BMD09:Improved-Approximation}.

Section~\ref{sec:lin-alg-prelim} surveys the basic linear algebraic tools
we need.
Section~\ref{sec:basic_err} uses these methods to derive a generic error
bound.  Afterward, we specialize these results to the case where the test matrix is
Gaussian~(\S\ref{sec:gaussians}) and the case where the test matrix is
a subsampled random Fourier transform~(\S\ref{sec:SRFTs}).

\section{Theoretical preliminaries} \label{sec:lin-alg-prelim}

We proceed with some additional background from linear algebra.
Section~\ref{sec:psd} sets out properties of positive-semidefinite
matrices, and \S\ref{sec:orthproj} offers some results
for orthogonal projectors.  Standard references for this material
include~\cite{Bha97:Matrix-Analysis,HJ85:Matrix-Analysis}.

\subsection{Positive semidefinite matrices} \label{sec:psd}

An Hermitian matrix $\mtx{M}$ is \term{positive semidefinite}
(briefly, \term{psd}) when $\vct{u}^\adj \mtx{M} \vct{u} \geq 0$ for all
$\vct{u} \neq \vct{0}$.
If the inequalities are strict, $\mtx{M}$ is \term{positive
definite} (briefly, \term{pd}).
The psd matrices form a convex cone, which induces a partial ordering
on the linear space of Hermitian matrices:
$\mtx{M} \psdle \mtx{N}$ if and only if
$\mtx{N} - \mtx{M}$ is psd.
This ordering allows us to write $\mtx{M} \psdge \mtx{0}$ to
indicate that the matrix $\mtx{M}$ is psd.

Alternatively, we can define a psd (resp., pd) matrix as an
Hermitian matrix with nonnegative (resp., positive) eigenvalues.
In particular, each psd matrix is diagonalizable, and the inverse
of a pd matrix is also pd.  The spectral norm of a psd matrix
$\mtx{M}$ has the variational characterization
\begin{equation}
\label{eq:norm_psd}
\norm{\mtx{M}}
    = \max_{\vct{u} \neq \vct{0}} \frac{\vct{u}^\adj \mtx{M} \vct{u}}{\vct{u}^\adj \vct{u}},
\end{equation}
according to the Rayleigh--Ritz
theorem~\cite[Thm.~4.2.2]{HJ85:Matrix-Analysis}.  It follows that
\begin{equation}
\label{eq:psdle_norm}
\mtx{M} \psdle \mtx{N} \quad\Longrightarrow\quad \norm{\mtx{M}}
\leq \norm{\mtx{N}}.
\end{equation}

A fundamental fact is that conjugation preserves the psd property.

\lsp

\begin{proposition}[Conjugation Rule]
\label{prop:conjugation}
Suppose that $\mtx{M} \psdge \mtx{0}$.  For every $\mtx{A}$, the
matrix $\mtx{A}^\adj \mtx{M} \mtx{A} \psdge \mtx{0}$.  In
particular,
$$
\mtx{M} \psdle \mtx{N} \quad\Longrightarrow\quad \mtx{A}^\adj
\mtx{M} \mtx{A} \psdle \mtx{A}^\adj \mtx{N} \mtx{A}.
$$
\end{proposition}

Our argument invokes the conjugation rule repeatedly.  As a
first application, we establish a perturbation bound for the
matrix inverse near the identity matrix.

\lsp

\begin{proposition}[Perturbation of Inverses] \label{prop:perturb-inv}
Suppose that $\mtx{M} \psdge \mtx{0}$.  Then
$$
\Id - (\Id + \mtx{M})^{-1} \psdle \mtx{M}
$$
\end{proposition}


\begin{proof}
Define $\mtx{R} = \mtx{M}^{1/2}$, the psd square root of $\mtx{M}$
promised by~\cite[Thm.~7.2.6]{HJ85:Matrix-Analysis}.  We have the
chain of relations
$$
\Id - (\Id + \mtx{R}^2)^{-1} = (\Id + \mtx{R}^2)^{-1} \mtx{R}^2
    = \mtx{R} (\Id + \mtx{R}^2)^{-1} \mtx{R}
    \psdle \mtx{R}^2.
$$
The first equality can be verified algebraically.  The second
holds because rational functions of a diagonalizable matrix, such
as $\mtx{R}$, commute.  The last relation follows from the
conjugation rule because $(\Id + \mtx{R}^2)^{-1} \psdle \Id$.
\end{proof}

\lsp

Next, we present a generalization of the fact that the spectral norm
of a psd matrix is controlled by its trace.

\lsp

\begin{proposition}
    \label{prop:spec-norm-sum}
We have $\norm{\mtx{M}} \leq \norm{\mtx{A}} + \norm{\mtx{C}}$
for each partitioned psd matrix
\begin{equation*} \label{eqn:M-matrix}
\mtx{M} =
\begin{bmatrix} \mtx{A} & \mtx{B} \\ \mtx{B}^\adj &
\mtx{C} \end{bmatrix}.
\end{equation*}
\end{proposition}


\begin{proof}
The variational characterization \eqref{eq:norm_psd} of the spectral norm
implies that
\begin{align*}
\norm{ \mtx{M} } &=
\sup_{\normsq{\vct{x}} + \normsq{\vct{y}}=1}
\begin{bmatrix} \vct{x} \\ \vct{y} \end{bmatrix}^\adj
\begin{bmatrix} \mtx{A} & \mtx{B} \\ \mtx{B}^\adj &
\mtx{C} \end{bmatrix}
\begin{bmatrix} \vct{x} \\ \vct{y} \end{bmatrix} \\
&\leq
\sup_{\normsq{\vct{x}} + \normsq{\vct{y}}=1}
\bigl(\norm{\mtx{A}} \normsq{\vct{x}}
        + 2 \norm{\mtx{B}} \norm{\vct{x}} \norm{\vct{y}}
        + \norm{\mtx{C}} \normsq{\vct{y}}\bigr).
\end{align*}
The block generalization of Hadamard's psd criterion~\cite[Thm.~7.7.7]{HJ85:Matrix-Analysis}
states that $\normsq{ \mtx{B} } \leq \norm{\mtx{A}} \norm{\mtx{C}}$.  Thus,
$$
\norm{ \mtx{M} } \leq \sup_{\normsq{\vct{x}} + \normsq{\vct{y}}=1}
\bigl(\norm{\mtx{A}}^{1/2} \norm{\vct{x}} + \norm{\mtx{C}}^{1/2} \norm{\vct{y}}\bigr)^{2} =
\norm{\mtx{A}} + \norm{\mtx{C}}.
$$
This point completes the argument.
\end{proof}

\subsection{Orthogonal projectors} \label{sec:orthproj}

An \term{orthogonal projector} is an Hermitian matrix $\mtx{P}$
that satisfies the polynomial $\mtx{P}^2 = \mtx{P}$. This identity implies $\mtx{0}
\psdle \mtx{P} \psdle \Id$. An orthogonal projector is completely
determined by its range. For a given matrix $\mtx{M}$, we write
$\mtx{P}_{\mtx{M}}$ for the unique orthogonal projector with
$\range(\mtx{P}_{\mtx{M}}) = \range(\mtx{M})$.  When $\mtx{M}$ has
full column rank, we can express this projector explicitly:
\begin{equation} \label{eqn:orth-proj-form}
\mtx{P}_{\mtx{M}} = \mtx{M} (\mtx{M}^\adj \mtx{M})^{-1} \mtx{M}^\adj.
\end{equation}
The orthogonal projector onto the complementary subspace,
$\range(\mtx{P})^{\perp}$, is the matrix $\Id - \mtx{P}$.
Our argument hinges on several other facts about orthogonal
projectors.


\lsp

\begin{proposition} \label{prop:conj-proj}
Suppose $\mtx{U}$ is unitary.  Then $\mtx{U}^\adj
\mtx{P}_{\mtx{M}} \mtx{U} = \mtx{P}_{\mtx{U}^\adj \mtx{M}}$.
\end{proposition}

\lsp

\begin{proof}
Abbreviate $\mtx{P} = \mtx{U}^\adj \mtx{P}_{\mtx{M}} \mtx{U}$. It is
clear that $\mtx{P}$ is an orthogonal projector since it is Hermitian
and $\mtx{P}^{2} = \mtx{P}$. Evidently,
$$
\range(\mtx{P}) = \mtx{U}^\adj \range(\mtx{M}) =
\range(\mtx{U}^\adj \mtx{M}).
$$
Since the range determines the orthogonal projector, we
conclude $\mtx{P} = \mtx{P}_{\mtx{U}^\adj \mtx{M}}$.
\end{proof}

\lsp

\begin{proposition} \label{prop:proj-range}
Suppose $\range(\mtx{N}) \subset \range(\mtx{M})$.  Then, for
each matrix $\mtx{A}$, it holds that $ \norm{\mtx{P}_{\mtx{N}}
\mtx{A}} \leq \norm{\mtx{P}_{\mtx{M}} \mtx{A}}$ and that
$\norm{(\Id - \mtx{P}_{\mtx{M}}) \mtx{A}}
    \leq \norm{(\Id - \mtx{P}_{\mtx{N}}) \mtx{A}}$.
\end{proposition}

\lsp

\begin{proof}
The projector $\mtx{P}_{\mtx{N}} \psdle \Id$, so the conjugation
rule yields
$\mtx{P}_{\mtx{M}}\mtx{P}_{\mtx{N}}\mtx{P}_{\mtx{M}} \psdle
\mtx{P}_{\mtx{M}}$.  The hypothesis $\range(\mtx{N}) \subset
\range(\mtx{M})$ implies that $\mtx{P}_{\mtx{M}} \mtx{P}_{\mtx{N}}
= \mtx{P}_{\mtx{N}}$, which results in
$$
\mtx{P}_{\mtx{M}} \mtx{P}_{\mtx{N}} \mtx{P}_{\mtx{M}}
    = \mtx{P}_{\mtx{N}} \mtx{P}_{\mtx{M}}
    = (\mtx{P}_{\mtx{M}} \mtx{P}_{\mtx{N}})^\adj
    = \mtx{P}_{\mtx{N}}.
$$
In summary, $\mtx{P}_{\mtx{N}} \psdle \mtx{P}_{\mtx{M}}$.
The conjugation rule shows that $\mtx{A}^\adj
\mtx{P}_{\mtx{N}} \mtx{A} \psdle \mtx{A}^\adj \mtx{P}_{\mtx{M}}
\mtx{A}$.  We conclude from~\eqref{eq:psdle_norm} that
$$
\normsq{ \mtx{P}_{\mtx{N}} \mtx{A} }
    = \norm{ \mtx{A}^\adj \mtx{P}_{\mtx{N}} \mtx{A} }
    \leq \norm{ \mtx{A}^\adj \mtx{P}_{\mtx{M}} \mtx{A} }
    = \normsq{ \mtx{P}_{\mtx{M}} \mtx{A} }.
$$
The second statement follows from the first by taking orthogonal
complements.
\end{proof}

\lsp

Finally, we need a generalization of the scalar inequality $\abs{px}^q \leq \abs{p} \abs{x}^q$, which holds when $\abs{p} \leq 1$ and $q \geq 1$.

\lsp


\begin{proposition} \label{prop:proj-power}
Let $\mtx{P}$ be an orthogonal projector, and let $\mtx{M}$ be a matrix.
For each positive number $q$,
\begin{equation}
\label{eq:trivial}
\norm{ \mtx{PM} } \leq \norm{ \mtx{P}
{(\mtx{MM}^\adj)}^{q} \mtx{M} }^{1/(2q+1)}.
\end{equation}
\end{proposition}

\begin{proof}
Suppose that $\mtx{R}$ is an orthogonal projector, $\mtx{D}$ is a nonnegative diagonal
matrix, and $t \geq 1$.  We claim that
\begin{equation} \label{eqn:power-claim}
\norm{ \mtx{RDR} }^t \leq \norm{\mtx{R} \mtx{D}^t \mtx{R}}.
\end{equation}
Granted this inequality, we quickly complete the proof.  Using an SVD $\mtx{M} = \mtx{U\Sigma V}^\adj$,
we compute
\begin{multline*}
\norm{ \mtx{PM} }^{2(2q+1)}
    = \norm{ \mtx{PMM}^\adj \mtx{P} }^{2q+1}
    = \norm{ (\mtx{U}^\adj \mtx{PU}) \cdot \mtx{\Sigma}^2 \cdot (\mtx{U}^\adj \mtx{P} \mtx{U}) }^{2q+1} \\
    \leq \smnorm{}{ (\mtx{U}^\adj \mtx{PU}) \cdot \mtx{\Sigma}^{2(2q+1)} \cdot (\mtx{U}^\adj \mtx{P} \mtx{U}) }
    = \smnorm{}{ \mtx{P} (\mtx{MM}^\adj)^{2(2q+1)} \mtx{P} } \\
    = \smnorm{}{ \mtx{P} (\mtx{MM}^\adj)^q \mtx{M} \cdot \mtx{M}^\adj (\mtx{MM}^\adj)^q \mtx{P} }
    = \norm{ \mtx{P} (\mtx{MM}^\adj)^{q} \mtx{M} }^2.
\end{multline*}
We have used the unitary invariance of the spectral norm in the second and fourth relations.
The inequality~\eqref{eqn:power-claim} applies because $\mtx{U}^\adj \mtx{PU}$
is an orthogonal projector.  Take a square root to finish the argument.

Now, we turn to the claim~\eqref{eqn:power-claim}.  This relation follows immediately
from~\cite[Thm.~IX.2.10]{Bha97:Matrix-Analysis}, but we offer a direct argument
based on more elementary considerations.  Let $\vct{x}$ be a unit vector at which
$$
\vct{x}^\adj (\mtx{RDR}) \vct{x} = \norm{\mtx{RDR}}.
$$
We must have $\mtx{R}\vct{x} = \vct{x}$.  Otherwise, $\norm{ \mtx{R}\vct{x} } < 1$ because $\mtx{R}$
is an orthogonal projector, which implies that the unit vector
$\vct{y} = \mtx{R}\vct{x} / \norm{ \mtx{R}\vct{x} }$ verifies
$$
\vct{y}^\adj (\mtx{RDR}) \vct{y}
    = \frac{(\mtx{R}\vct{x})^\adj (\mtx{RDR}) (\mtx{R}\vct{x})}{ \normsq{\mtx{R}\vct{x}} }
    = \frac{\vct{x}^\adj (\mtx{RDR}) \vct{x}}{ \normsq{\mtx{R}\vct{x}} }
    > \vct{x}^\adj (\mtx{RDR}) \vct{x}.
$$
Writing $x_j$ for the entries of $\vct{x}$ and $d_j$ for the diagonal entries of $\mtx{D}$, we find that
\begin{multline*}
\norm{\mtx{RDR}}^t = [\vct{x}^\adj (\mtx{RDR}) \vct{x}]^t = [\vct{x}^\adj \mtx{D} \vct{x}]^t
    = \left[ \sum\nolimits_j d_j x_j^2 \right]^t \\
    \leq \left[ \sum\nolimits_j d_j^t x_j^2 \right]
    = \vct{x}^\adj \mtx{D}^t \vct{x}
    = (\mtx{R} \vct{x})^\adj \mtx{D}^t (\mtx{R} \vct{x})
    \leq \norm{ \mtx{RD}^t\mtx{R} }.
\end{multline*}
The inequality is Jensen's, which applies because $\sum x_j^2 = 1$
and the function $z \mapsto \abs{z}^t$ is convex for $t \geq 1$.
\end{proof}

\section{Error bounds via linear algebra}
\label{sec:basic_err}

We are now prepared to develop a deterministic error analysis for the proto-algorithm
described in~\S\ref{sec:sketchofalgorithm}.
To begin, we must introduce some notation.  Afterward, we establish the key
error bound, which strengthens a result from the literature~\cite[Lem.~4.2]{BMD09:Improved-Approximation}.
Finally, we explain why the power method can be used to improve the
performance of the proto-algorithm.

\subsection{Setup}

Let $\mtx{A}$ be an $m \times n$ matrix that has a singular value
decomposition $\mtx{A} = \mtx{U\Sigma V}^{\adj}$, as described
in Section \ref{sec:SVD}.
Roughly speaking, the proto-algorithm tries to approximate
the subspace spanned by the first $k$ left singular vectors,
where $k$ is now a fixed number.  To perform the analysis,
it is appropriate to partition the singular value decomposition
as follows.
\begin{equation}
\label{eq:mugpart}
\begin{array}{@{}c@{}r@{}c@{}c@{}c@{}c@{}c}
        && k & n - k && n & \\
    \mtx{A} = \mtx{U} &\left. \begin{array}{c} \\ \\ \end{array} \!\!\! \right[ &
    \begin{array}{c} \mtx{\Sigma}_1 \\ \phantom{\mtx{\Sigma}_2} \end{array} &
    \begin{array}{c} \phantom{\mtx{\Sigma}_1} \\ \mtx{\Sigma}_2 \end{array} &
    \left] \!\!\! \begin{array}{c} \\ \\ \end{array} \right. &
    \left[\begin{array}{c} \mtx{V}_{1}^{\adj} \\ \mtx{V}_{2}^{\adj}\end{array}\right]\,
    & \begin{array}{c} k \\ n - k \\ \end{array}
\end{array}
\end{equation}
The matrices $\mtx{\Sigma}_1$ and $\mtx{\Sigma}_2$ are square.
We will see that the left unitary factor $\mtx{U}$ does not play
a significant role in the analysis.




Let $\mtx{\Omega}$ be an $n \times \ell$ test matrix, where $\ell$ denotes the number of samples.  We assume only that $\ell \geq k$.
Decompose the test matrix in the
coordinate system determined by the right unitary factor of
$\mtx{A}$:
\begin{equation}
\label{eq:def_Omegaj}
\mtx{\Omega}_1 = \mtx{V}_1^\adj\, \mtx{\Omega} \quad\text{and}\quad
\mtx{\Omega}_2 = \mtx{V}_2^\adj\, \mtx{\Omega}.
\end{equation}
The error bound for the proto-algorithm depends critically on the
properties of the matrices $\mtx{\Omega}_{1}$ and $\mtx{\Omega}_{2}$.
With this notation, the sample matrix
$\mtx{Y}$ can be expressed as
\begin{equation*} \label{eqn:X-struct}
    \begin{array}{@{}c@{}c@{}c}
    & \ell & \\
    \mtx{Y} = \mtx{A}\mtx{\Omega} =
        \mtx{U} \left. \begin{array}{@{}c} \\ \\ \end{array} \right[ &
    \begin{array}{c} \mtx{\Sigma}_1 \mtx{\Omega}_1 \\
    \mtx{\Sigma}_2 \mtx{\Omega}_2 \end{array} &
    \left] \begin{array}{c} k \\ n - k \end{array} \right.
    \end{array}
\end{equation*}
It is a useful intuition that
the block $\mtx{\Sigma}_1 \mtx{\Omega}_1$ in \eqref{eqn:X-struct} reflects the
gross behavior of $\mtx{A}$, while the block $\mtx{\Sigma}_2 \mtx{\Omega}_2$
represents a perturbation.

\subsection{A deterministic error bound for the proto-algorithm}

The proto-algorithm constructs an orthonormal basis $\mtx{Q}$
for the range of the sample matrix $\mtx{Y}$, and our goal is
to quantify how well this basis captures the action of the
input $\mtx{A}$.  Since $\mtx{QQ}^\adj = \mtx{P}_{\mtx{Y}}$, the
challenge is to obtain bounds on the approximation error
$$
\triplenorm{ \mtx{A} - \mtx{QQ}^\adj \mtx{A} }
    =\triplenorm{ (\Id - \mtx{P}_{\mtx{Y}}) \mtx{A} }.
$$
The following theorem shows that the behavior of the proto-algorithm
depends on the interaction between the test matrix and the right singular
vectors of the input matrix, as well as the singular spectrum
of the input matrix.



\lsp

\begin{theorem}[Deterministic error bound] \label{thm:main-error-bd} 
Let $\mtx{A}$ be an $m\times n$ matrix with singular value decomposition
$\mtx{A} = \mtx{U\Sigma V}^\adj$, and fix $k \geq 0$.  Choose a
test matrix $\mtx{\Omega}$, and construct the sample matrix $\mtx{Y}
= \mtx{A\Omega}$. Partition $\mtx{\Sigma}$ as specified
in \eqref{eq:mugpart}, and define $\mtx{\Omega}_1$ and $\mtx{\Omega}_2$
via \eqref{eq:def_Omegaj}.  Assuming that $\mtx{\Omega}_1$ has
full row rank, the approximation error satisfies
\begin{equation}
\label{eq:main-error-bd}
\triplenorm{ (\Id - \mtx{P}_{\mtx{Y}}) \mtx{A} }^2
    \leq \triplenorm{\mtx{\Sigma}_2}^2 + \smtriplenorm{\mtx{\Sigma}_2\mtx{\Omega}_2 \mtx{\Omega}_1^\psinv}^2,
\end{equation}
where $\triplenorm{\cdot}$ denotes either the spectral norm or the
Frobenius norm.
\end{theorem}




\lsp

Theorem~\ref{thm:main-error-bd} sharpens the result~\cite[Lem.~2]{BMD09:Improved-Approximation}, which lacks the squares present in~\eqref{eq:main-error-bd}.  This refinement yields slightly better error estimates than the earlier bound, and it has consequences for the probabilistic behavior of the error when the test matrix $\mtx{\Omega}$ is random.  The proof here is different in spirit from the earlier analysis; our argument is inspired by the perturbation theory of orthogonal projectors~\cite{Ste77:Perturbation-Pseudoinverse}.

%


\lsp

\begin{proof}
We establish the bound for the spectral-norm error.  The bound for
the Frobenius-norm error follows from an analogous argument that
is slightly easier.


Let us begin with some preliminary simplifications.
First, we argue that the left unitary factor $\mtx{U}$ plays no
essential role in the argument.  In effect, we execute
the proof for an auxiliary input matrix $\widetilde{\mtx{A}}$ and
an associated sample matrix $\widetilde{\mtx{Y}}$ defined by
\begin{equation} \label{eqn:aux-matrices}
\widetilde{\mtx{A}} = \mtx{U}^\adj \mtx{A} =
	\begin{bmatrix} \mtx{\Sigma}_1\mtx{V}_1^\adj \\
	\mtx{\Sigma}_2 \mtx{V}_2^\adj \end{bmatrix}
\quad\text{and}\quad
\widetilde{\mtx{Y}} = \widetilde{\mtx{A}} \mtx{\Omega} =
	\begin{bmatrix} \mtx{\Sigma}_1 \mtx{\Omega}_1 \\
	\mtx{\Sigma}_2 \mtx{\Omega}_2 \end{bmatrix}.
\end{equation}
Owing to the unitary invariance of the spectral norm and to
Proposition~\ref{prop:conj-proj}, we have the identity
\begin{equation} \label{eqn:A-to-A0}
\norm{(\Id - \mtx{P}_{\mtx{Y}}) \mtx{A}}
    = \smnorm{}{\mtx{U}^\adj (\Id - \mtx{P}_{\mtx{Y}}) \mtx{U} \widetilde{\mtx{A}}}
    = \smnorm{}{(\Id - \mtx{P}_{\mtx{U}^\adj \mtx{Y}}) \widetilde{\mtx{A}} }
    = \smnorm{}{(\Id - \mtx{P}_{\widetilde{\mtx{Y}}}) \widetilde{\mtx{A}} }.
\end{equation}
In view of~\eqref{eqn:A-to-A0}, it suffices to prove that
\begin{equation} \label{eqn:aux-error-bd}
\smnorm{}{(\Id - \mtx{P}_{\widetilde{\mtx{Y}}}) \widetilde{\mtx{A}} }
	\leq \smnorm{}{\mtx{\Sigma}_2}^2 + \smnorm{}{\mtx{\Sigma}_2\mtx{\Omega}_2 \mtx{\Omega}_1^\psinv}^2.
\end{equation}


Second, we assume that the number $k$ is chosen so the diagonal entries of
$\mtx{\Sigma}_1$ are strictly positive.  Suppose not.  Then $\mtx{\Sigma}_2$
is zero because of the ordering of the singular values.  As a consequence,
$$
\range(\widetilde{\mtx{A}})
	= \range \begin{bmatrix} \mtx{\Sigma}_1 \mtx{V}_1^\adj \\ \mtx{0} \end{bmatrix}
	= \range \begin{bmatrix} \mtx{\Sigma}_1 \mtx{\Omega}_1 \\ \mtx{0} \end{bmatrix}
	= \range( \widetilde{\mtx{Y}} ).
$$
This calculation uses the decompositions presented in~\eqref{eqn:aux-matrices}, as well as the fact that both $\mtx{V}_1^\adj$ and $\mtx{\Omega}_1$ have full row rank.
We conclude that
$$
	\smnorm{}{ (\Id - \mtx{P}_{\widetilde{\mtx{Y}}}) \widetilde{\mtx{A}} }
	= 0,
$$
so the error bound~\eqref{eqn:aux-error-bd} holds trivially.  (In fact, both sides are zero.)


The main argument is based on ideas from perturbation theory.  To illustrate the concept, we start with a matrix related to $\widetilde{\mtx{Y}}$:
$$
\begin{array}{@{}c@{}c@{}c}
& \ell & \\
    \mtx{W} = \left. \begin{array}{@{}c} \\ \\ \end{array} \right[ &
    \begin{array}{c} \mtx{\Sigma}_1 \mtx{\Omega}_1 \\
    \mtx{0} \end{array} &
    \left] \begin{array}{c} k \\ n - k \end{array} \right.
    \end{array}
$$
The matrix $\mtx{W}$ has the same range as a related matrix formed by ``flattening out'' the spectrum of the top block.  Indeed, since $\mtx{\Sigma}_1 \mtx{\Omega}_1$ has full row rank,
$$
\begin{array}{@{}c@{}c@{}c}
& k & \\
    \range(\mtx{W}) = \range \left. \begin{array}{@{}c} \\ \\ \end{array} \right[ &
    \begin{array}{c} \Id \\
    \mtx{0} \end{array} &
    \left] \begin{array}{c} k \\ n - k \end{array} \right.
    \end{array}
$$
The matrix on the right-hand side has full column rank, so it is legal to apply the formula
~\eqref{eqn:orth-proj-form} for an orthogonal projector, which immediately yields
\begin{equation} \label{eqn:IPW}
\mtx{P}_{\mtx{W}}
	= \begin{bmatrix} \Id & \mtx{0} \\ \mtx{0} & \mtx{0} \end{bmatrix}
\quad\text{and}\quad
\Id - \mtx{P}_{\mtx{W}}
	= \begin{bmatrix} \mtx{0} & \mtx{0} \\ \mtx{0} & \Id \end{bmatrix}.
\end{equation}
In words, the range of $\mtx{W}$ aligns with the first $k$ coordinates, which span the same subspace as the first $k$ left singular vectors of the auxiliary input matrix $\widetilde{\mtx{A}}$.  Therefore, $\range(\mtx{W})$ captures the action of $\widetilde{\mtx{A}}$, which is what we wanted from $\range(\widetilde{\mtx{Y}})$.

We treat the auxiliary sample matrix $\widetilde{\mtx{Y}}$ as a perturbation of $\mtx{W}$, and we hope that their ranges are close to each other.  To make the comparison rigorous, let us emulate the arguments outlined in the last paragraph.
Referring to the display~\eqref{eqn:aux-matrices}, we flatten out the top block of $\widetilde{\mtx{Y}}$ to obtain the matrix
\begin{equation} \label{eqn:ZF-def}
\mtx{Z} = \widetilde{\mtx{Y}} \cdot \mtx{\Omega}_1^\psinv \mtx{\Sigma}_1^{-1}
	= \begin{bmatrix} \Id \\ \mtx{F} \end{bmatrix}
\quad\text{where}\quad
\mtx{F} = \mtx{\Sigma}_2 \mtx{\Omega}_2 \mtx{\Omega}_1^\psinv \mtx{\Sigma}_1^{-1}.
\end{equation}

Let us return to the error bound~\eqref{eqn:aux-error-bd}.  The construction~\eqref{eqn:ZF-def} ensures that $\range(\mtx{Z}) \subset \range(\widetilde{\mtx{Y}})$, so Proposition~\ref{prop:proj-range} implies that the error satisfies
$$
\smnorm{}{ (\Id - \mtx{P}_{\widetilde{\mtx{Y}}}) \widetilde{\mtx{A}} }
	\leq \smnorm{}{ (\Id - \mtx{P}_{\mtx{Z}}) \widetilde{\mtx{A}} }.
$$
Squaring this relation, we obtain
\begin{equation} \label{eqn:proj-norm-sq}
\smnorm{}{(\Id - \mtx{P}_{\widetilde{\mtx{Y}}}) \widetilde{\mtx{A}}}^2
    \leq \smnorm{}{(\Id - \mtx{P}_{\mtx{Z}})\widetilde{\mtx{A}}}^2
    = \smnorm{}{ \widetilde{\mtx{A}}^\adj (\Id - \mtx{P}_{\mtx{Z}}) \widetilde{\mtx{A}} }
    = \norm{ \mtx{\Sigma}^\adj (\Id - \mtx{P}_{\mtx{Z}}) \mtx{\Sigma} }.
\end{equation}
The last identity follows from the definition $\widetilde{\mtx{A}} = \mtx{\Sigma} \mtx{V}^\adj$ and the unitary invariance of the spectral norm.  Therefore, we can complete the proof of~\eqref{eqn:aux-error-bd} by producing a suitable bound for the right-hand side of \eqref{eqn:proj-norm-sq}.

To continue, we need a detailed representation of the projector $\Id
- \mtx{P}_{\mtx{Z}}$.  The construction~\eqref{eqn:ZF-def} ensures that
$\mtx{Z}$ has full column rank, so we can apply the formula~\eqref{eqn:orth-proj-form}
for an orthogonal projector to see that
$$
\mtx{P}_{\mtx{Z}} = \mtx{Z} (\mtx{Z}^\adj \mtx{Z})^{-1}
\mtx{Z}^\adj
	= \begin{bmatrix} \Id \\ \mtx{F} \end{bmatrix}
        (\Id + \mtx{F}^\adj\mtx{F})^{-1}
        \begin{bmatrix} \Id \\ \mtx{F} \end{bmatrix}^\adj.
$$
Expanding this expression, we determine that the complementary
projector satisfies
\begin{equation}
\Id - \mtx{P}_{\mtx{Z}}
    = \begin{bmatrix} \Id - (\Id + \mtx{F}^\adj \mtx{F})^{-1}
        & -(\Id + \mtx{F}^\adj \mtx{F})^{-1} \mtx{F}^\adj \\
        - \mtx{F} (\Id + \mtx{F}^\adj \mtx{F})^{-1}
        & \Id - \mtx{F}(\Id + \mtx{F}^\adj\mtx{F})^{-1} \mtx{F}^\adj
        \end{bmatrix}. \label{eqn:block-proj}
\end{equation}
The partitioning here conforms with the partitioning of
$\mtx{\Sigma}$.  When we conjugate the matrix by
$\mtx{\Sigma}$, copies of $\mtx{\Sigma}_1^{-1}$, presently hidden
in the top-left block, will cancel to happy effect.


The latter point may not seem obvious, owing to the complicated
form of \eqref{eqn:block-proj}.  In reality, the block matrix is
less fearsome than it looks.  Proposition~\ref{prop:perturb-inv},
on the perturbation of inverses, shows that the top-left block
verifies
$$
\Id - (\Id + \mtx{F}^\adj \mtx{F})^{-1} \psdle \mtx{F}^\adj
\mtx{F}.
$$
The bottom-right block satisfies
$$
\Id - \mtx{F}(\Id + \mtx{F}^\adj \mtx{F})^{-1} \mtx{F}^\adj \psdle
\Id
$$
because the conjugation rule guarantees that $\mtx{F}(\Id +
\mtx{F}^\adj \mtx{F})^{-1} \mtx{F}^\adj \psdge \mtx{0}$.
We abbreviate the off-diagonal blocks with the symbol
$\mtx{B} = -(\Id + \mtx{F}^\adj
\mtx{F})^{-1}\mtx{F}^\adj$. In summary,
$$
\Id - \mtx{P}_{\mtx{Z}}
    \psdle \begin{bmatrix} \mtx{F}^\adj \mtx{F} & \mtx{B} \\
        \mtx{B}^\adj & \Id \end{bmatrix}.
$$
This relation exposes the key structural properties of the projector.
Compare this relation with the expression~\eqref{eqn:IPW}
for the ``ideal'' projector $\Id - \mtx{P}_{\mtx{W}}$.

Moving toward the estimate required by~\eqref{eqn:proj-norm-sq}, we conjugate the last relation
by $\mtx{\Sigma}$ to obtain
$$
\mtx{\Sigma}^\adj (\Id - \mtx{P}_{\mtx{Z}}) \mtx{\Sigma}
    \psdle \begin{bmatrix} \mtx{\Sigma}_1^\adj \mtx{F}^\adj \mtx{F} \mtx{\Sigma}_1
        & \mtx{\Sigma}_1^\adj \mtx{B} \mtx{\Sigma}_2 \\
        \mtx{\Sigma}_2^\adj \mtx{B}^\adj \mtx{\Sigma}_1
        & \mtx{\Sigma}_2^\adj \mtx{\Sigma}_2 \end{bmatrix}.
$$
The conjugation rule demonstrates that the matrix on the left-hand
side is psd, so the matrix on the right-hand side is too.
Proposition~\ref{prop:spec-norm-sum} results in the norm bound
$$
\norm{ \mtx{\Sigma}^\adj (\Id - \mtx{P}_{\mtx{Z}}) \mtx{\Sigma} }
    \leq \norm{ \mtx{\Sigma}_1^\adj \mtx{F}^\adj \mtx{F} \mtx{\Sigma}_1 }
        + \norm{ \mtx{\Sigma}_2^\adj \mtx{\Sigma}_2 }
    = \normsq{\mtx{F} \mtx{\Sigma}_1} + \normsq{ \mtx{\Sigma}_2 }.
$$
Recall that $\mtx{F} = \mtx{\Sigma}_2 \mtx{\Omega}_2
\mtx{\Omega}_1^\psinv \mtx{\Sigma}_1^{-1}$, so the factor
$\mtx{\Sigma}_1$ cancels neatly.
Therefore,
$$
\norm{ \mtx{\Sigma}^\adj(\Id - \mtx{P}_{\mtx{Z}}) \mtx{\Sigma}}
    \leq \smnorm{}{\mtx{\Sigma}_2 \mtx{\Omega}_2 \mtx{\Omega}_1^\psinv }^2 + \normsq{\mtx{\Sigma}_2}.
$$
Finally, introduce the latter inequality into~\eqref{eqn:proj-norm-sq} to complete the proof.
\end{proof}

\subsection{Analysis of the power scheme}

Theorem~\ref{thm:main-error-bd} suggests that the performance of
the proto-algorithm depends strongly on the relationship between the
large singular values of $\mtx{A}$ listed in $\mtx{\Sigma}_1$
and the small singular values listed in $\mtx{\Sigma}_2$.
When a substantial proportion of the mass of $\mtx{A}$ appears in the small
singular values, the constructed basis $\mtx{Q}$ may have low accuracy.  Conversely,
when the large singular values dominate, it is much easier to identify a
good low-rank basis.

To improve the performance of the proto-algorithm, we can
run it with a closely related input matrix whose singular values
decay more rapidly~\cite{Gu-personal,tygert_szlam}.
Fix a positive integer $q$, and set
$$
\mtx{B} = (\mtx{AA}^\adj)^{q} \mtx{A}
    = \mtx{U}\mtx{\Sigma}^{2q+1} \mtx{V}^\adj.
$$
We apply the proto-algorithm to $\mtx{B}$, which generates a sample
matrix $\mtx{Z} = \mtx{B\Omega}$ and constructs a basis $\mtx{Q}$
for the range of $\mtx{Z}$.
Section~\ref{sec:powerscheme} elaborates on the implementation details,
and describes a reformulation that sometimes improves the accuracy when
the scheme is executed in finite-precision arithmetic.
The following result describes how well
we can approximate the \emph{original} matrix $\mtx{A}$ within the
range of $\mtx{Z}$.

\lsp

\begin{theorem}[Power scheme] \label{thm:power-method}
Let $\mtx{A}$ be an $m\times n$ matrix, and let $\mtx{\Omega}$ be an $n\times \ell$
matrix. Fix a nonnegative integer $q$, form $\mtx{B} = {(\mtx{A}^\adj \mtx{A})}^q \mtx{A}$,
and compute the sample matrix $\mtx{Z} = \mtx{B\Omega}$.  Then
$$
\norm{ (\Id - \mtx{P}_{\mtx{Z}}) \mtx{A} }
    \leq \norm{ (\Id - \mtx{P}_{\mtx{Z}}) \mtx{B} }^{1/(2q+1)}.
$$
\end{theorem}

\begin{proof}
We determine that
$$
\norm{ (\Id - \mtx{P}_{\mtx{Z}}) \mtx{A} }
    \leq \norm{ (\Id - \mtx{P}_{\mtx{Z}}) {(\mtx{AA}^\adj)}^q \mtx{A} }^{1/(2q+1)}
    = \norm{ (\Id - \mtx{P}_{\mtx{Z}}) \mtx{B} }^{1/(2q+1)}
$$
as a direct consequence of Proposition~\ref{prop:proj-power}.
\end{proof}

\lsp


Let us illustrate how the power scheme interacts with the main
error bound~\eqref{eq:main-error-bd}.
Let $\sigma_{k+1}$ denote the $(k+1)$th singular value of $\mtx{A}$.
First, suppose we approximate $\mtx{A}$ in the range of the sample matrix $\mtx{Y} = \mtx{A\Omega}$.
Since $\norm{\mtx{\Sigma}_2} = \sigma_{k+1}$, Theorem~\ref{thm:main-error-bd} implies that
\begin{equation} \label{eq:simple}
\norm{ (\Id - \mtx{P}_{\mtx{Y}}) \mtx{A} }
    \leq \left( 1 + \smnorm{}{\mtx{\Omega}_2 \mtx{\Omega}_1^\psinv}^2 \right)^{1/2}
    \sigma_{k+1}.
\end{equation}
Now, define $\mtx{B} = (\mtx{AA}^\adj)^q \mtx{A}$, and suppose
we approximate $\mtx{A}$ within the range of the sample matrix
$\mtx{Z} = \mtx{B\Omega}$.  Together, Theorem~\ref{thm:power-method}
and Theorem~\ref{thm:main-error-bd} imply that
$$
\norm{ (\Id - \mtx{P}_{\mtx{Z}}) \mtx{A} }
    \leq \norm{ (\Id - \mtx{P}_{\mtx{Z}}) \mtx{B} }^{1/(2q+1)}
    \leq \left( 1 + \smnorm{}{\mtx{\Omega}_2 \mtx{\Omega}_1^\psinv}^2 \right)^{1/(4q+2)}
        \sigma_{k+1}
$$
because $\sigma_{k+1}^{2q+1}$ is the $(k+1)$th singular value of $\mtx{B}$.
In effect, the power scheme drives down the suboptimality of the bound \eqref{eq:simple}
exponentially fast as the power $q$ increases.  In principle, we can make the extra
factor as close to one as we like, although this increases the cost of the algorithm.

\subsection{Analysis of truncated SVD} \label{sec:truncation-analysis}


Finally, let us study the truncated SVD described in Remark~\ref{rem:truncation}.
Suppose that we approximate the input matrix $\mtx{A}$ inside the range of the sample matrix $\mtx{Z}$.
In essence, the truncation step computes a best rank-$k$ approximation $\widehat{\mtx{A}}_{(k)}$ of the
compressed matrix $\mtx{P}_{\mtx{Z}} \mtx{A}$.  The next result provides a simple error bound for this method; this argument was proposed by Ming Gu.

\lsp

\begin{theorem}[Analysis of Truncated SVD] \label{thm:truncation}
Let $\mtx{A}$ be an $m \times n$ matrix with singular values $\sigma_1 \geq \sigma_2 \geq \sigma_3 \geq \dots$,
and let $\mtx{Z}$ be an $m \times \ell$ matrix, where $\ell \geq k$.
Suppose that $\widehat{\mtx{A}}_{(k)}$ is a best rank-$k$ approximation of $\mtx{P}_{\mtx{Z}} \mtx{A}$ with respect to the spectral norm.  Then
$$
\smnorm{}{ \mtx{A} - \widehat{\mtx{A}}_{(k)} }
	\leq \sigma_{k+1} + \norm{(\Id - \mtx{P}_{\mtx{Z}}) \mtx{A}}.
$$
\end{theorem}


\begin{proof}
Apply the triangle inequality to split the error into two components.
\begin{equation} \label{eqn:rand-pca-1}
\smnorm{}{ \mtx{A} - \widehat{\mtx{A}}_k }
	\leq \smnorm{}{ \mtx{A} - \mtx{P}_{\mtx{Z}} \mtx{A} }
	+ \smnorm{}{ \mtx{P}_{\mtx{Z}} \mtx{A} - \widehat{\mtx{A}}_{(k)} }.
\end{equation}
We have already developed a detailed theory for estimating the first term.  To analyze the second term, we introduce a best rank-$k$ approximation $\mtx{A}_{(k)}$ of the matrix $\mtx{A}$.  Note that
$$
\smnorm{}{ \mtx{P}_{\mtx{Z}} \mtx{A} - \widehat{\mtx{A}}_{(k)} }
	\leq \smnorm{}{ \mtx{P}_{\mtx{Z}} \mtx{A} - \mtx{P}_{\mtx{Z}} \mtx{A}_{(k)} }
$$
because $\widehat{\mtx{A}}_{(k)}$ is a best rank-$k$ approximation to the matrix $\mtx{P}_{\mtx{Z}} \mtx{A}$, whereas $\mtx{P}_{\mtx{Z}} \mtx{A}_{(k)}$ is an undistinguished rank-$k$ matrix.  It follows that
\begin{equation} \label{eqn:rand-pca-2}
\smnorm{}{ \mtx{P}_{\mtx{Z}} \mtx{A} - \widehat{\mtx{A}}_{(k)} }
	\leq \smnorm{}{ \mtx{P}_{\mtx{Z}} (\mtx{A} - \mtx{A}_{(k)}) }
	\leq \smnorm{}{ \mtx{A} - \mtx{A}_{(k)} }
	= \sigma_{k+1}.
\end{equation}
The second inequality holds because the orthogonal projector is a contraction; the last identity follows from Mirsky's theorem~\cite{Mir60:Symmetric-Gauge}.  Combine~\eqref{eqn:rand-pca-1} and~\eqref{eqn:rand-pca-2} to reach the main result.
\end{proof}

\lsp

\begin{remark} \rm
In the randomized setting, the truncation step appears to be less damaging than the error bound of Theorem~\ref{thm:truncation} suggests, but we currently lack a complete theoretical understanding of its behavior.
\end{remark}

\lsp

\section{Gaussian test matrices}
\label{sec:gaussians}

The error bound in Theorem~\ref{thm:main-error-bd} shows that the
performance of the proto-algorithm depends on the interaction
between the test matrix $\mtx{\Omega}$ and the right singular
vectors of the input matrix $\mtx{A}$.  Algorithm~\ref{alg:basic}
is a particularly simple version of the proto-algorithm that
draws the test matrix according to the standard
Gaussian distribution.  The literature contains a wealth of
information about these matrices, which allows us to perform
a very precise error analysis.




We focus on the real case in this section.  Analogous results hold
in the complex case, where the algorithm even exhibits superior
performance.

\subsection{Technical background}

A \term{standard Gaussian matrix} is a random matrix whose
entries are independent standard normal variables.
The distribution of a standard Gaussian matrix is rotationally invariant:
If $\mtx{U}$ and $\mtx{V}$ are orthonormal matrices, then $\mtx{U}^\adj
\mtx{G} \mtx{V}$ also has the standard Gaussian distribution.

Our analysis requires detailed information about the properties of Gaussian
matrices.  In particular, we must understand how the norm of a Gaussian
matrix and its pseudoinverse vary.  We summarize the relevant results
and citations here, reserving the details for Appendix~\ref{app:gauss}.


\lsp

\begin{proposition}[Expected norm of a scaled Gaussian matrix] \label{prop:scaled-gauss}
Fix matrices $\mtx{S}, \mtx{T}$, and draw a standard Gaussian matrix $\mtx{G}$.  Then
\begin{gather}
\left( \Expect \fnormsq{ \mtx{SGT} } \right)^{1/2}
    = \fnorm{\mtx{S}} \fnorm{\mtx{T}}
    \quad\text{and}
    \label{eqn:avg-fnormsq} \\
\Expect \norm{ \mtx{SGT} }
    \leq \norm{\mtx{S}} \fnorm{\mtx{T}} + \fnorm{\mtx{S}} \norm{\mtx{T}}.
    \label{eqn:avg-specnorm}
\end{gather}
\end{proposition}


The identity~\eqref{eqn:avg-fnormsq} follows from a direct
calculation.  The second bound~\eqref{eqn:avg-specnorm} relies on
methods developed by
Gordon~\cite{Gor85:Some-Inequalities,Gor88:Gaussian-Processes}.
See Propositions~\ref{prop:gauss-frob-expect}
and~\ref{prop:gauss-spec-expect}.


\lsp

\begin{proposition}[Expected norm of a pseudo-inverted Gaussian matrix] \label{prop:gauss-inv-expect}
Draw a $k \times (k + p)$ standard Gaussian matrix $\mtx{G}$ with $k \geq 2$ and $p \geq 2$.  Then
\begin{gather}
\left( \Expect \fnormsq{ \mtx{G}^\psinv } \right)^{1/2} = \sqrt{\frac{k}{p-1}}
	\quad\text{and}
    \label{eqn:avg-inv-fnormsq} \\
\Expect \norm{ \mtx{G}^\psinv } \leq \frac{\econst\sqrt{k+p}}{p}
    \label{eqn:avg-inv-specnorm}.
\end{gather}
\end{proposition}

The first identity is a standard result from multivariate
statistics~\cite[p.~96]{Mui82:Aspects-Multivariate}. The second
follows from work of Chen and
Dongarra~\cite{CD05:Condition-Numbers}.  See
Proposition~\ref{prop:inv-gauss-spec-expect}
and~\ref{prop:inv-gauss-frob-expect}.

To study the probability that Algorithm~\ref{alg:basic} produces a large error,
we rely on tail bounds for functions of Gaussian matrices. The next
proposition rephrases a well-known result on concentration of
measure~\cite[Thm.~4.5.7]{Bog98:Gaussian-Measures}. See
also~\cite[\S1.1]{LT91:Probability-Banach} and
\cite[\S5.1]{Led01:Concentration-Measure}.

\lsp

\begin{proposition}[Concentration for functions of a Gaussian matrix] \label{prop:gauss-tail}
Suppose that $h$ is a Lipschitz function on matrices:
$$
\abs{ h(\mtx{X}) - h(\mtx{Y}) } \leq L \fnorm{ \mtx{X} - \mtx{Y} }
\quad\text{for all $\mtx{X}, \mtx{Y}$.}
$$
Draw a standard Gaussian matrix $\mtx{G}$.  Then
$$
\Prob{ h(\mtx{G}) \geq \Expect h(\mtx{G}) + Lt } \leq \econst^{-t^2/2}.
$$
\end{proposition}

Finally, we state some large deviation bounds for the norm of a pseudo-inverted
Gaussian matrix.

\lsp

\begin{proposition}[Norm bounds for a pseudo-inverted Gaussian matrix]
\label{prop:gauss-inv-tails}
Let $\mtx{G}$ be a $k \times (k + p)$ Gaussian matrix where $p \geq 4$.  For all
$t \geq 1$,
\begin{gather}
\Prob{ \smnorm{\rm F}{\mtx{G}^\psinv} \geq \sqrt{\frac{12 k}{p}} \cdot t } \leq 4 t^{-p}
    \quad\text{and}
    \label{eqn:tail-inv-fnormsq} \\
\Prob{ \smnorm{}{\mtx{G}^\psinv} \geq \frac{\econst\sqrt{k+p}}{p+1} \cdot t } \leq t^{-(p+1)}.
    \label{eqn:tail-inv-specnorm}
\end{gather}
\end{proposition}


Compare these estimates with Proposition~\ref{prop:gauss-inv-expect}.
It seems that~\eqref{eqn:tail-inv-fnormsq} is new;
we were unable to find a comparable analysis in the random matrix
literature.  Although the form of~\eqref{eqn:tail-inv-fnormsq} is not optimal,
it allows us to produce more transparent results than
a fully detailed estimate.  The bound~\eqref{eqn:tail-inv-specnorm}
essentially appears in the work of Chen and Dongarra~\cite{CD05:Condition-Numbers}.
See Propositions~\ref{prop:inv-gauss-spec-tail} and Theorem~\ref{thm:inv-gauss-frob-tail}
for more information.

\subsection{Average-case analysis of Algorithm~\ref{alg:basic}}
    \label{sec:gauss-avg-case}

We separate our analysis into two pieces.
First, we present information about expected values.
In the next subsection, we describe bounds on the probability of a large deviation.


We begin with the simplest result, which provides an
estimate for the expected approximation error in the Frobenius
norm.  All proofs are postponed to the end of the section.

\lsp

\begin{theorem}[Average Frobenius error] \label{thm:avg-frob-error-gauss}
Suppose that $\mtx{A}$ is a \emph{real} $m \times n$ matrix with
singular values $\sigma_1 \geq \sigma_2 \geq \sigma_3 \geq \dots$.
Choose a target rank $k \geq 2$ and an oversampling parameter $p \geq 2$,
where $k + p \leq \min\{m,n\}$.  Draw an $n \times (k+p)$ standard
Gaussian matrix $\mtx{\Omega}$, and construct the sample matrix $\mtx{Y} =
\mtx{A\Omega}$.  Then the expected approximation error
$$
\Expect \fnorm{(\Id - \mtx{P}_{\mtx{Y}}) \mtx{A}}
    \leq \left( 1 + \frac{k}{p-1} \right)^{1/2}
    \left( \sum\nolimits_{j>k} \sigma_j^2 \right)^{1/2}.
$$
\end{theorem}


This theorem predicts several intriguing behaviors
of Algorithm~\ref{alg:basic}.
The Eckart--Young theorem~\cite{EY36:Approximation-One} shows that
$(\sum\nolimits_{j > k} \sigma_j^2 )^{1/2}$ is the minimal Frobenius-norm
error when approximating $\mtx{A}$ with a rank-$k$ matrix.  This quantity
is the appropriate benchmark for the performance of the algorithm.
If the small singular values of $\mtx{A}$ are very flat,
the series may be as large as $\sigma_{k+1} \sqrt{\min\{m,n\} - k}$.  On the
other hand, when the singular values exhibit some decay, the error may
be on the same order as $\sigma_{k+1}$.

The error bound always exceeds this baseline error,
but it may be polynomially larger, depending on the ratio between the
target rank $k$ and the oversampling parameter $p$.  For $p$ small
(say, less than five), the error is somewhat variable because
the small singular values of a nearly square Gaussian matrix are very
unstable.  As the oversampling increases, the performance improves quickly.
When $p \sim k$, the error is already within a constant factor of the baseline.





The error bound for the spectral norm is somewhat more
complicated, but it reveals some interesting new features.

\lsp

\begin{theorem}[Average spectral error] \label{thm:avg-spec-error-gauss}
Under the hypotheses of Theorem~\ref{thm:avg-frob-error-gauss},
$$
\Expect \norm{(\Id - \mtx{P}_{\mtx{Y}}) \mtx{A}}
    \leq \left(1 + \sqrt{\frac{k}{p-1}} \right) \sigma_{k+1}
        + \frac{\econst\sqrt{k+p}}{p}
        \left(\sum\nolimits_{j>k} \sigma_{j}^2 \right)^{1/2}.
$$
\end{theorem}


Mirsky~\cite{Mir60:Symmetric-Gauge} has shown that the quantity
$\sigma_{k+1}$ is the minimum spectral-norm error when
approximating $\mtx{A}$ with a rank-$k$ matrix, so
the first term in Theorem~\ref{thm:avg-spec-error-gauss} is analogous
with the error bound in Theorem~\ref{thm:avg-frob-error-gauss}.  The
second term represents a new phenomenon: we also pay for
the Frobenius-norm error in approximating $\mtx{A}$.
Note that, as the amount $p$ of oversampling increases, the
polynomial factor in the second term declines much more quickly
than the factor in the first term.  When $p \sim k$, the factor on the $\sigma_{k+1}$ term is constant, while the factor on the series has order $k^{-1/2}$

We also note that the bound in Theorem~\ref{thm:avg-spec-error-gauss}
implies
$$
\Expect \norm{(\Id - \mtx{P}_{\mtx{Y}}) \mtx{A}}
    \leq \left[ 1 + \sqrt{\frac{k}{p-1}}  + \frac{\econst\sqrt{k+p}}{p} \cdot \sqrt{\min\{m,n\} - k} \right]
        \sigma_{k+1},
$$
so the average spectral-norm error always lies within a small
polynomial factor of the baseline $\sigma_{k+1}$.

Let us continue with the proofs of these results.

\lsp

\begin{proof}[Theorem~\ref{thm:avg-frob-error-gauss}]
Let $\mtx{V}$ be the right unitary factor of $\mtx{A}$.  Partition
$\mtx{V} = [ \mtx{V}_1 \ | \ \mtx{V}_2 ]$ into blocks containing,
respectively, $k$ and $n - k$ columns.  Recall that
$$
\mtx{\Omega}_1 = \mtx{V}_1^\adj \mtx{\Omega} \quad\text{and}\quad
\mtx{\Omega}_2 = \mtx{V}_2^\adj \mtx{\Omega}.
$$
The Gaussian distribution is rotationally invariant, so
$\mtx{V}^\adj \mtx{\Omega}$ is also a standard Gaussian matrix.
Observe that $\mtx{\Omega}_1$ and $\mtx{\Omega}_2$ are
\emph{nonoverlapping} submatrices of $\mtx{V}^\adj \mtx{\Omega}$, so
these two matrices are not only standard Gaussian but also
stochastically independent.  Furthermore, the rows of a (fat)
Gaussian matrix are almost surely in general position, so
the $k \times (k + p)$ matrix $\mtx{\Omega}_1$ has full row rank
with probability one.

H{\"o}lder's inequality and Theorem~\ref{thm:main-error-bd} together imply
that
$$
\Expect \fnorm{ (\Id - \mtx{P}_{\mtx{Y}}) \mtx{A}}
    \leq \left( \Expect \fnormsq{(\Id - \mtx{P}_{\mtx{Y}}) \mtx{A}} \right)^{1/2}
    \leq \left( \smnorm{\rm F}{\mtx{\Sigma}_2}^2 + \Expect \smnorm{\rm F}{ \mtx{\Sigma}_2
                \mtx{\Omega}_2 \mtx{\Omega}_1^\psinv }^2 \right)^{1/2}.
$$
We compute this expectation by conditioning on the value of $\mtx{\Omega}_1$ and
applying Proposition~\ref{prop:scaled-gauss} to the scaled Gaussian matrix $\mtx{\Omega}_2$.
Thus,
\begin{multline*}
\Expect \smnorm{\rm F}{ \mtx{\Sigma}_2 \mtx{\Omega}_2 \mtx{\Omega}_1^\psinv }^2
    = \Expect \left( \Expect \left[ \smnorm{\rm F}{ \mtx{\Sigma}_2 \mtx{\Omega}_2
\mtx{\Omega}_1^\psinv }^2 \ \big\vert \ \mtx{\Omega}_1 \right] \right)
    = \Expect \left( \fnormsq{\mtx{\Sigma}_2} \smnorm{\rm F}{\mtx{\Omega}_1^\psinv}^2 \right) \\
    = \fnormsq{\mtx{\Sigma}_2} \cdot \Expect \smnorm{\rm F}{\mtx{\Omega}_1^\psinv}^2
    = \frac{k}{p-1} \cdot \fnormsq{\mtx{\Sigma}_2},
\end{multline*}
where the last expectation follows from relation~\eqref{eqn:avg-inv-fnormsq}
of Proposition~\ref{prop:gauss-inv-expect}.  In summary,
$$
\Expect \fnorm{ (\Id - \mtx{P}_{\mtx{Y}}) \mtx{A}}
    \leq \left(1 + \frac{k}{p-1}\right)^{1/2} \fnorm{ \mtx{\Sigma}_2 }.
$$
Observe that $\fnormsq{\mtx{\Sigma}_2} = \sum_{j>k} \sigma_j^2$ to complete the proof.
\end{proof}

\lsp

\begin{proof}[Theorem~\ref{thm:avg-spec-error-gauss}]
The argument is similar to the proof of Theorem~\ref{thm:avg-frob-error-gauss}.
First, Theorem~\ref{thm:main-error-bd} 
implies that
$$
\Expect \norm{ (\Id - \mtx{P}_{\mtx{Y}}) \mtx{A} }
    \leq \Expect \left( \normsq{\mtx{\Sigma}_2}
        + \smnorm{}{\mtx{\Sigma}_2 \mtx{\Omega}_2 \mtx{\Omega}_1^\psinv }^2 \right)^{1/2}
    \leq \norm{\mtx{\Sigma}_2} + \Expect \smnorm{}{ \mtx{\Sigma}_2 \mtx{\Omega}_2 \mtx{\Omega}_1^\psinv }.
$$
We condition on $\mtx{\Omega}_1$ and apply
Proposition~\ref{prop:scaled-gauss} to bound the expectation with
respect to $\mtx{\Omega}_2$.  Thus,
\begin{align*}
\Expect \smnorm{}{ \mtx{\Sigma}_2 \mtx{\Omega}_2
\mtx{\Omega}_1^\psinv }
    &\leq \Expect \left( \norm{\mtx{\Sigma}_2} \smnorm{\rm F}{\mtx{\Omega}_1^\psinv}
        + \fnorm{\mtx{\Sigma}_2} \smnorm{}{\mtx{\Omega}_1^\psinv} \right) \\
    &\leq \norm{\mtx{\Sigma}_2} \left( \Expect \smnorm{\rm F}{\mtx{\Omega}_1^\psinv}^2 \right)^{1/2}
        + \fnorm{\mtx{\Sigma}_2} \cdot \Expect \smnorm{}{\mtx{\Omega}_1^\psinv}.
\end{align*}
where the second relation requires H{\"o}lder's inequality.
Applying both parts of Proposition~\ref{prop:gauss-inv-expect}, we obtain
$$
\Expect \smnorm{}{ \mtx{\Sigma}_2 \mtx{\Omega}_2
\mtx{\Omega}_1^\psinv }
    \leq \sqrt{\frac{k}{p-1}} \norm{\mtx{\Sigma}_2}
    + \frac{\econst\sqrt{k+p}}{p} \fnorm{\mtx{\Sigma}_2}.
$$
Note that $\norm{\mtx{\Sigma}_2} = \sigma_{k+1}$ to wrap up.
\end{proof}

\subsection{Probabilistic error bounds for Algorithm~\ref{alg:basic}}
\label{sec:prob-failure}

We can develop tail bounds for the approximation error, which
demonstrate that the average performance of the algorithm
is representative of the actual performance.  We begin with
the Frobenius norm because the result is somewhat simpler.

\lsp

\begin{theorem}[Deviation bounds for the Frobenius error] \label{thm:tail-frob-error-gauss}
Frame the hypotheses of Theorem~\ref{thm:avg-frob-error-gauss}.
Assume further that $p \geq 4$.  For all $u, t \geq 1$,
$$
\fnorm{ (\Id - \mtx{P}_{\mtx{Y}}) \mtx{A} }
    \leq \left( 1 + t \cdot \sqrt{12k/p} \right)
    \left( \sum\nolimits_{j > k} \sigma_j^2 \right)^{1/2}
    + ut \cdot \frac{\econst\sqrt{k+p}}{p+1} \cdot \sigma_{k+1},
$$
with failure probability at most $5 t^{-p} + 2 \econst^{-u^2/2}$.
\end{theorem}

\lsp

To parse this theorem, observe that the first term in the error bound corresponds with
the expected approximation error in Theorem~\ref{thm:avg-frob-error-gauss}.  The second
term represents a deviation above the mean.






An analogous result holds for the spectral norm.


\lsp

\begin{theorem}[Deviation bounds for the spectral error] \label{thm:tail-spec-error-gauss}
Frame the hypotheses of Theorem~\ref{thm:avg-frob-error-gauss}.  Assume further that $p \geq 4$.
For all $u, t \geq 1$,
\begin{multline*}
\norm{ (\Id - \mtx{P}_{\mtx{Y}}) \mtx{A} } \\
    \leq \left[ \left( 1 + t \cdot \sqrt{12k/p} \right) \sigma_{k+1}
    + t \cdot \frac{\econst\sqrt{k+p}}{p+1} \left( \sum\nolimits_{j>k} \sigma_j^2 \right)^{1/2} \right]
    + ut \cdot \frac{\econst\sqrt{k+p}}{p+1} \sigma_{k+1},
\end{multline*}
with failure probability at most $5 t^{-p} + \econst^{-u^2/2}$.
\end{theorem}

\lsp

The bracket corresponds with the expected spectral-norm error while the remaining term represents
a deviation above the mean.  Neither the numerical constants nor the precise form of the bound
are optimal because of the slackness in Proposition~\ref{prop:gauss-inv-tails}.  Nevertheless,
the theorem gives a fairly good picture of what is actually happening.

We acknowledge that the current form of Theorem~\ref{thm:tail-spec-error-gauss} is complicated.
To produce more transparent results, we make appropriate selections for the parameters $u, t$
and bound the numerical constants.

\lsp

\begin{corollary}[Simplified deviation bounds for the spectral error] \label{cor:tail-spec-error-gauss}
Frame the hypotheses of Theorem~\ref{thm:avg-frob-error-gauss}, and assume further that $p \geq 4$.  Then
$$
\norm{ (\Id - \mtx{P}_{\mtx{Y}}) \mtx{A} }
    \leq \left( 1 + 17 \sqrt{1 + k/p} \right) \sigma_{k+1}
        + \frac{8\sqrt{k+p}}{p+1} \left( \sum\nolimits_{j > k} \sigma_j^2 \right)^{1/2},
$$
with failure probability at most $6\econst^{-p}$. Moreover, 
$$
\norm{ (\Id - \mtx{P}_{\mtx{Y}}) \mtx{A} }
    \leq \left( 1 + 8 \sqrt{(k + p) \cdot p \log p} \right) \sigma_{k+1}
        + 3 \sqrt{k+p} \left( \sum\nolimits_{j > k} \sigma_j^2 \right)^{1/2},
$$
with failure probability at most $6 p^{-p}$.
\end{corollary}



\lsp

\begin{proof}
The first part of the result follows from the choices
$t = \econst$ and $u = \sqrt{2p}$, and the second
emerges when $t = p$ and $u = \sqrt{2p\log p}$.
Another interesting parameter selection is
$t = p^{c/p}$ and $u = \sqrt{2c\log p}$,
which yields a failure probability $6p^{-c}$.
\end{proof}

\lsp

Corollary~\ref{cor:tail-spec-error-gauss} should be compared with~\cite[Obs.~4.4--4.5]{random1}.
Although our result contains sharper error estimates,
the failure probabilities are usually worse.
The error bound~\eqref{eq:intro_err_prob} presented in~\S\ref{sec:prototheorem}
follows after further simplification of the second bound from Corollary~\ref{cor:tail-spec-error-gauss}.



We continue with a proof of Theorem~\ref{thm:tail-spec-error-gauss}.
The same argument can be used to obtain a bound for the
Frobenius-norm error, but we omit a detailed account.

\lsp

\begin{proof}[Theorem~\ref{thm:tail-spec-error-gauss}]
Since $\mtx{\Omega}_1$ and $\mtx{\Omega}_2$ are independent from each other,
we can study how the error depends on the matrix $\mtx{\Omega}_2$ by conditioning
on the event that $\mtx{\Omega}_1$ is not too irregular.
To that end, we define a (parameterized) event on which the spectral and Frobenius
norms of the matrix $\mtx{\Omega}_1^\psinv$ are both controlled.  For $t \geq 1$, let
$$
E_t = \left\{ \mtx{\Omega}_1 :
    \smnorm{}{ \mtx{\Omega}_1^\psinv } \leq \frac{\econst\sqrt{k+p}}{p+1} \cdot t
    \quad\text{and}\quad
    \smnorm{\rm F}{ \mtx{\Omega}_1^\psinv } \leq \sqrt{\frac{12k}{p}} \cdot t \right\}.
$$
Invoking both parts of Proposition~\ref{prop:gauss-inv-tails}, we find that
$$
\Probe{E_t^c} \leq t^{-(p+1)} + 4 t^{-p} \leq 5 t^{-p}.
$$

Consider the function
$h( \mtx{X} ) = \smnorm{}{ \mtx{\Sigma}_2 \mtx{X} \mtx{\Omega}_1^\psinv }$.
We quickly compute its Lipschitz constant $L$ with the lower triangle inequality
and some standard norm estimates:
\begin{multline*}
\abs{ h(\mtx{X}) - h(\mtx{Y}) }
    \leq \smnorm{}{ \mtx{\Sigma}_2 (\mtx{X} - \mtx{Y}) \mtx{\Omega}_1^\psinv } \\
    \leq \norm{ \mtx{\Sigma}_2 } \norm{ \mtx{X} - \mtx{Y} } \smnorm{}{ \mtx{\Omega}_1^\psinv }
    \leq \norm{ \mtx{\Sigma}_2 } \smnorm{}{ \mtx{\Omega}_1^\psinv } \fnorm{ \mtx{X} - \mtx{Y} }.
\end{multline*}
Therefore, $L \leq \norm{\mtx{\Sigma}_2} \smnorm{}{\mtx{\Omega}_1^\psinv}$.
Relation~\eqref{eqn:avg-specnorm} of Proposition~\ref{prop:scaled-gauss} implies that
$$
\Expect[ h(\mtx{\Omega}_2) \; | \; \mtx{\Omega}_1 ]
    \leq \norm{\mtx{\Sigma}_2} \smnorm{\rm F}{ \mtx{\Omega}_1^\psinv }
        + \fnorm{\mtx{\Sigma}_2} \smnorm{}{ \mtx{\Omega}_1^\psinv }.
$$
Applying the concentration of measure inequality, Proposition~\ref{prop:gauss-tail}, conditionally
to the random variable
$h(\mtx{\Omega}_2) = \smnorm{}{ \mtx{\Sigma}_2 \mtx{\Omega}_2 \mtx{\Omega}_1^\psinv }$
results in
\begin{equation*}
\Prob{ \smnorm{}{\mtx{\Sigma}_2\mtx{\Omega}_2\mtx{\Omega}_1^\psinv}
    > \norm{\mtx{\Sigma}_2} \smnorm{\rm F}{ \mtx{\Omega}_1^\psinv }
        + \fnorm{\mtx{\Sigma}_2} \smnorm{}{ \mtx{\Omega}_1^\psinv }
    + \norm{\mtx{\Sigma}_2} \smnorm{}{\mtx{\Omega}_1^\psinv} \cdot u \ \big\vert \ E_t }
    \leq \econst^{-u^2/2}.
\end{equation*}
Under the event $E_t$, we have explicit bounds on the norms of $\mtx{\Omega}_1^\psinv$,
so
\begin{multline*}
\Prob{ \smnorm{}{\mtx{\Sigma}_2\mtx{\Omega}_2\mtx{\Omega}_1^\psinv}
    > \norm{\mtx{\Sigma}_2} \sqrt{\frac{12k}{p}} \cdot t
    + \fnorm{\mtx{\Sigma}_2} \frac{\econst\sqrt{k+p}}{p+1} \cdot t
    + \norm{\mtx{\Sigma}_2} \frac{\econst\sqrt{k+p}}{p+1} \cdot ut \ \bigg\vert \ E_t } \\
    \leq \econst^{-u^2/2}.
\end{multline*}
Use the fact $\Probe{E_t^c} \leq 5 t^{-p}$ to remove the conditioning.  Therefore,
\begin{multline*}
\Prob{ \smnorm{}{\mtx{\Sigma}_2\mtx{\Omega}_2\mtx{\Omega}_1^\psinv}
    > \norm{\mtx{\Sigma}_2} \sqrt{\frac{12k}{p}} \cdot t
    + \fnorm{\mtx{\Sigma}_2} \frac{\econst\sqrt{k+p}}{p+1} \cdot t
    + \norm{\mtx{\Sigma}_2} \frac{\econst\sqrt{k+p}}{p+1} \cdot ut } \\
    \leq 5 t^{-p} + \econst^{-u^2/2}.
\end{multline*}
Insert the expressions for the norms of $\mtx{\Sigma}_2$ into this result to complete
the probability bound.  Finally, introduce this estimate into the error bound from
Theorem~\ref{thm:main-error-bd}.
\end{proof}

\subsection{Analysis of the power scheme}
\label{sec:avg-power-method}

Theorem~\ref{thm:avg-spec-error-gauss} makes it clear that the
performance of the randomized approximation scheme, Algorithm~\ref{alg:basic},
depends heavily on the singular spectrum of the input matrix.
The power scheme outlined in Algorithm~\ref{alg:poweriteration} addresses this problem
by enhancing the decay of spectrum.
We can combine our analysis of Algorithm~\ref{alg:basic} with
Theorem~\ref{thm:power-method}
to obtain a detailed report on the behavior of the performance
of the power scheme using a Gaussian matrix.

\lsp

\begin{corollary}[Average spectral error for the power scheme] \label{cor:power-method-spec-gauss}
Frame the hypotheses of Theorem~\ref{thm:avg-frob-error-gauss}.
Define $\mtx{B} = {(\mtx{A}\mtx{A}^\adj)}^{q} \mtx{A}$ for a
nonnegative integer $q$, and construct the sample matrix $\mtx{Z} =
\mtx{B\Omega}$.  Then
$$
\Expect \norm{(\Id - \mtx{P}_{\mtx{Z}}) \mtx{A}}
    \leq \left[ \left( 1 + \sqrt{\frac{k}{p-1}} \right) \sigma_{k+1}^{2q+1}
    + \frac{\econst\sqrt{k+p}}{p} \left( \sum\nolimits_{j>k} \sigma_j^{2(2q+1)} \right)^{1/2} \right]^{1/(2q+1)}.
$$
\end{corollary}

\begin{proof}
By H{\"o}lder's inequality and Theorem~\ref{thm:power-method},
$$
\Expect \norm{(\Id - \mtx{P}_{\mtx{Z}}) \mtx{A}}
    \leq \left( \Expect \norm{(\Id - \mtx{P}_{\mtx{Z}}) \mtx{A}}^{2q+1} \right)^{1/(2q+1)}
    \leq \left( \Expect \norm{(\Id - \mtx{P}_{\mtx{Z}}) \mtx{B}} \right)^{1/(2q+1)}.
$$
Invoke Theorem~\ref{thm:avg-spec-error-gauss} to bound the
right-hand side, noting that $\sigma_{j}(\mtx{B}) = \sigma_j^{2q+1}$.
\end{proof}

\lsp

The true message of Corollary~\ref{cor:power-method-spec-gauss}
emerges if we bound the series using its largest term $\sigma_{k+1}^{4q+2}$
and draw the factor $\sigma_{k+1}$ out of the bracket:
\begin{equation*} \label{eqn:power-method-weak-bd}
\Expect \norm{(\Id - \mtx{P}_{\mtx{Z}}) \mtx{A}}
    \leq \left[ 1 + \sqrt{\frac{k}{p-1}}
    + \frac{\econst\sqrt{k+p}}{p} \cdot \sqrt{ \min\{m,n\} - k } \right]^{1/(2q+1)}
    \sigma_{k+1}.
\end{equation*}
In words, as we increase the exponent $q$, the power scheme drives
the extra factor in the error to one exponentially fast.
By the time $q \sim \log\left( \min\{m,n\} \right)$,
$$
\Expect \norm{(\Id - \mtx{P}_{\mtx{Z}}) \mtx{A}}
    \sim \sigma_{k+1},
$$
which is the baseline for the spectral norm.

In most situations, the error bound given by Corollary~\ref{cor:power-method-spec-gauss} is substantially better than the estimates discussed in the last paragraph.  For example, suppose that the tail singular values exhibit the decay profile
$$
\sigma_j \lesssim j^{(1+\eps)/(4q+2)}
\quad\text{for $j > k$ and $\eps > 0$}.
$$
Then the series in Corollary~\ref{cor:power-method-spec-gauss} is comparable with its largest term, which allows us to remove the dimensional factor $\min\{m,n\}$ from the error bound.




To obtain large deviation bounds for the performance of the power scheme,
simply combine Theorem~\ref{thm:power-method}
with Theorem~\ref{thm:tail-spec-error-gauss}.
We omit a detailed statement.


\lsp

\begin{remark} \rm
We lack an analogous theory for the Frobenius norm because
Theorem~\ref{thm:power-method} depends on Proposition~\ref{prop:proj-power},
which is not true for the Frobenius norm.  It is possible to obtain some results
by estimating the Frobenius norm in terms of the spectral norm.
\end{remark}

\lsp

\section{SRFT test matrices}
\label{sec:SRFTs}

Another way to implement the proto-algorithm from \S\ref{sec:sketchofalgorithm}
is to use a structured random matrix so that the matrix product in Step~2
can be performed quickly.  One type of structured random matrix that has
been proposed in the literature is the \term{subsampled random Fourier transform},
or SRFT, which we discussed in \S\ref{sec:ailonchazelle}.
In this section, we present bounds on the performance of the
proto-algorithm when it is implemented with an SRFT test matrix.
In contrast with the results
for Gaussian test matrices, the results in this section hold for both
real and complex input matrices.





\subsection{Construction and Properties}

Recall from~\S\ref{sec:ailonchazelle} that an SRFT is a
tall $n \times \ell$ matrix of the form $\mtx{\Omega} = \sqrt{n/\ell}
\cdot \mtx{DFR}^\adj$ where

\lsp

\begin{itemize}
\item   $\mtx{D}$ is a random $n \times n$ diagonal matrix whose
entries are independent and uniformly distributed on the complex
unit circle;

\item   $\mtx{F}$ is the $n \times n$ unitary discrete Fourier
transform; and

\item   $\mtx{R}$ is a random $\ell \times n$ matrix that
restricts an $n$-dimensional vector to $\ell$ coordinates, chosen
uniformly at random.
\end{itemize}

\lsp

\noindent
Up to scaling, an SRFT is just a section of a unitary matrix,
so it satisfies the norm identity $\norm{\mtx{\Omega}} = \sqrt{n/\ell}$.
The critical fact is that an appropriately designed SRFT approximately preserves
the geometry of an \emph{entire subspace of vectors}.

\lsp

\begin{theorem}[The SRFT preserves geometry] \label{thm:SRFT-spec-bd}
Fix an $n \times k$ orthonormal matrix $\mtx{V}$, and
draw an $n \times \ell$ SRFT matrix $\mtx{\Omega}$
where the parameter $\ell$
satisfies
$$
4 \left[ \sqrt{k} + \sqrt{8\log(kn)} \right]^{2} \log(k) \leq \ell \leq n.
$$
Then
$$
0.40 \leq \sigma_{k}(\mtx{V}^\adj \mtx{\Omega})
\quad\text{and}\quad
\sigma_{1}(\mtx{V}^\adj \mtx{\Omega}) \leq 1.48
$$
with failure probability at most $\bigO(k^{-1})$.
\end{theorem}



\lsp

In words, the kernel of an SRFT of dimension $\ell \sim k \log(k)$ is unlikely to intersect a fixed $k$-dimensional subspace. In contrast
with the Gaussian case, the logarithmic factor $\log(k)$ in the lower
bound on $\ell$ cannot generally be removed (Remark~\ref{rem:coupon}).

Theorem~\ref{thm:SRFT-spec-bd} follows from a straightforward variation of the argument in~\cite{Tro10:Improved-Analysis}, which establishes equivalent bounds for a real analog of the SRFT, called the \term{subsampled randomized Hadamard transform} (SRHT).  We omit further details.



\lsp

\begin{remark} \rm
For large problems, we can obtain better numerical constants~\cite[Thm.~3.2]{Tro10:Improved-Analysis}.
Fix a small, positive number $\iota$.  If $k \gg \log(n)$, then sampling
$$
\ell \geq (1 + \iota) \cdot k \log(k)
$$
coordinates is sufficient to ensure that $\sigma_k(\mtx{V}^\adj \mtx{\Omega}) \geq \iota$ with failure probability at most $\bigO(k^{-\cnst{c}\iota})$.  This sampling bound is essentially optimal because $(1 - \iota) \cdot k \log(k)$ samples are not adequate in the worst case; see Remark~\ref{rem:coupon}.
\end{remark}
\lsp

\begin{remark} \label{rem:coupon} \rm
The logarithmic factor in Theorem~\ref{thm:SRFT-spec-bd} is \emph{necessary} when the orthonormal matrix $\mtx{V}$ is particularly evil.  Let us describe an infinite family of worst-case examples.  Fix an integer $k$, and let $n = k^2$.  Form an
$n \times k$ orthonormal matrix $\mtx{V}$ by regular decimation of the $n \times n$ identity matrix.  More precisely, $\mtx{V}$ is the matrix whose $j$th row has a unit entry in column $(j - 1)/k$ when $j \equiv 1 \pmod{k}$ and is zero otherwise.
To see why this type of matrix is nasty, it is helpful to consider the auxiliary matrix $\mtx{W} = \mtx{V}^\adj \mtx{DF}$.  Observe that, up to scaling and modulation of columns, $\mtx{W}$ consists of $k$ copies of a $k \times k$ DFT concatenated horizontally.

Suppose that we apply the SRFT $\mtx{\Omega} = \mtx{DFR}^\adj$ to the matrix $\mtx{V}^\adj$.  We obtain a matrix of the form $\mtx{X} = \mtx{V}^\adj \mtx{\Omega} = \mtx{WR}^\adj$, which consists of $\ell$ random columns sampled from $\mtx{W}$.
Theorem~\ref{thm:SRFT-spec-bd} certainly cannot hold unless $\sigma_k(\mtx{X}) > 0$.  To ensure the latter event occurs, we must pick at least one copy each of the $k$ distinct columns of $\mtx{W}$.  This is the coupon collector's problem~\cite[Sec.~3.6]{MR95:Randomized-Algorithms} in disguise.  To obtain a complete set of $k$ coupons (i.e., columns) with nonnegligible probability, we must draw at least $k \log(k)$ columns.  The fact that we are sampling without replacement does not improve
the analysis appreciably because the matrix has too many columns.
\end{remark}



\lsp

\subsection{Performance guarantees}

We are now prepared to present detailed information on the
performance of the proto-algorithm when the test matrix
$\mtx{\Omega}$ is an SRFT.

\lsp

\begin{theorem}[Error bounds for SRFT]
\label{thm:SRFT}
Fix an $m \times n$ matrix $\mtx{A}$ with singular values
$\sigma_1 \geq \sigma_2 \geq \sigma_3 \geq \dots$.
Draw an $n \times \ell$ SRFT matrix $\mtx{\Omega}$, where
$$
4 \left[\sqrt{k} + \sqrt{8\log(kn)} \right]^2 \log(k) \leq \ell \leq n.
$$
Construct the sample matrix $\mtx{Y} = \mtx{A\Omega}$. Then
\begin{align*}
\norm{ (\Id - \mtx{P}_{\mtx{Y}}) \mtx{A} }
    &\leq \sqrt{1 + 7n/\ell} \cdot \sigma_{k+1}  \quad\text{and} \\
 \fnorm{ (\Id - \mtx{P}_{\mtx{Y}}) \mtx{A} }
    &\leq \sqrt{1 + 7n/\ell} \cdot \left( \sum\nolimits_{j > k} \sigma_j^2 \right)^{1/2}
\end{align*}
with failure probability at most $\bigO(k^{-1})$.
\end{theorem}

\lsp


As we saw in~\S\ref{sec:gauss-avg-case},
the quantity $\sigma_{k+1}$ is the minimal spectral-norm error possible
when approximating $\mtx{A}$ with a rank-$k$ matrix.
Similarly, the series in the second bound
is the minimal Frobenius-norm error when approximating $\mtx{A}$
with a rank-$k$ matrix.  We see that both error bounds
lie within a polynomial factor of the baseline, and this factor decreases
with the number $\ell$ of samples we retain.

The likelihood of error with an SRFT test matrix is substantially worse than
in the Gaussian case.  The failure probability here is roughly $k^{-1}$,
while in the Gaussian case, the failure probability is roughly $\econst^{-(\ell - k)}$.
This qualitative difference is not an artifact of the analysis; discrete sampling techniques inherently fail with higher probability.

Matrix approximation schemes based on SRFTs often perform much better in practice
than the error analysis here would indicate. While it is not generally possible
to guarantee accuracy with a sampling parameter less than $\ell \sim k \log(k)$,
we have found empirically that the choice $\ell = k+20$ is adequate
in almost all applications. Indeed, SRFTs sometimes perform even \textit{better}
than Gaussian matrices (see, e.g., Figure \ref{fig:SRFT_errors}).






We complete the section with the proof of Theorem~\ref{thm:SRFT}.

\lsp

\begin{proof}[Theorem \ref{thm:SRFT}]
Let $\mtx{V}$ be the right unitary factor of matrix $\mtx{A}$, and
partition $\mtx{V} = [ \mtx{V}_1 \ | \ \mtx{V}_2 ]$ into blocks
containing, respectively, $k$ and $n - k$ columns.  Recall that
$$
\mtx{\Omega}_1 = \mtx{V}_1^\adj \mtx{\Omega} \quad\text{and}\quad
\mtx{\Omega}_2 = \mtx{V}_2^\adj \mtx{\Omega}.
$$
where $\mtx{\Omega}$ is the conjugate transpose of an SRFT.
Theorem~\ref{thm:SRFT-spec-bd} ensures that the submatrix $\mtx{\Omega}_1$
has full row rank, with failure probability at most $\bigO(k^{-1})$.
Therefore, Theorem~\ref{thm:main-error-bd} implies that
$$
\triplenorm{ (\Id - \mtx{P}_{\mtx{Y}}) \mtx{A} }
    \leq \triplenorm{\mtx{\Sigma}_2} \left[ 1 + \smnorm{}{ \mtx{\Omega}_1^\psinv }^2 \cdot
        \normsq{ \mtx{\Omega}_2 } \right]^{1/2},
$$
where $\triplenorm{\cdot}$ denotes either the spectral norm or the Frobenius norm.
Our application of Theorem~\ref{thm:SRFT-spec-bd} also ensures that the
spectral norm of $\mtx{\Omega}_1^\psinv$ is under control.
$$
\smnorm{}{ \mtx{\Omega}_1^\psinv }^2 \leq \frac{1}{0.40^2} < 7.
$$
We may bound the spectral norm of $\mtx{\Omega}_2$ deterministically.
$$
\norm{\mtx{\Omega}_2} = \norm{ \mtx{V}_2^\adj \mtx{\Omega} }
    \leq \norm{ \mtx{V}_2^\adj } \norm{ \mtx{\Omega} }
    = \sqrt{n/\ell}
$$
since $\mtx{V}_2$ and $\sqrt{\ell/n} \cdot \mtx{\Omega}$ are both orthonormal matrices.
Combine these estimates to complete the proof.
\end{proof}

\section*{Acknowledgments}
The authors have benefited from valuable discussions with many
researchers, among them Inderjit Dhillon, Petros Drineas, Ming Gu,
Edo Liberty, Michael Mahoney, Vladimir Rokhlin,
Yoel Shkolnisky, and Arthur Szlam.
In particular, we would like to thank Mark Tygert for his insightful
remarks on early drafts of this paper.
The example in Section \ref{sec:graph_laplacian} was provided by
Fran{\c{c}}ois Meyer of the University of Colorado at Boulder.
The example in Section \ref{sec:eigenfaces} comes from the FERET
database of facial images collected under the FERET program,
sponsored by the DoD Counterdrug Technology Development Program
Office.
The work reported was initiated during the program
\textit{Mathematics of Knowledge and Search Engines}
held at IPAM in the fall of 2007.
Finally, we would like to thank the anonymous referees, whose
thoughtful remarks have helped us to improve the manuscript
dramatically.


\lsp

\begin{appendix}

\section{On Gaussian matrices} \label{app:gauss}
This appendix collects some of the properties of Gaussian matrices that we use in our analysis.
Most of the results follow quickly from material that is already available in the literature.
One fact, however, requires a surprisingly difficult new argument.  We focus on the
real case here; the complex case is similar but actually yields better results.

\subsection{Expectation of norms}

We begin with the expected Frobenius norm of a scaled Gaussian matrix,
which follows from an easy calculation.

\lsp

\begin{proposition} \label{prop:gauss-frob-expect}
Fix real matrices $\mtx{S}, \mtx{T}$, and draw a standard Gaussian matrix $\mtx{G}$.  Then
$$
\left( \Expect \fnormsq{ \mtx{SGT} } \right)^{1/2}
    = \fnorm{\mtx{S}} \fnorm{\mtx{T}}.
$$
\end{proposition}


\begin{proof}
The distribution of a Gaussian matrix is invariant
under orthogonal transformations, and the Frobenius norm is also
unitarily invariant.  As a result, it represents no loss of
generality to assume that $\mtx{S}$ and $\mtx{T}$ are diagonal.
Therefore,
$$
\Expect \fnormsq{ \mtx{SGT} }
    = \Expect \left[ \sum\nolimits_{jk} \abssq{s_{jj} g_{jk} t_{kk}} \right]
    = \sum\nolimits_{jk} \abssq{s_{jj}} \abssq{t_{kk}}
    = \fnormsq{ \mtx{S} } \fnormsq{\mtx{T} }.
$$
Since the right-hand side is unitarily invariant, we have also
identified the value of the expectation for general matrices
$\mtx{S}$ and $\mtx{T}$.
\end{proof}

\lsp


The literature
contains an excellent bound for the expected spectral norm of a scaled Gaussian matrix.
The result is due to Gordon~\cite{Gor85:Some-Inequalities,Gor88:Gaussian-Processes},
who established the bound using a sharp version of Slepian's lemma.
See \cite[\S3.3]{LT91:Probability-Banach} and
\cite[\S2.3]{DS02:Local-Operator} for additional discussion.

\lsp

\begin{proposition} \label{prop:gauss-spec-expect}
Fix real matrices $\mtx{S}, \mtx{T}$, and draw a standard Gaussian matrix $\mtx{G}$.  Then
$$
\Expect \norm{ \mtx{SGT} }
    \leq \norm{\mtx{S}} \fnorm{\mtx{T}} + \fnorm{\mtx{S}} \norm{\mtx{T}}.
$$
\end{proposition}



\subsection{Spectral norm of pseudoinverse}

Now, we turn to the pseudoinverse of a Gaussian matrix.
Recently, Chen and Dongarra developed a good bound on the
probability that its spectral norm is large.  The statement
here follows from~\cite[Lem.~4.1]{CD05:Condition-Numbers}
after an application of Stirling's approximation. 
See also~\cite[Lem.~2.14]{random1}

\lsp

\begin{proposition} \label{prop:inv-gauss-spec-tail}
Let $\mtx{G}$ be an $m \times n$ standard Gaussian matrix with $n \geq m \geq 2$.
For each $t > 0$,
\begin{equation*} \label{eqn:gauss-lower-bd-2}
\Prob{ \smnorm{}{\mtx{G}^\psinv} > t }
    \leq \frac{1}{\sqrt{2\pi(n-m+1)}} \left[ \frac{\econst\sqrt{n}}{n-m+1}\right]^{n-m+1} t^{-(n-m+1)}.
\end{equation*}
\end{proposition}


We can use Proposition~\ref{prop:inv-gauss-spec-tail} to bound
the expected spectral norm of a pseudo-inverted
Gaussian matrix.

\lsp

\begin{proposition} \label{prop:inv-gauss-spec-expect}
Let $\mtx{G}$ be a $m \times n$ standard Gaussian matrix
with $n-m \geq 1$ and $m \geq 2$.  Then
$$
\Expect \smnorm{}{\mtx{G}^\psinv}
    < 
    \frac{\econst\sqrt{n}}{n-m}
$$
\end{proposition}



\begin{proof}
Let us make the abbreviations $p = n - m$ and
$$
C = \frac{1}{\sqrt{2\pi(p+1)}} \left[
\frac{\econst\sqrt{n}}{p+1}\right]^{p+1}.
$$
We compute the expectation by way of a standard argument.  The
integral formula for the mean of a nonnegative random variable
implies that, for all $E > 0$,
\begin{multline*}
\Expect \smnorm{}{\mtx{G}^\psinv}
    = \int_0^{\infty} \Prob{ \smnorm{}{\mtx{G}^\psinv} > t }\idiff{t}
    \leq E + \int_E^{\infty} \Prob{ \smnorm{}{\mtx{G}^\psinv} > t }\idiff{t} \\
    \leq E + C \int_E^\infty t^{-(p+1)} \idiff{t}
    = E + \frac{1}{p} C E^{-p},
\end{multline*}
where the second inequality follows from Proposition~\ref{prop:inv-gauss-spec-tail}.
The right-hand side is minimized when $E = C^{1/(p+1)}$.
Substitute and simplify.
\end{proof}


\subsection{Frobenius norm of pseudoinverse}

The squared Frobenius norm of a pseudo-inverted Gaussian matrix is
closely connected with the trace of an inverted Wishart matrix.
This observation leads to an exact expression for the expectation.

\lsp

\begin{proposition} \label{prop:inv-gauss-frob-expect}
Let $\mtx{G}$ be an $m \times n$ standard Gaussian matrix with $n - m \geq 2$.  Then
$$
\Expect \fnormsq{\mtx{G}^\psinv}
    = \frac{m}{n-m-1}.
$$
\end{proposition}

\begin{proof}
Observe that
$$
\smnorm{\rm F}{ \mtx{G}^\psinv }^2
    = \trace \left[ (\mtx{G}^\psinv)^\adj \mtx{G}^\psinv  \right]
    = \trace \left[ (\mtx{GG}^\adj)^{-1} \right].
$$
The second identity holds almost surely because the Wishart matrix $\mtx{GG}^\adj$ is invertible with probability one.
The random matrix $(\mtx{GG}^\adj)^{-1}$ follows the inverted
Wishart distribution, so we can compute its expected trace explicitly
using a formula from~\cite[p.~97]{Mui82:Aspects-Multivariate}.
\end{proof}

\lsp

On the other hand, very little seems to be known about the tail
behavior of the Frobenius norm of a pseudo-inverted Gaussian matrix.
The following theorem, which is new, provides an adequate bound on
the probability of a large deviation.

\lsp

\begin{theorem} \label{thm:inv-gauss-frob-tail}
Let $\mtx{G}$ be an $m \times n$ standard Gaussian matrix with $n - m \geq 4$.
For each $t \geq 1$,
$$
\Prob{ \fnormsq{\mtx{G}^\psinv} > \frac{12 m}{n-m} \cdot t }
    \leq 4 t^{-(n-m)/2}.
$$
\end{theorem}

Neither the precise form of Theorem~\ref{thm:inv-gauss-frob-tail}
nor the constants are ideal;
we have focused instead on establishing a useful bound with minimal fuss.
The rest of the section is devoted to the rather lengthy proof.
Unfortunately, most of the standard methods for producing tail bounds fail
for random variables that do not exhibit normal or exponential concentration.
Our argument relies on special properties of Gaussian matrices and a
dose of brute force.

\subsubsection{Technical background}

We begin with a piece of notation.  For any number $q \geq 1$, we define the $L_q$ norm of
a random variable $Z$ by
$$
\Expect^q(Z) = \left( \Expect \abs{Z}^q \right)^{1/q}.
$$
In particular, the $L_q$ norm satisfies the triangle inequality.

We continue with a collection of technical results.
First, we present a striking fact about
the structure of Gaussian matrices~\cite[\S3.5]{Ede89:Eigenvalues-Condition}.

\lsp

\begin{proposition} \label{prop:bidiag}
For $n \geq m$, an $m \times n$ standard Gaussian matrix is orthogonally equivalent with a random bidiagonal matrix
\begin{equation} \label{eqn:bidiag}
\mtx{L} =
\begin{bmatrix}
X_n \\
Y_{m-1} & X_{n-1} \\
& Y_{m-2} & X_{n-2} \\
&& \ddots & \ddots \\
&&& Y_1 & X_{n-(m-1)} &&&
\end{bmatrix}_{m \times n},
\end{equation}
where, for each $j$, the random variables $X_j^2$ and $Y_j^2$ follow the $\chi^2$
distribution with $j$ degrees of freedom.  Furthermore, these variates are mutually
independent.
\end{proposition}

\lsp

We also require the moments of a chi-square variate, which are expressed in terms of special functions.

\lsp

\begin{proposition} \label{prop:chisq-moment-exact}
Let $\Xi$ be a $\chi^2$ variate with $k$ degrees of freedom.  When $0 \leq q < k/2$,
$$
\Expect \left(\Xi^q \right) = \frac{ 2^q \Gamma(k/2+q)}{\Gamma(k/2)}
\quad\text{and}\quad
\Expect \left(\Xi^{-q} \right) = \frac{ \Gamma(k/2 - q)}{2^q \Gamma(k/2)}.
$$
\end{proposition}


\begin{proof}
Recall that a $\chi^2$ variate with $k$ degrees of freedom has the probability density function
$$
f(t) = \frac{1}{2^{k/2} \Gamma(k/2)} t^{k/2 - 1} \econst^{-t/2},
    \quad\text{for $t \geq 0$}.
$$
By the integral formula for expectation,
$$
\Expect(\Xi^q) = \int_0^{\infty} t^q f(t)\idiff{t}
    = \frac{ 2^q \Gamma(k/2+q)}{\Gamma(k/2)},
$$
where the second equality follows from Euler's integral expression for the gamma function.  The other calculation is similar.
\end{proof}

\lsp

To streamline the proof, we eliminate the gamma functions from Proposition~\ref{prop:chisq-moment-exact}.
The next result bounds the positive moments of a chi-square variate.

\lsp

\begin{lemma} \label{lem:chisq-moment}
Let $\Xi$ be a $\chi^2$ variate with $k$ degrees of freedom.  For $q \geq 1$,
$$
\Expect^q (\Xi) \leq k+q.
$$
\end{lemma}

\begin{proof}
Write $q = r + \theta$, where $r = \lfloor q \rfloor$.  Repeated application of the
functional equation $z\Gamma(z) = \Gamma(z+1)$ yields
$$
\Expect^q (\Xi)
    = \left[ \frac{2^\theta \Gamma(k/2 + \theta)}{\Gamma(k/2)}
        \cdot \prod_{j=1}^{r} (k + 2(q-j)) \right]^{1/q}.
$$
The gamma function is logarithmically convex, so
$$
\frac{2^\theta \Gamma(k/2 + \theta)}{\Gamma(k/2)}
    \leq \frac{2^\theta \cdot \Gamma(k/2)^{1-\theta} \cdot \Gamma(k/2 + 1)^\theta}{\Gamma(k/2)}
    = k^\theta
    \leq \left[ \prod_{j=1}^r (k + 2(q-j)) \right]^{\theta/r}.
$$
The second inequality holds because $k$ is smaller than each term in
the product, hence is smaller than their geometric mean.  As a
consequence,
$$
\Expect^q(\Xi)
    \leq \left[\prod_{j=1}^r (k + 2(q-j)) \right]^{1/r}
    \leq \frac{1}{r} \sum_{j=1}^r (k + 2(q-j))
    \leq k + q.
$$
The second relation is the inequality between the geometric and arithmetic mean.
\end{proof}

\lsp

Finally, we develop a bound for the negative moments of a chi-square variate.

\lsp

\begin{lemma} \label{lem:inv-chisq-moment}
Let $\Xi$ be a $\chi^2$ variate with $k$ degrees of freedom, where $k \geq 5$.
When $2 \leq q \leq (k-1)/2$,
$$
\Expect^q \left( \Xi^{-1} \right)
    < \frac{3}{k}.
$$
\end{lemma}

\begin{proof}
We establish the bound for $q = (k-1)/2$.  For smaller values of
$q$, the result follows from H{\"o}lder's inequality.
Proposition~\ref{prop:chisq-moment-exact} shows that
$$
\Expect^q \left(\Xi^{-1}\right)
    = \left[ \frac{\Gamma(k/2-q)}{2^q \Gamma(k/2)} \right]^{1/q}
    = \left[ \frac{\Gamma(1/2)}{2^q \Gamma(k/2)} \right]^{1/q}.
$$
Stirling's approximation ensures that
$\Gamma(k/2) \geq \sqrt{2\pi} \cdot (k/2)^{(k-1)/2} \cdot \econst^{-k/2}$.
Since the value $\Gamma(1/2) = \sqrt{\pi}$, 
$$
\Expect^q \left(\Xi^{-1}\right)
    \leq \left[ \frac{\sqrt{\pi}}{2^q\sqrt{2\pi}\cdot (k/2)^{q}\cdot \econst^{-q-1/2}} \right]^{1/q}
    = \frac{\econst}{k} \left[ \frac{\econst}{2} \right]^{1/2q}
    < \frac{3}{k},
$$
where we used the assumption $q \geq 2$ to complete the numerical estimate.
\end{proof}

\subsubsection{Proof of Theorem~\ref{thm:inv-gauss-frob-tail}}

Let $\mtx{G}$ be an $m \times n$ Gaussian matrix, where we assume that
$n - m \geq 4$.  Define the random variable
$$
Z = \fnormsq{\mtx{G}^\psinv}.
$$
Our goal is to develop a tail bound for $Z$.  The argument is inspired
by work of Szarek~\cite[\S6]{Sza90:Spaces-Large} for square Gaussian matrices.

The first step is to find an explicit, tractable representation for
the random variable. According to Proposition~\ref{prop:bidiag}, a
Gaussian matrix $\mtx{G}$ is orthogonally equivalent with a
bidiagonal matrix $\mtx{L}$ of the form~\eqref{eqn:bidiag}.
Making an analogy with the inversion formula for a 
triangular matrix, we realize that the pseudoinverse of $\mtx{L}$ is given by
$$
\mtx{L}^{\psinv} =
\begin{bmatrix}
X_n^{-1} \\
\frac{-Y_{m-1}}{X_n X_{n-1}} & X_{n-1}^{-1} \\
& \frac{-Y_{m-2}}{X_{n-1} X_{n-2}} & X_{n-2}^{-1} \\
&& \ddots & \ddots \\
&&& \frac{-Y_1}{X_{n-(m-2)} X_{n-(m-1)}} & X_{n-(m-1)}^{-1} \\
\\ \\
\end{bmatrix}_{n \times m}.
$$
Because $\mtx{L}^\psinv$ is orthogonally equivalent with $\mtx{G}^\psinv$
and the Frobenius norm is unitarily invariant, we have the relations
$$
Z = \smnorm{\rm F}{ \mtx{G}^\psinv }^2
    = \smnorm{\rm F}{ \mtx{L}^{\psinv} }^2
    \leq \sum_{j=0}^{m-1} \frac{1}{X_{n-j}^2}
        \left( 1 + \frac{Y_{m-j}^2}{X_{n-j+1}^2}\right),
$$
where we have added an extra subdiagonal term (corresponding with $j=0$) so that
we can avoid exceptional cases later.  We abbreviate the summands as
$$
W_j = \frac{1}{X_{n-j}^2} \left( 1 + \frac{Y_{m-j}^2}{X_{n-j+1}^2}\right),
\quad j = 0, 1, 2, \dots, m - 1.
$$

Next, we develop a large deviation bound for each summand by computing a moment
and invoking Markov's inequality.  For the exponent $q = (n-m)/2$,
Lemmas~\ref{lem:chisq-moment} and~\ref{lem:inv-chisq-moment} yield
\begin{align*}
\Expect^q (W_j)
    &= \Expect^q(X_{n-j}^{-2}) \cdot \Expect^q\left[1+ \frac{Y_{m-j}^2}{X_{n-j+1}^2}\right] \\
    &\leq \Expect^q(X_{n-j}^{-2}) \left[1
        + \Expect^q (Y_{m-j}^2) \cdot \Expect^q(X_{n-j+1}^{-2}) \right] \\
    &\leq \frac{3}{n-j} \left[ 1 + \frac{3(m-j+q)}{n-j+1} \right] \\
    &= \frac{3}{n-j} \left[1 + 3 - \frac{3(n-m+1-q)}{n-j+1}\right]
\end{align*}
Note that the first two relations require the independence of the variates
and the triangle inequality for the $L_q$ norm.
The maximum value of the bracket evidently occurs when $j = 0$, so
$$
\Expect^q(W_j) < \frac{12}{n - j},
\quad j = 0, 1, 2, \dots, m-1.
$$
Markov's inequality results in
$$
\Prob{ W_j \geq \frac{12}{n - j}  \cdot u } \leq u^{-q}.
$$
Select $u = t \cdot (n-j)/(n-m)$ to reach
$$
\Prob{ W_j \geq \frac{12}{n - m}  \cdot t }
    \leq \left[ \frac{n-m}{n-j} \right]^{q} t^{-q}.
$$

To complete the argument, we combine these estimates by means of the union bound and clean
up the resulting mess.  Since $Z \leq \sum_{j=0}^{m-1} W_j$,
$$
\Prob{ Z \geq \frac{12m}{n-m} \cdot t }
    \leq t^{-q} \sum_{j=0}^{m-1} \left[ \frac{n-m}{n-j} \right]^q.
$$
To control the sum on the right-hand side, observe that
\begin{multline*}
\sum_{j=0}^{m-1} \left[ \frac{n-m}{n-j} \right]^q
    < (n-m)^q \int_0^m (n-x)^{-q} \idiff{x} \\
    < \frac{(n-m)^q}{q-1} (n-m)^{-q + 1}
    = \frac{2(n-m)}{n-m-2}
    \leq 4,
\end{multline*}
where the last inequality follows from the hypothesis $n - m \geq 4$.
Together, the estimates in this paragraph produce the advertised bound.

\lsp

\begin{remark} \rm
It would be very interesting to find more conceptual and extensible argument
that yields accurate concentration results for inverse spectral functions of
a Gaussian matrix.
\end{remark}

%
%

\end{appendix}


\bibliography{matrix-approx-revised-bib}
\bibliographystyle{siam}

\end{document}